\documentclass{amsart}%
\usepackage{amsfonts, amsmath, amssymb}%
\usepackage{graphicx, epic, enumerate}
\usepackage{float}

\newtheorem{theorem}{Theorem}
\theoremstyle{definition}

\newtheorem{definition}{Definition}

\newtheorem{lemma}{Lemma}

\numberwithin{equation}{section}

\begin{document}
\allowdisplaybreaks
\title{Virtual singular braids and links}

\author{Carmen Caprau}
\address{Department of Mathematics, California State University, Fresno, CA 93740, USA}
\email{ccaprau@csufresno.edu}
\urladdr{}
\author{Andrew de la Pena}
\address{Department of Mathematics, North Dakota State University, Fargo, ND 58108, USA}
\email{andrew.delapena@ndsu.edu}
\author{Sarah McGahan}
\address{Department of Mathematics, California State University, Fresno, CA 93740, USA}
\email{srmcg@mail.fresnostate.edu}

\date{}

\begin{abstract} 
Virtual singular braids are generalizations of singular braids and virtual braids. We define the virtual singular braid monoid via generators and relations, and prove Alexander- and Markov-type theorems for virtual singular links. We also show that the virtual singular braid monoid has another presentation with fewer generators.
\end{abstract}

\maketitle
\section{Introduction}\label{sec:intro}

 J.W. Alexander~\cite{A} showed that any oriented classical link can be represented as the closure of a braid. Moreover, it is well-known that two braids have isotopic closures if and only if they are related by braid isotopy and a finite sequence of the so-called Markov's moves (see~\cite{M, W}). The first complete proof of this result was given by J. Birman~\cite{B1}. Other proofs have been provided by D. Bennequin~\cite{Be}, H. Morton~\cite{Mo}, P. Traczyk~\cite{T}, and S. Lambropoulou~\cite{L}. 

Analogous theorems for the virtual braid group have been proven by L.H. Kauffman and S. Lambropoulou~\cite{KL2} using the, so-called, $L$-equivalence and by S. Kamada~\cite{Ka} using Gauss data. Moreover, J. Birman~\cite{B2} proved an Alexander-type theorem for the singular braid monoid and singular links and B. Gemein~\cite{G} provided a Markov-type theorem for singular braids. Further, S. Lambropoulou~\cite{L2} derived the $L$-move analogue for singular braids via $L$-move methods, recovering the result of Gemein.

In this paper we consider oriented virtual singular links and prove Alexander-and Markov-type theorems for this class of links. These theorems are crucial in understanding the structure of virtual singular knots and links. We first define the virtual singular braid monoid using generators and relations. This definition reveals that the virtual singular braid monoid on $n$ strands is an extension of the singular braid monoid on $n$ strands by the symmetric group on $n$ letters. Various braiding algorithms can be used to prove that the Alexander theorem extends to the class of virtual singular braids. For our purpose, we borrow the braiding algorithm described in~\cite{KL2} and extend it to include singular crossings. We then show that the $L$-moves used in~\cite{KL2} for the class of virtual braids and links can be extended to the class of virtual singular braids and links. In the presence of singular crossings and additional relations describing the virtual singular braid monoid, we need to introduce a new type of $L$-moves, namely a new type of  `threaded $L_v$-moves' involving classical, singular, and virtual crossings. We state and prove first an $L$-move Markov-type theorem for virtual singular braids and then use it to provide an algebraic Markov-type theorem for virtual singular braids. 

During our study of this problem, we found that we were able to modify the arguments of~\cite{KL2}
and take the same diagrammatic geometry so as to prove our main results. Consequently, several figures in this paper are similar or exactly the same as certain figures in~\cite{KL2}. For example, if the reader would examine in this paper Figures~\ref{fig:b-c} through~\ref{fig:threaded-move} and compare with Figures 7, 9, 11, 12, and 13 in~\cite{KL2}, they would see the precise analogy of our arguments and the arguments of~\cite{KL2}.

Motivated by L.H. Kauffman and S. Lambropoulou's work in~\cite[Section 3] {KL1}, we also prove that the virtual singular braid monoid on $n$ strands admits a reduced presentation using fewer generators, namely three braiding elements together with the generators of the symmetric group on $n$ letters.

%%%%%%%%%%%%%%%%%%%%%%%%%%%%%%%%%%%%%%%%%%%%%%%%%%%
\section{Virtual singular links}\label{sec:vsl}

A virtual singular link diagram is a decorated immersion of (finitely many) disjoint copies of $S^1$ into $\mathbb{R}^2$, with finitely many transverse double points each of which has information of over/under, singular, and virtual crossings as in Figure~\ref{fig:crossings}. The over/under markings are the classical crossings, which we will refer to as \textit{real crossings}. \textit{Virtual crossings} are represented by placing a small circle around the point where the two arcs meet transversely. A filled in circle is used to represent a \textit{singular crossing}. We assume that virtual singular link diagrams are the same if they are isotopic in $\mathbb{R}^2$. 

 \begin{figure}[ht]
\[ \includegraphics[height=0.4in]{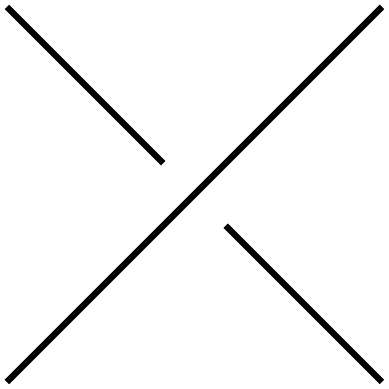}  \hspace{.5in} \includegraphics[height=0.4in, angle=90]{real} 
 \hspace{.5in} \includegraphics[height=0.4in]{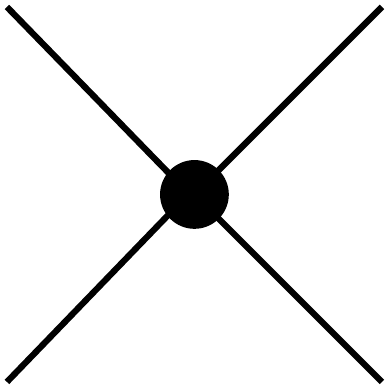}\hspace{.5in} \includegraphics[height=0.4in]{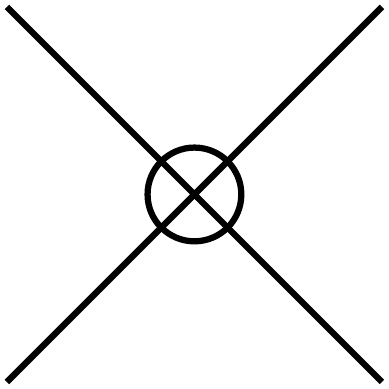}\]
\caption{Types of crossings in a virtual singular link diagram} \label{fig:crossings}
\end{figure}

Note that the set of classical link diagrams, or singular link diagrams, or virtual link diagrams comprise subsets of the set of virtual singular link diagrams. 

\begin{definition}
Two virtual singular link diagrams are said to be \textit{equivalent} if they are related by a finite sequence of the \textit{extended virtual Reidemeister moves} depicted in Figure~\ref{fig:isotopies} (where only one possible choice of crossings is indicated in the diagrams). A \textit{virtual singular link} (or a \textit{virtual singular link type}) is the equivalence class of a virtual singular link diagram. 
\end{definition}

\begin{figure}[ht]

\[\raisebox{-13pt}{\includegraphics[height=0.4in]{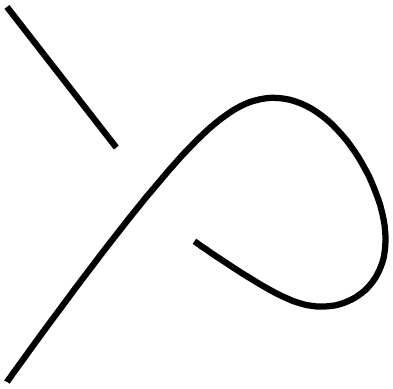}}\,\,  \stackrel{R1}{\longleftrightarrow} \,\, \raisebox{-13pt}{\includegraphics[height=0.4in]{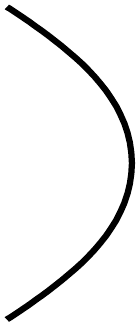}}  \hspace{1.5cm}
 \raisebox{-13pt}{\includegraphics[height=0.4in]{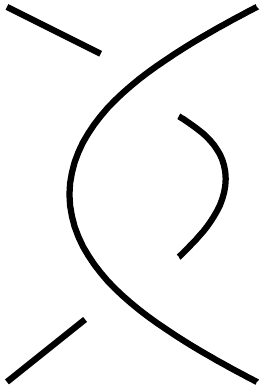}} \,\, \stackrel{R2}{\longleftrightarrow} \,\, \raisebox{-13pt}{\includegraphics[height=0.4in]{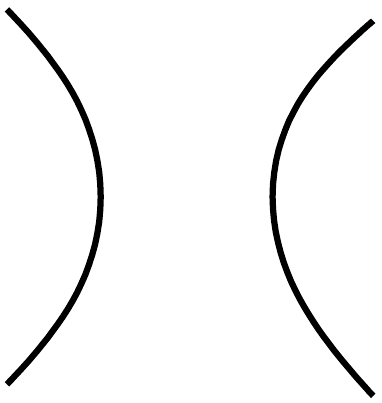}} \hspace{1.5cm}
\raisebox{-13pt}{\includegraphics[height=0.4in]{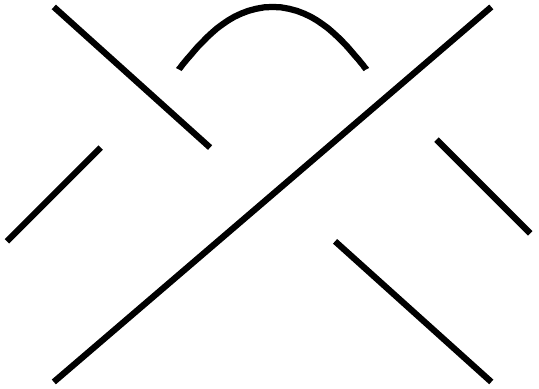}} \,\, \stackrel{R3}{\longleftrightarrow} \,\,  \raisebox{15pt}{\includegraphics[height=0.4in, angle = 180]{reid3}}\]

\[\raisebox{-13pt}{\includegraphics[height=0.4in]{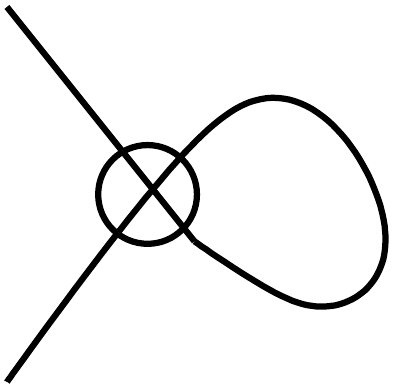}}\,\,  \stackrel{V1}{\longleftrightarrow} \,\, \raisebox{-13pt}{\includegraphics[height=0.4in]{arc}} \hspace{1.5cm}
 \raisebox{-13pt}{\includegraphics[height=0.4in]{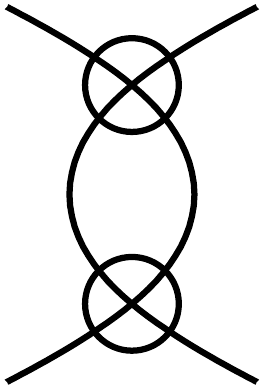}} \,\, \stackrel{V2}{\longleftrightarrow} \,\, \raisebox{-13pt}{\includegraphics[height=0.4in]{A-smoothing}} \hspace{1.5cm}
\raisebox{-13pt}{\includegraphics[height=0.4in]{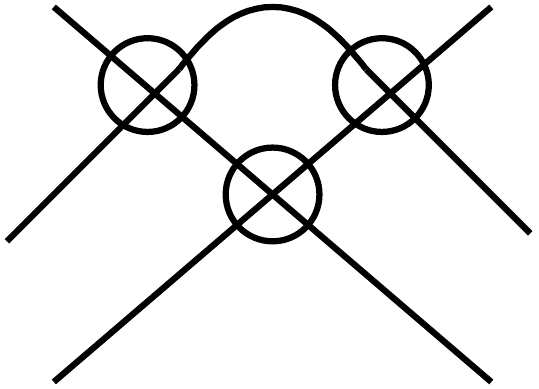}} \,\, \stackrel{V3}{\longleftrightarrow} \,\,   \raisebox{15pt}{\includegraphics[height=0.4in, angle=180]{reid3-virt}}\]

\[ \raisebox{-15pt}{\includegraphics[height=0.4in]{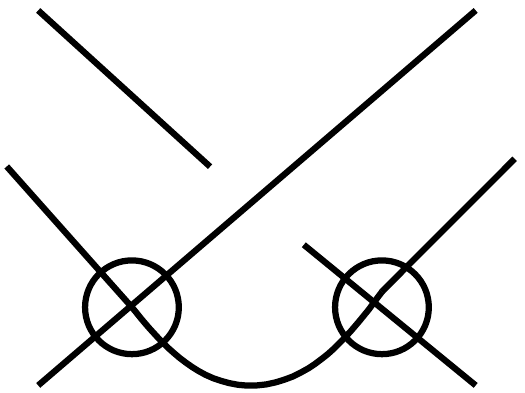}} \hspace{0.2cm}\stackrel{VR3}{\longleftrightarrow} \hspace{0.2cm} \raisebox{15pt}{\includegraphics[height=0.4in, angle=180]{VR3}} \hspace{1cm}
\raisebox{-15pt}{\includegraphics[height=0.4in]{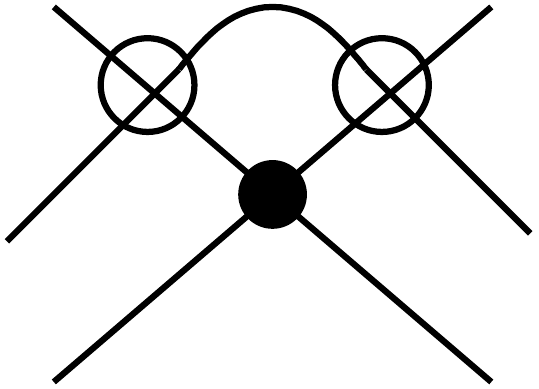}} \hspace{0.2cm}\stackrel{VS3}{\longleftrightarrow} \hspace{0.2cm} \raisebox{15pt}{\includegraphics[height=0.4in, angle=180]{VS3}}\]

\[\raisebox{-13pt}{\includegraphics[height=0.4in]{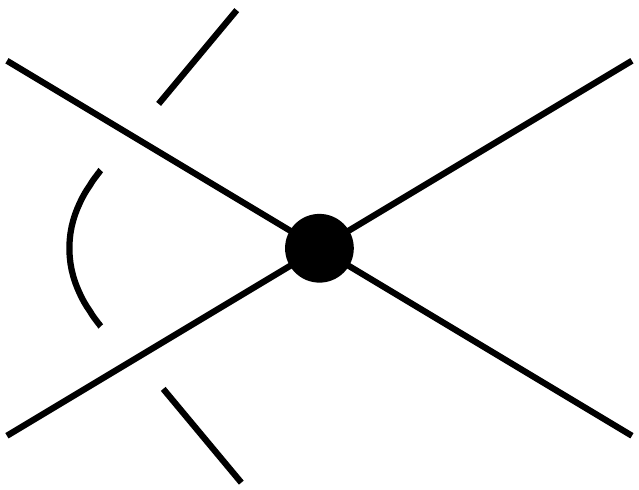}} \hspace{0.2cm} \stackrel{RS3}{\longleftrightarrow} \hspace{0.2cm} \reflectbox{\raisebox{-13pt}{\includegraphics[height=0.4in]{RS3}}}
\hspace{1cm}
\raisebox{-11pt}{\includegraphics[height=0.35in]{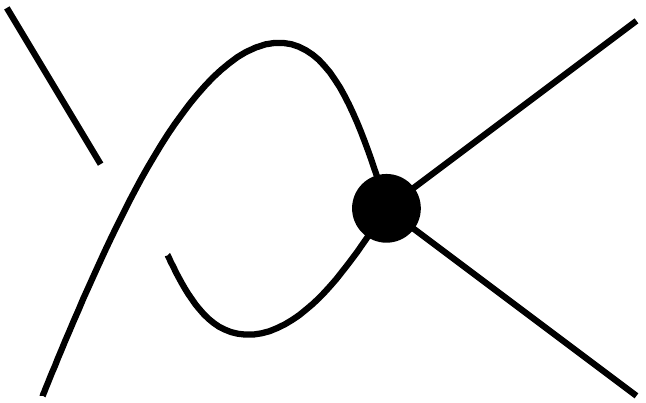}} \hspace{0.2cm} \stackrel{RS1}{\longleftrightarrow}\hspace{0.2cm} \raisebox{15pt}{\includegraphics[height=0.35in,angle=180]{RS1}}\]
\caption{The extended virtual Reidemeister moves}\label{fig:isotopies}
\end{figure}

Note that the moves involving virtual crossings can be considered as special cases of the \textit{detour move} depicted in Figure~\ref{fig:detour-move} (\cite{K1, KL1, KL2}). This move is the representation of the principle that the virtual crossings are not really there but that are rather byproducts of the projection. To understand the detour move, suppose an arc is free of real (classical) and singular crossings, and which may contain a consecutive sequence of virtual crossings. Then that arc can be arbitrarily moved, keeping its endpoints fixed, to any new location and placed transversally to the rest of the diagram, adding virtual crossings whenever these intersections occur. (In Figure~\ref{fig:detour-move}, the grey box represents an arbitrary virtual singular tangle diagram; a braid representation of the detour move is given in Figure~\ref{fig:braid detour-moves}.)

\begin{figure}[ht]
\[ \raisebox{-25pt}{\includegraphics[height=0.8in]{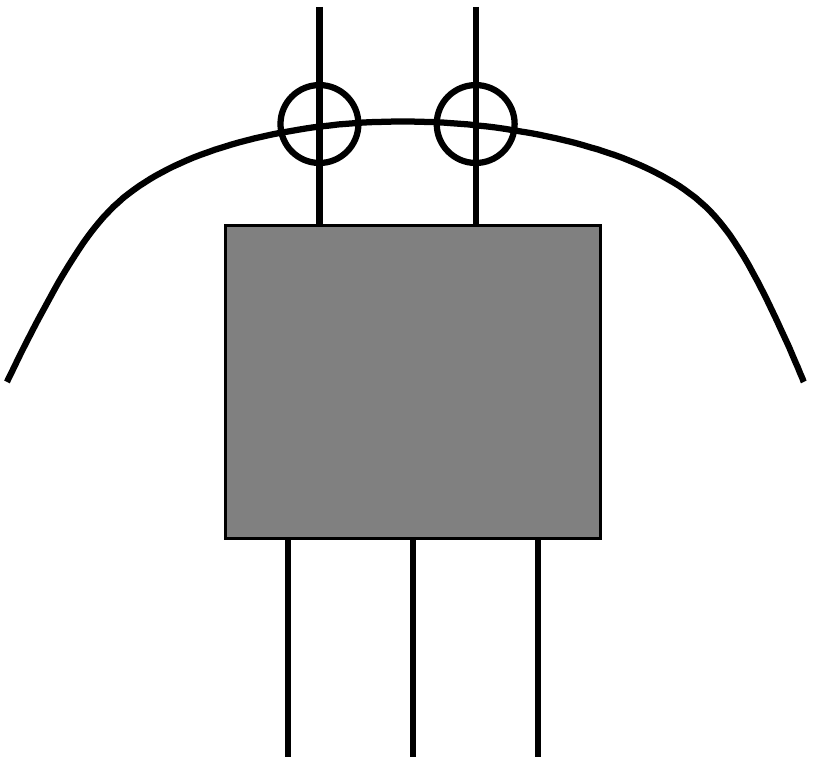}} \hspace{0.3cm}\longleftrightarrow \hspace{0.3cm} \raisebox{-25pt}{\includegraphics[height=0.8in]{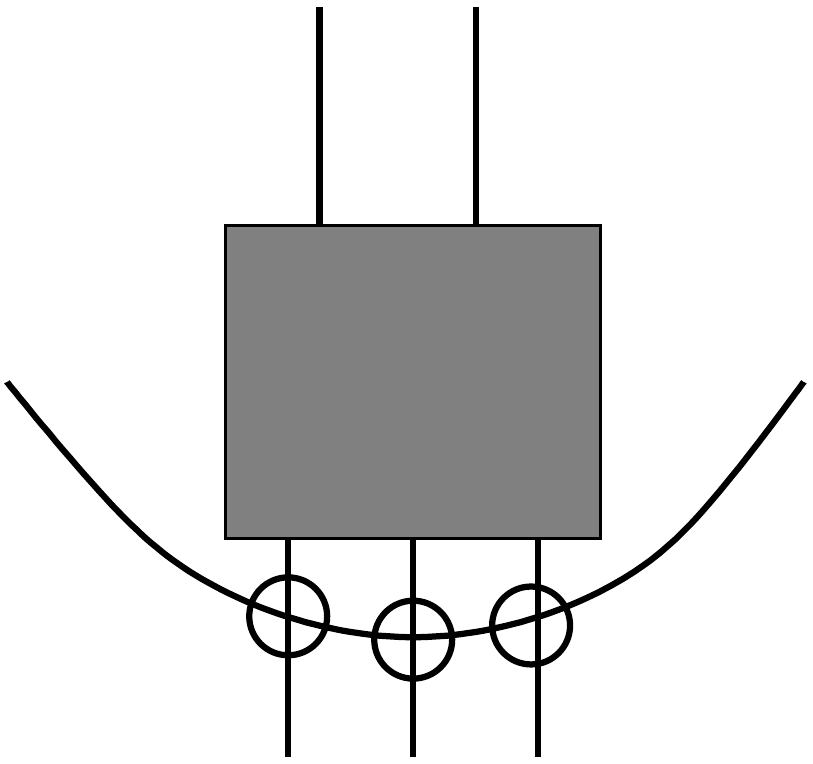}}
\]
\caption{The detour move}\label{fig:detour-move}
\end{figure} 

Conversely, the detour move can be obtained by a finite sequence of the moves shown in Figure~\ref{fig:isotopies} that involve virtual crossings. Consequently, the virtual singular equivalence is generated by the Reidemeister-type moves for singular link diagrams (that is, the classical Reidemeister moves together with the moves $RS1$ and $RS3$) and the detour move.

When working with equivalent virtual singular link diagrams, it is important to avoid the moves depicted in Figure ~\ref{fig:forbidden moves}. Although these moves are similar to some of the extended virtual Reidemeister moves, the diagrams of the two sides of a forbidden move do not represent equivalent virtual singular links. For this reason, we refer to these as the \textit{forbidden moves} for virtual singular link diagrams.

\begin{figure}[ht]
\[ \raisebox{-13pt}{\includegraphics[height=.4in]{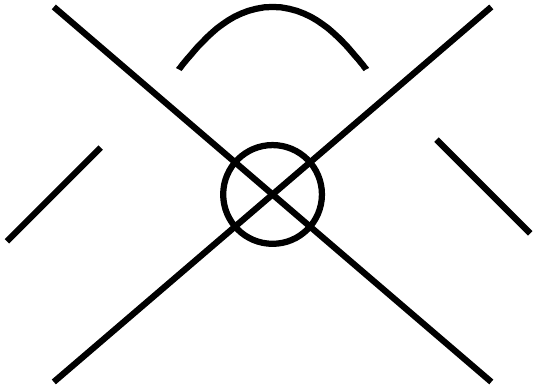}}\,\,\nleftrightarrow \,\,\raisebox{15pt}{\includegraphics[height=0.4in, angle = 180]{fmove1}}\hspace{1.5cm} 
\raisebox{-13pt}{\includegraphics[height=.4in]{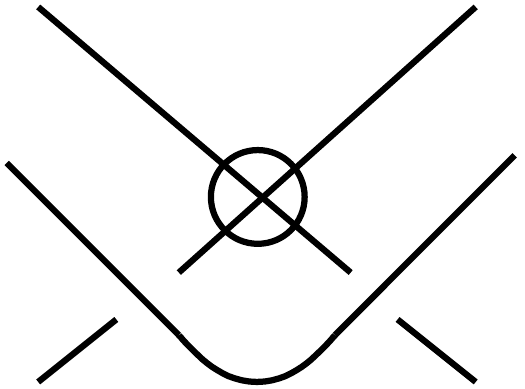}} \,\,\nleftrightarrow \,\,\raisebox{15pt}{\includegraphics[height=0.4in, angle=180]{fmove2}} \hspace{1cm}
\]

\[  \raisebox{-13pt}{\includegraphics[height=.4in]{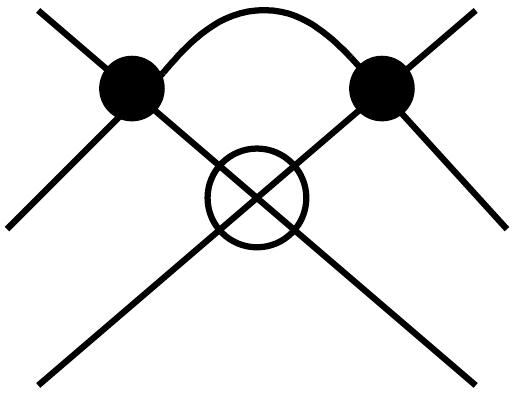}}\,\,\nleftrightarrow \,\, \raisebox{15pt}{\includegraphics[height=0.4in, angle = 180]{fmove3}} \hspace{1cm}
\reflectbox{\raisebox{-13pt}{\includegraphics[height=0.4in]{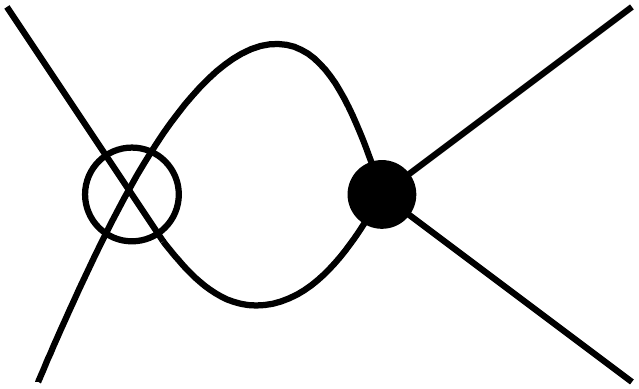}}}\,\, \nleftrightarrow \,\,\raisebox{-13pt}{\includegraphics[height=.35in]{sing}}\,\,\nleftrightarrow \,\,\raisebox{-13pt}{\includegraphics[height=0.4in]{forbid}}  \]

\caption{The forbidden moves for virtual singular link diagrams}\label{fig:forbidden moves}
\end{figure}

Recall that a \textit{singular link} is an immersion of a disjoint union of circles in three-dimensional space, which has finitely many singularities (namely  singular crossings) that are all transverse double points. Equivalently, a singular link is an embedding in three-dimensional space of a 4-valent graph with rigid vertices (where these vertices are the singular crossings). These type of embedding are also called rigid vertex knotted graphs.

Similar to the case of virtual knot theory, there is a useful topological interpretation for virtual singular knot theory in terms of embeddings of singular links (or equivalently, of rigid vertex knotted graphs) in thickened surfaces. For this, interpret each virtual crossing as a detour of one of the arcs in the crossings through a 1-handle that has been attached to the 2-sphere of the original diagram. We obtain an embedding of a collection of immersed circles into a thickened surface $S_g \times I$, where $I$ is the unit interval, $S_g$ is a compact oriented surface of genus $g$, and $g$ is the number of virtual crossings in the original diagram. Then singular knot theory in $S_g \times I$ is represented by diagrams drawn on $S_g$ taken up to the Reidemeister-type moves for singular link diagrams transferred to diagrams on $S_g$. Recall that the Reidemeister-type moves for singular link diagrams contain the classical Reidemeister moves $R1, R2$ and $R3$ together with the moves $RS1$ and $RS3$ shown in Figure~\ref{fig:isotopies}.

%%%%%%%%%%%%%%%%%%%%%%%%%%%%%%%%%%%%%%%%%%%%%%%%%%%
\section{Alexander- and Markov-type theorems}\label{sec:theorems}

A \textit{virtual singular braid} on $n$ strands is a braid in the classical sense, which may contain real, singular, and virtual crossings as `interactions' among the $n$ strands of the braid. By connecting the top endpoints with the corresponding bottom endpoints of a virtual singular braid using parallel arcs without introducing new crossings we obtain a virtual singular link diagram, called the \textit{closure} of the braid. 

Similar to the case of classical braids, virtual singular braids are composed using vertical concatenation. For two $n$-stranded virtual singular braids $\beta$ and $\beta'$, the braid $\beta \beta'$ is obtained by placing $\beta$ on top of $\beta'$ and connecting their endpoints. The set of isotopy classes of virtual singular braids on $n$ strands forms a monoid, which we denote by ${VSB}_n$. The monoid operation is the composition of braids, and the identity element, denoted by $1_n$, is the braid with $n$ vertical strands.
  
  \subsection{The virtual singular braid monoid}\label{sec:monoid}
The \textit{virtual singular braid monoid} on $n$ strands, ${VSB}_n$, is the monoid generated by the virtual singular braids $\sigma_i,\sigma_i^{-1},v_i$ and $\tau_i$, for $1 \leq i \leq n-1$, depicted below:

\[  \sigma_i \,\,\, =\,\,\,  \raisebox{-17pt}{\includegraphics[height=.5in]{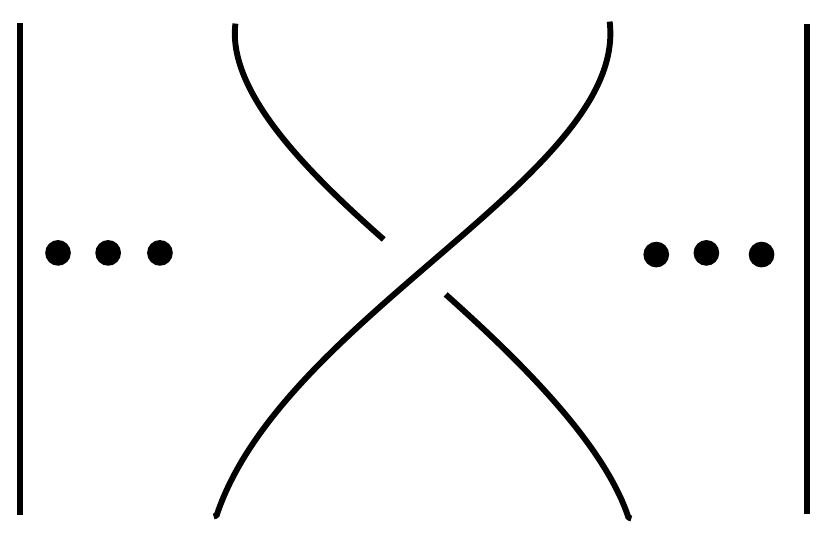}} \hspace{1cm} \sigma_i^{-1} \,\,\,=\,\,\, \raisebox{-17pt}{\includegraphics[height=.5in]{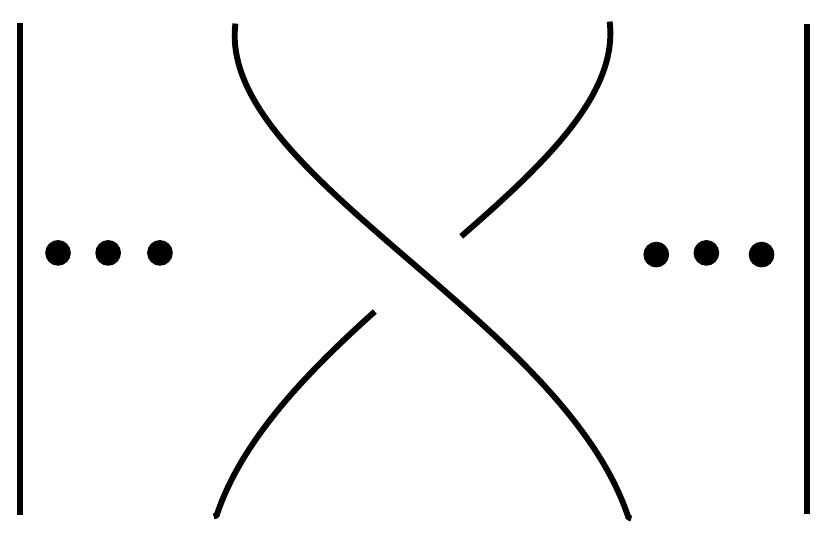}}
\put(-180, 21){\fontsize{7}{7}$1$}
\put(-165, 21){\fontsize{7}{7}$i$}
\put(-150, 21){\fontsize{7}{7}$i+1$}
\put(-127,21){\fontsize{7}{7}$n$}
\put(-58, 21){\fontsize{7}{7}$1$}
\put(-40, 21){\fontsize{7}{7}$i$}
\put(-25, 21){\fontsize{7}{7}$i+1$}
\put(-3,21){\fontsize{7}{7}$n$}
\]
\vspace{0.2cm}
\[ \tau_i \,\,\,= \,\,\, \stackrel \,\, \raisebox{-17pt}{\includegraphics[height=.5in]{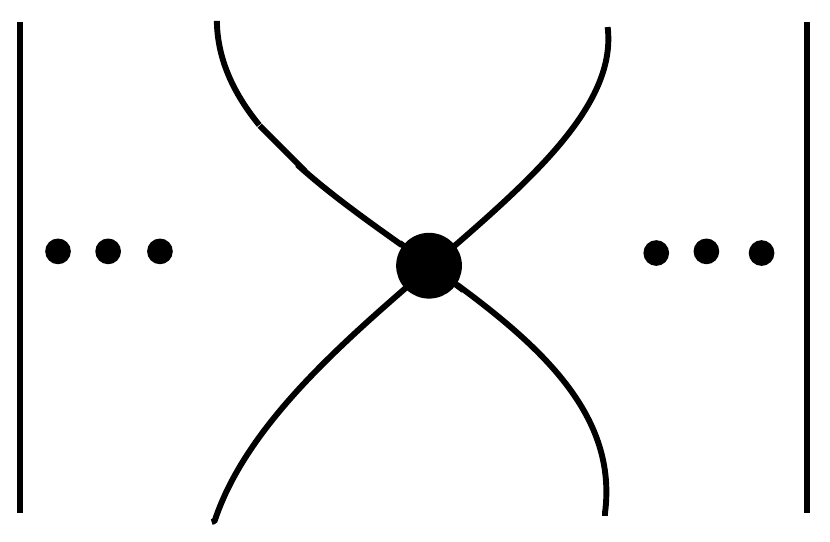}} \hspace{1cm} v_i \,\,\, = \,\,\, \raisebox{-17pt}{\includegraphics[height=.5in]{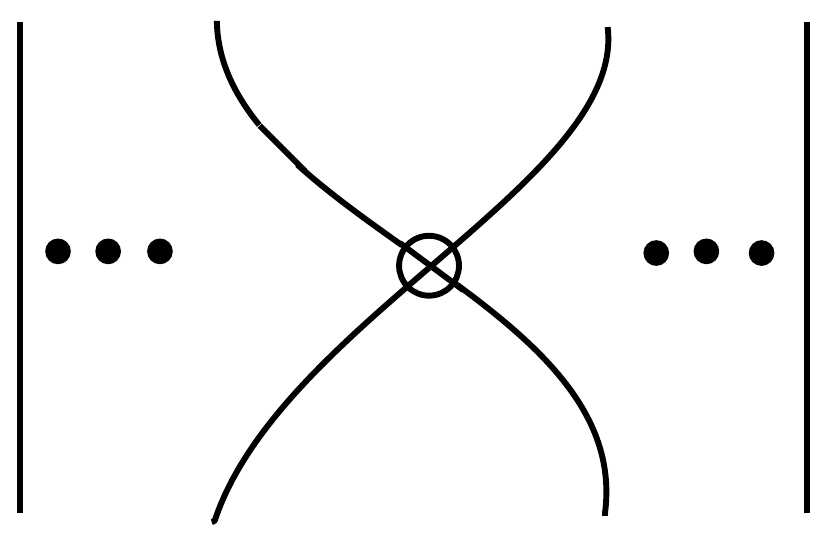}}
\put(-171, 21){\fontsize{7}{7}$1$}
\put(-158, 21){\fontsize{7}{7}$i$}
\put(-140, 21){\fontsize{7}{7}$i+1$}
\put(-119,21){\fontsize{7}{7}$n$}
\put(-56, 21){\fontsize{7}{7}$1$}
\put(-42, 21){\fontsize{7}{7}$i$}
\put(-25, 21){\fontsize{7}{7}$i+1$}
\put(-3,21){\fontsize{7}{7}$n$}
\]
and subject to the following relations:

\begin{itemize}

\item $\sigma_i\sigma_i^{-1}=\sigma_i^{-1}\sigma_i=1_n$
\begin{eqnarray*}
\raisebox{-.5cm}{\includegraphics[height=0.5in]{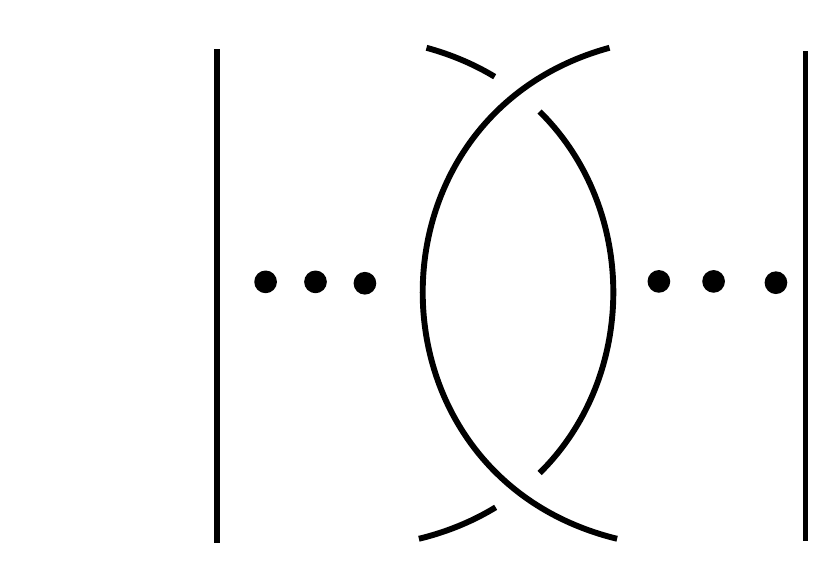}} \hspace{0.5cm} \stackrel{R2}{=} \hspace{0.2cm} &\raisebox{-.5cm}{\includegraphics[height=0.5in]{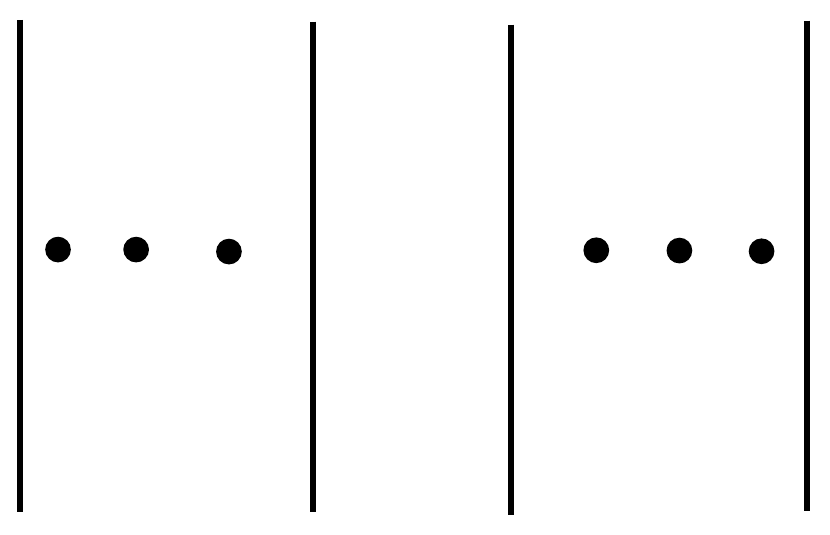}}
\end{eqnarray*}

\item $v_i^2=1_n$
\begin{eqnarray*}
\raisebox{-.5cm}{\includegraphics[height=0.5in]{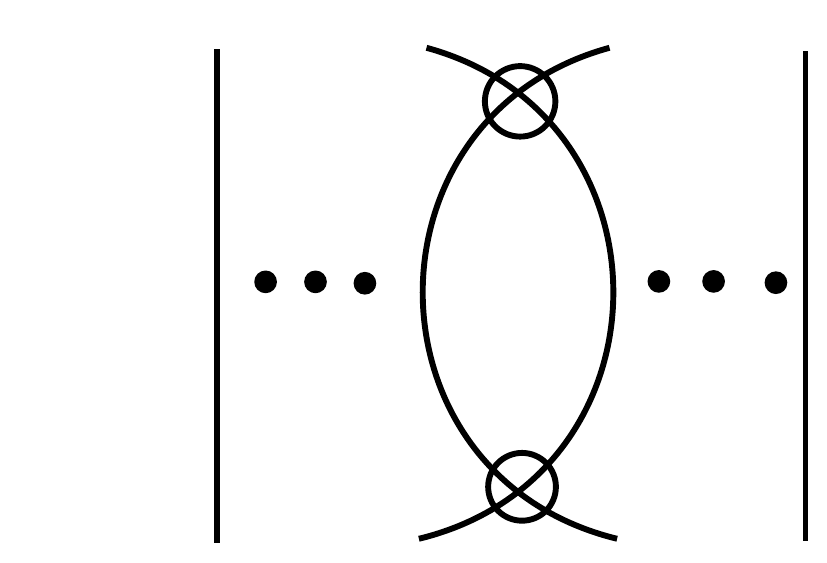}} \hspace{0.5cm} \stackrel{V2}{=} \hspace{0.2cm} &\raisebox{-.5cm}{\includegraphics[height=0.5in]{idb}}
\end{eqnarray*}

\item $\sigma_i\sigma_j\sigma_i=\sigma_j\sigma_i\sigma_j$, for $|i-j|=1$
\begin{eqnarray*}
\raisebox{-.5cm}{\includegraphics[height=0.5in]{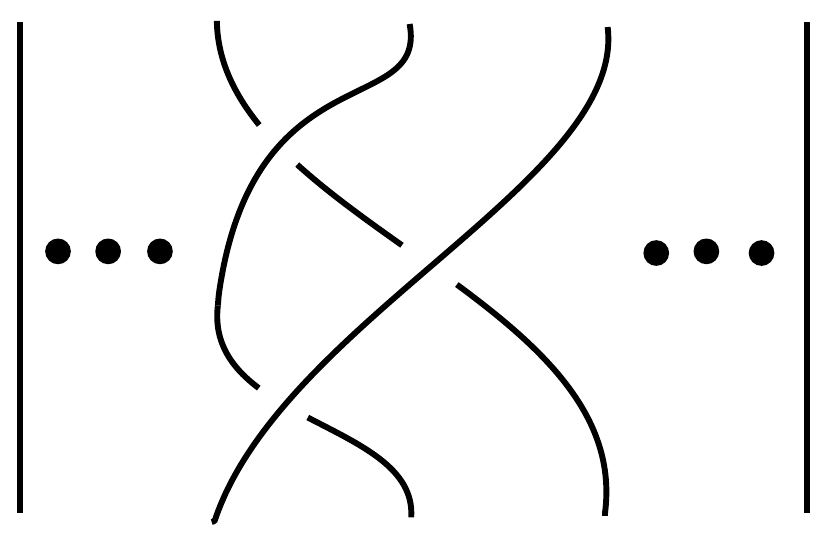}} \hspace{0.5cm} \stackrel{R3}{=} \hspace{0.2cm} &\raisebox{-.5cm}{\includegraphics[height=0.5in]{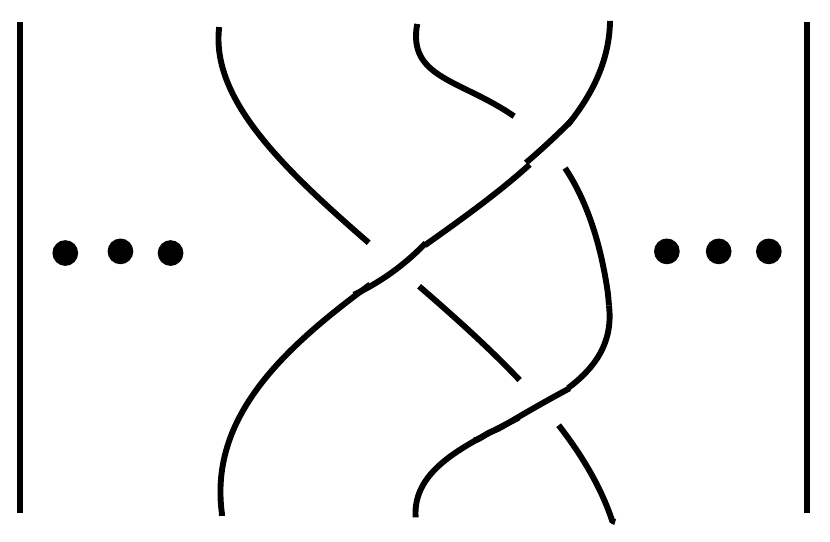}}
\end{eqnarray*}

\item $v_{i}v_{j}v_{i}=v_{j}v_{i}v_{j}$, for $|i-j|=1$
\begin{eqnarray*}
\raisebox{-.5cm}{\includegraphics[height=0.5in]{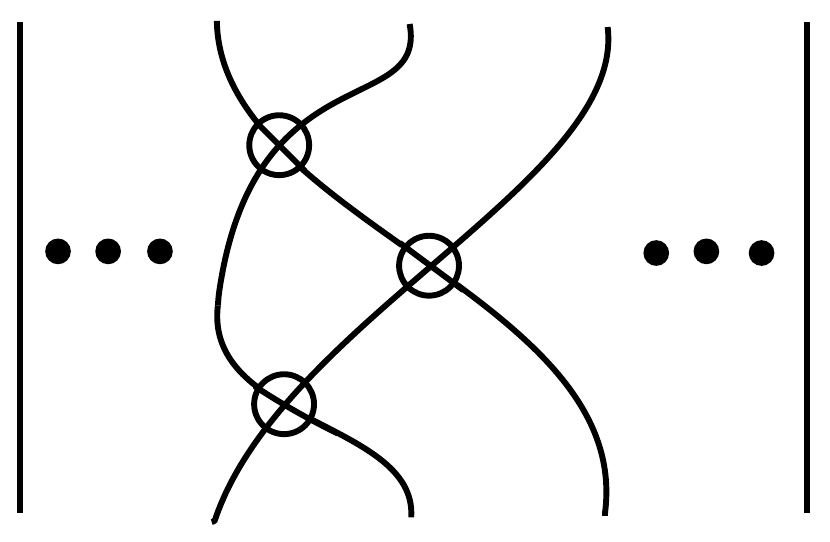}} \hspace{0.5cm} \stackrel{V3}{=} \hspace{0.2cm}&\raisebox{-.5cm}{\includegraphics[height=0.5in]{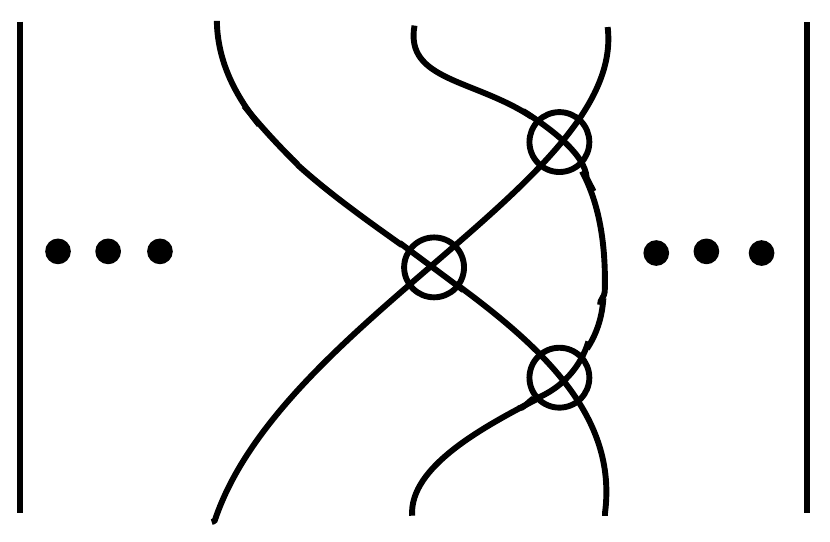}}
\end{eqnarray*}

\item $v_i\sigma_{j}v_i=v_{j}\sigma_iv_{j}$, for $|i-j|=1$ 
\begin{eqnarray*}
\raisebox{-.5cm}{\includegraphics[height=0.5in]{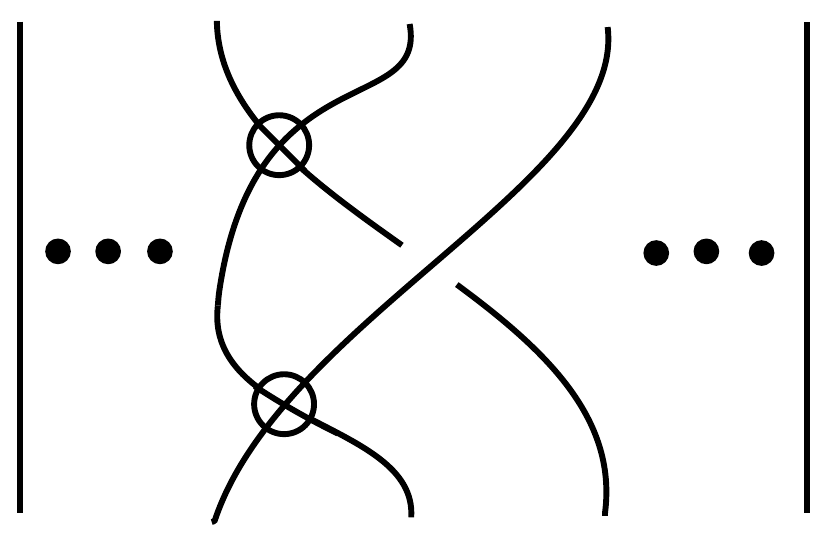}} \hspace{0.5cm} \stackrel{VR3}{=}\hspace{0.2cm} &\raisebox{-.5cm}{\includegraphics[height=0.5in]{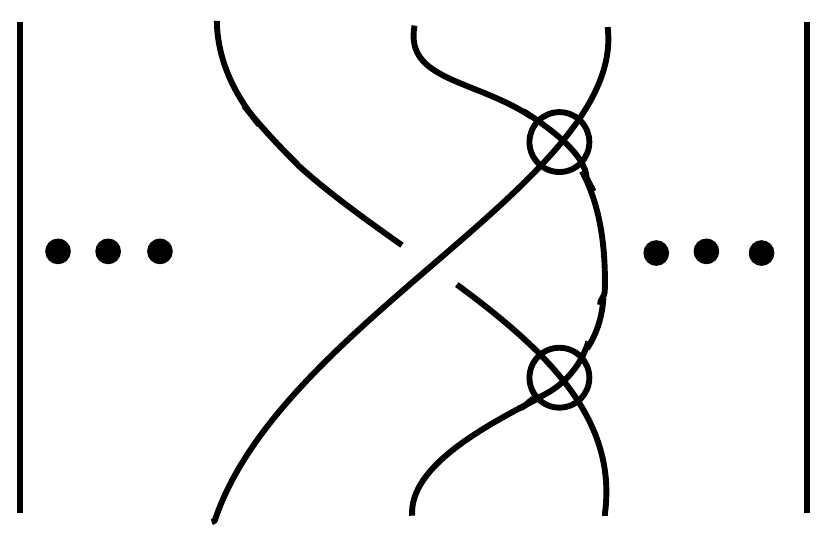}}
\end{eqnarray*}

\item $v_i\tau_{j}v_i=v_{j}\tau_iv_{j}$, for $|i-j|=1$
\begin{eqnarray*}
\raisebox{-.5cm}{\includegraphics[height=0.5in]{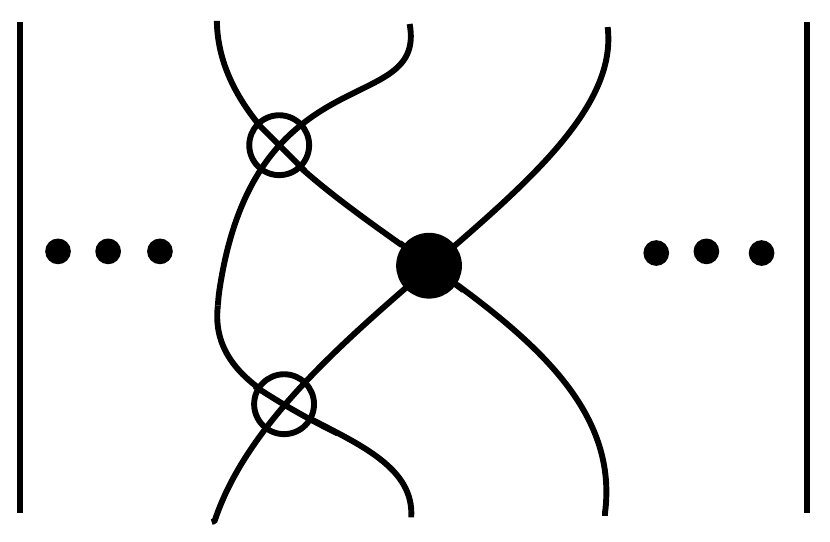}} \hspace{0.5cm} \stackrel{VS3}{=} \hspace{0.2cm} &\raisebox{-.5cm}{\includegraphics[height=0.5in]{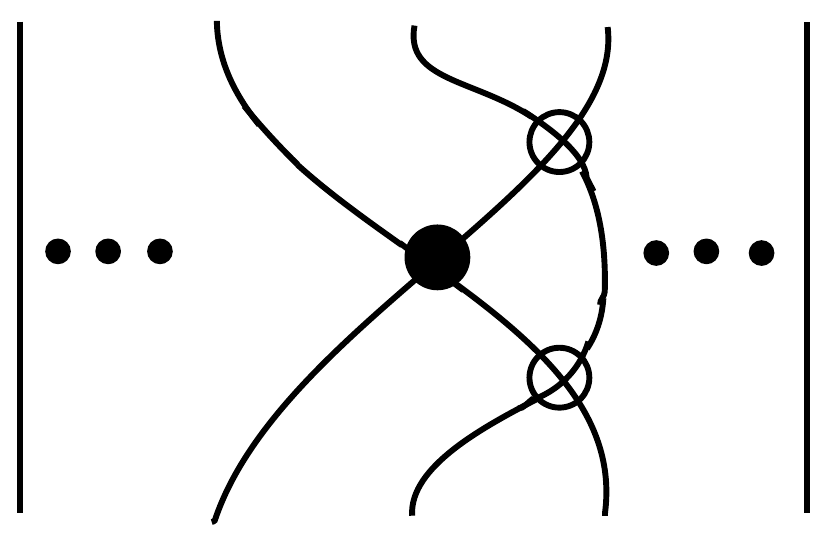}}
\end{eqnarray*}

\item $\sigma_i\sigma_j\tau_i=\tau_j\sigma_i\sigma_j$ for $|i-j|=1$
\begin{eqnarray*}
\raisebox{-.5cm}{\includegraphics[height=0.5in]{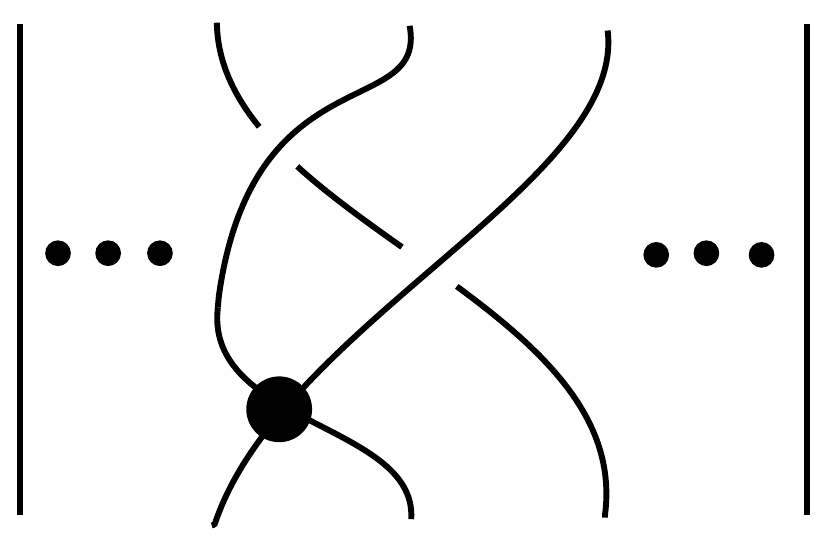}} \hspace{0.5cm} \stackrel{RS3}{=} \hspace{0.2cm} &\raisebox{-.5cm}{\includegraphics[height=0.5in]{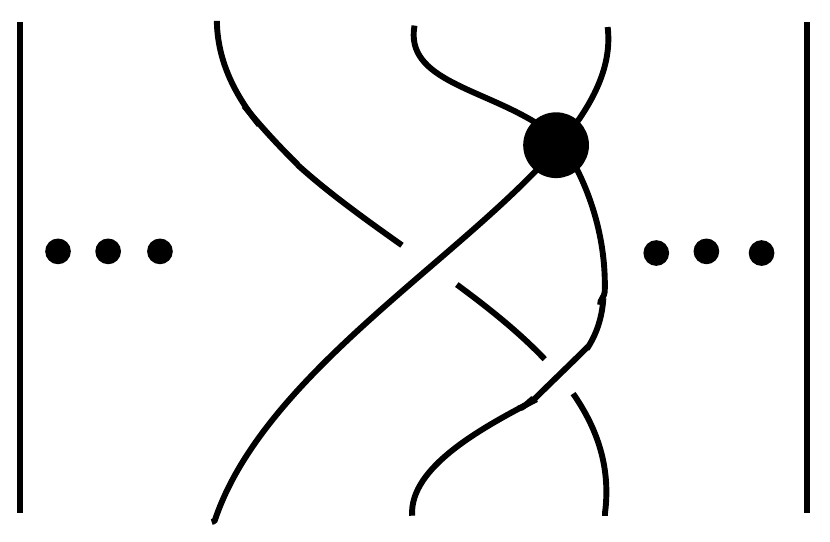}}
\end{eqnarray*}

\item $\sigma_i\tau_i=\tau_i\sigma_i$
\begin{eqnarray*}
\raisebox{-.5cm}{\includegraphics[height=0.5in]{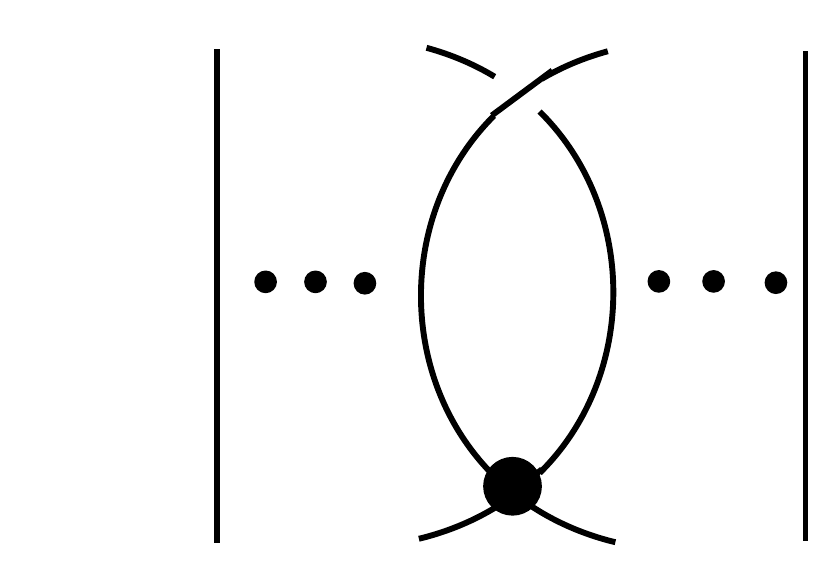}} \hspace{0.5cm} \stackrel{RS1}{=}  \hspace{-0.3cm} &\raisebox{-.5cm}{\includegraphics[height=0.5in]{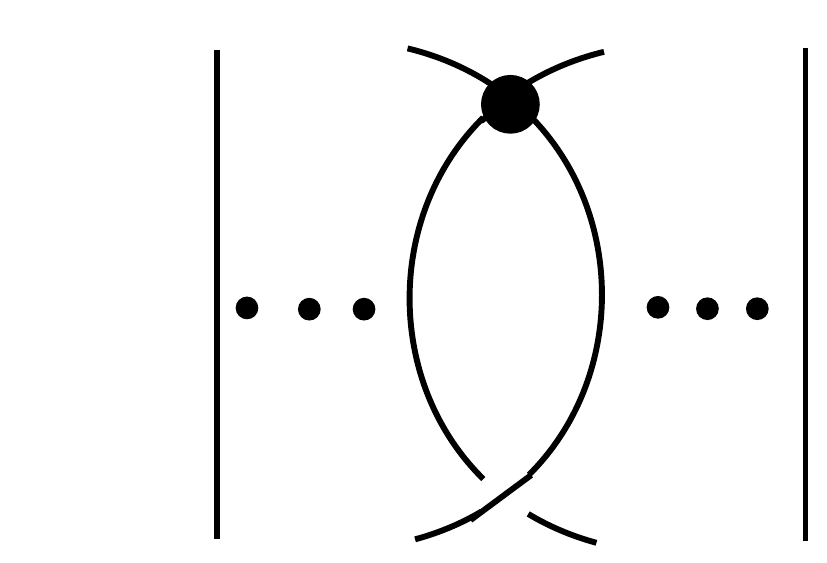}}
\end{eqnarray*}

\item $g_ih_i = h_ig_i, \,\, \forall g_i, h_i \in \{ \sigma_i, \tau_i, v_i  \} $ with $|i-j| >1$
\begin{eqnarray*}
\raisebox{-.5cm}{\includegraphics[height=0.5in]{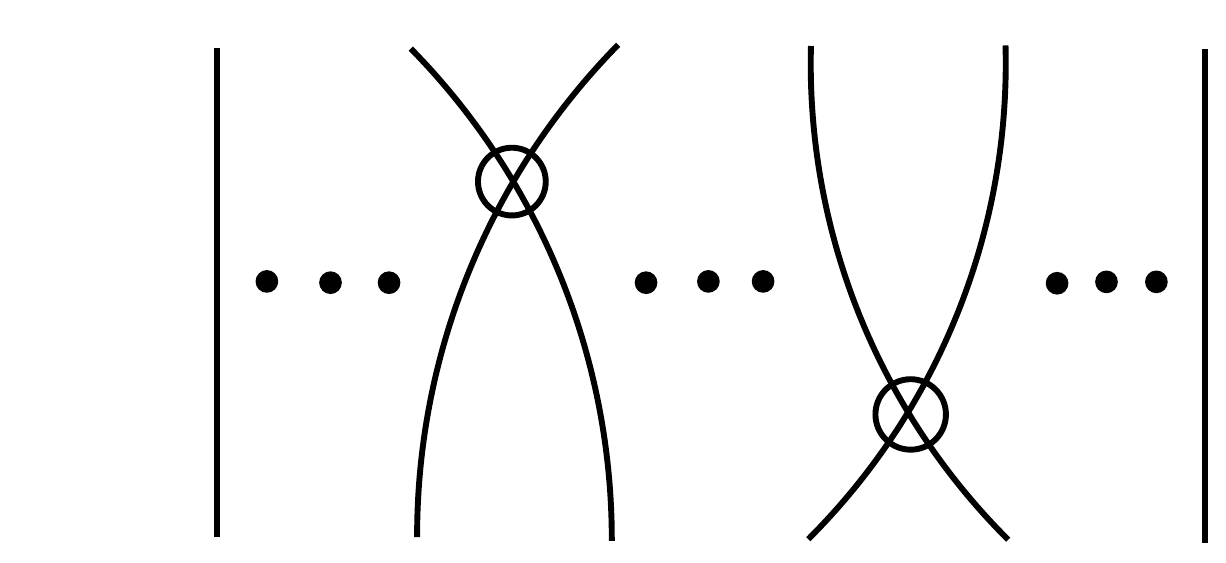}} \hspace{0.5cm} = \hspace{-.3cm} &\raisebox{-.5cm}{\includegraphics[height=0.5in]{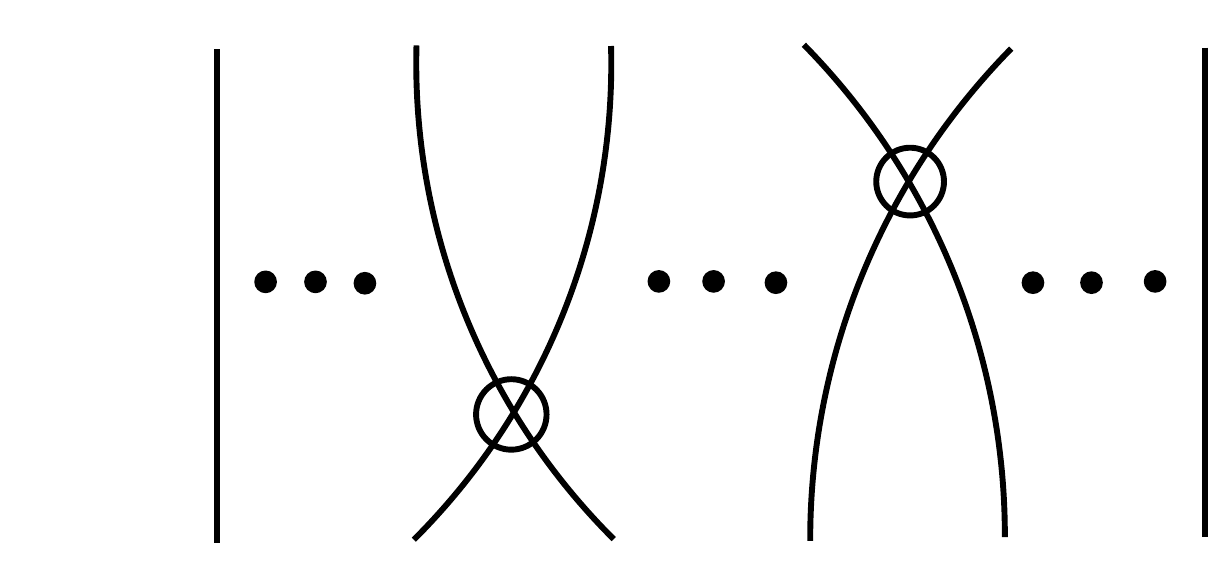}}
\end{eqnarray*}
\end{itemize}

These relations taken together define the isotopies for virtual singular braids. Each relation in ${VSB}_n$ is a braided version of a virtual singular link isotopy. That is, two equivalent virtual singular braids have isotopic closures. Note that the type 1 moves $R1$ and $V1$ are not reflected in the defining relations for ${VSB}_n$, because these moves cannot be represented using braids. Note also that only the generators $\tau_i$ are not invertible in ${VSB}_n$.

\subsection{A braiding algorithm} \label{sec:braiding}
In this section we present a method for transforming any virtual singular link diagram into the closure of a virtual singular braid. For that, we borrow the braiding algorithm introduced in~\cite{KL2} and extend it to our set-up where we add singular crossings, to prove a theorem for virtual singular links analogous to the Alexander theorem for classical braids and links.

We will work in the piecewise linear category, which gives rise to the operation of the \textit{subdivision} of an arc (in a virtual singular link diagram) into smaller arcs, by marking it with a point. Note that local minima and maxima are subdivision points of a diagram. 

\begin{definition}
We fix a height function in the plane of the diagram, and use the following conventions necessary for our braiding algorithm: First it is understood that only one crossing (real, singular or virtual) can occur at each level (with respect to the height function) in a virtual singular link diagram. Likewise, we arrange our diagram so that no crossings or subdivision points are vertically aligned, so as to avoid triple points when new pairs of braid strands are created with the same endpoints (this will be made more clear later as we explain our braiding algorithm). In addition, a crossing must not coincide with a local maximum or minimum. Lastly, a diagram should not have any horizontal arcs (it will only have \textit{up-arcs} and \textit{down-arcs}). If a virtual singular link is arranged so that it satisfies each of these conventions, we say that the diagram is in \textit{general position}. 
\end{definition}

It is easy to see that by applying small planar shifts, if necessary, any virtual singular link diagram can be transformed into a diagram in general position. 

When converting a virtual singular link diagram to a diagram in general position, we make certain choices which result in local shifts (which are called \textit{direction sensitive moves} in~\cite{KL2}) of crossings and subdivision points with respect to the horizontal or vertical direction. The \textit{swing moves} given in Figure~\ref{fig:swing moves} are the most interesting direction sensitive moves; these moves are necessary so that we avoid the coincidence between a crossing (real, singular, or virtual) and a maximum or minimum in a diagram.

\begin{figure}[ht]
\[ \raisebox{-10pt}{\includegraphics[height=0.4in]{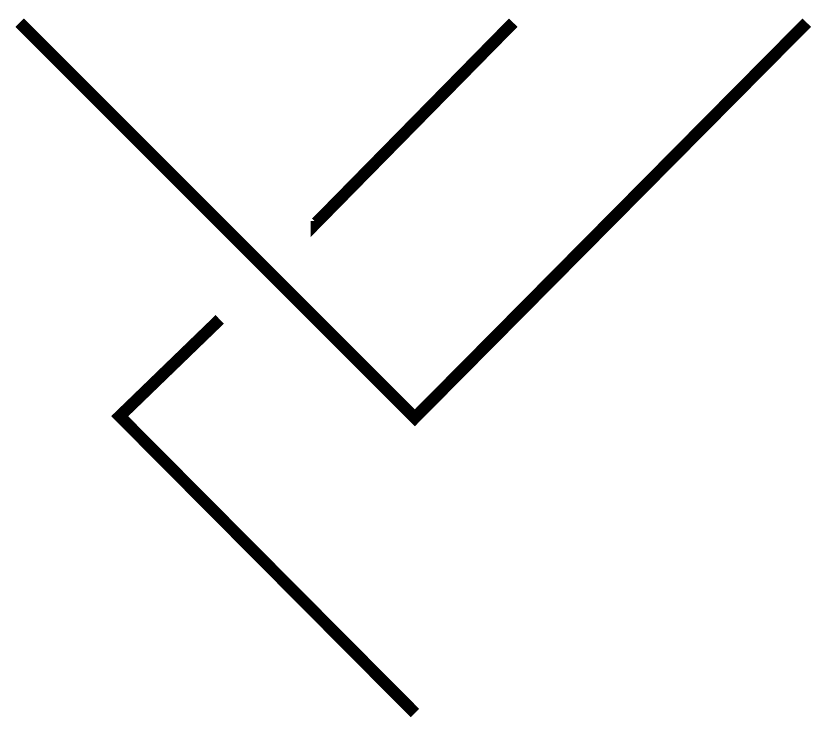}}\,\, \longleftrightarrow \,\, \raisebox{-10pt}{\includegraphics[height=0.4in]{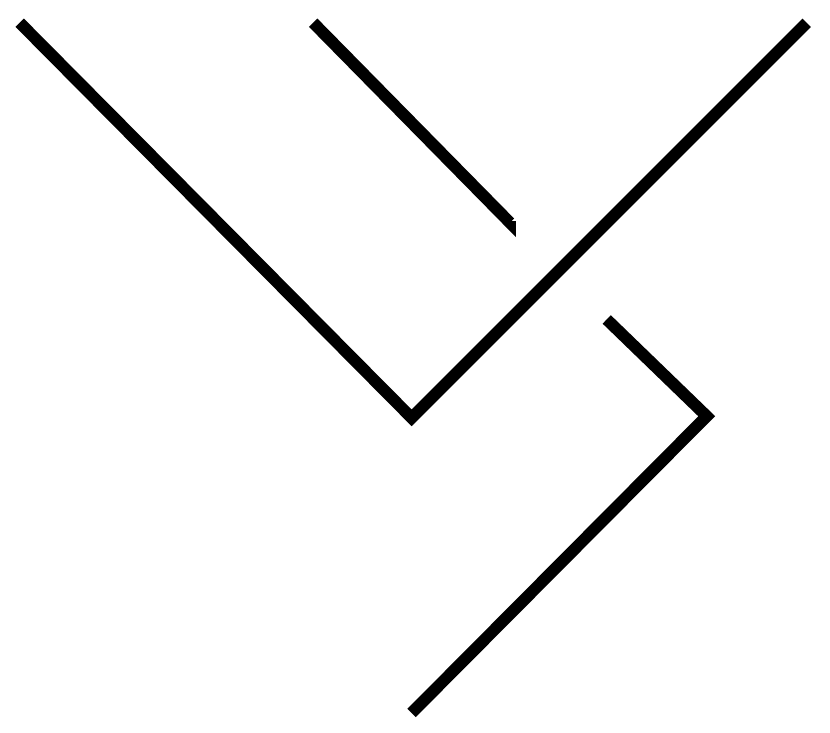}}
 \hspace{0.3cm} 
 \raisebox{-10pt}{\includegraphics[height=0.42in]{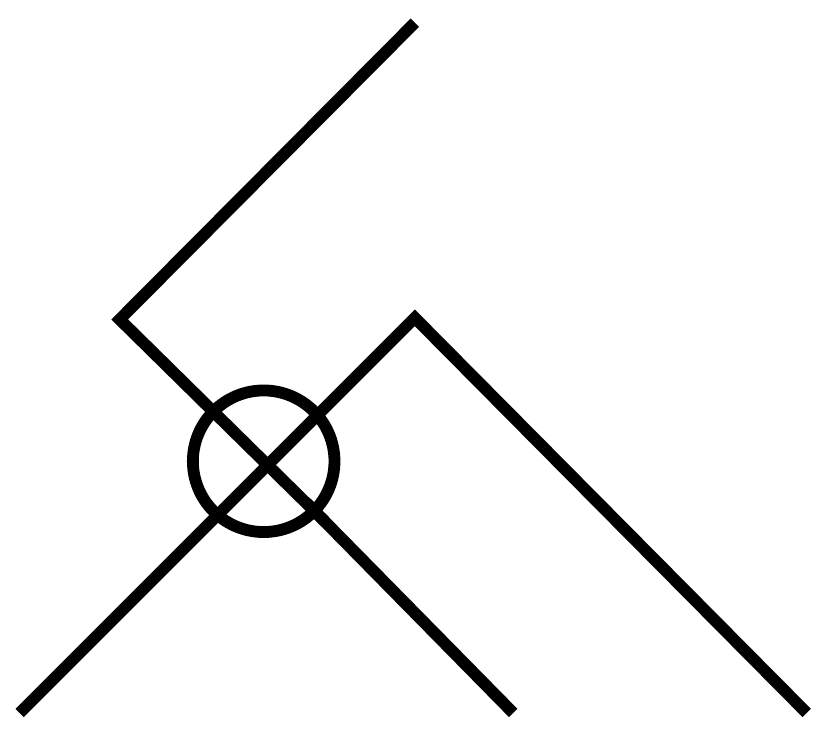}}\,\,\longleftrightarrow \,\, \raisebox{-10pt}{\includegraphics[height=0.42in]{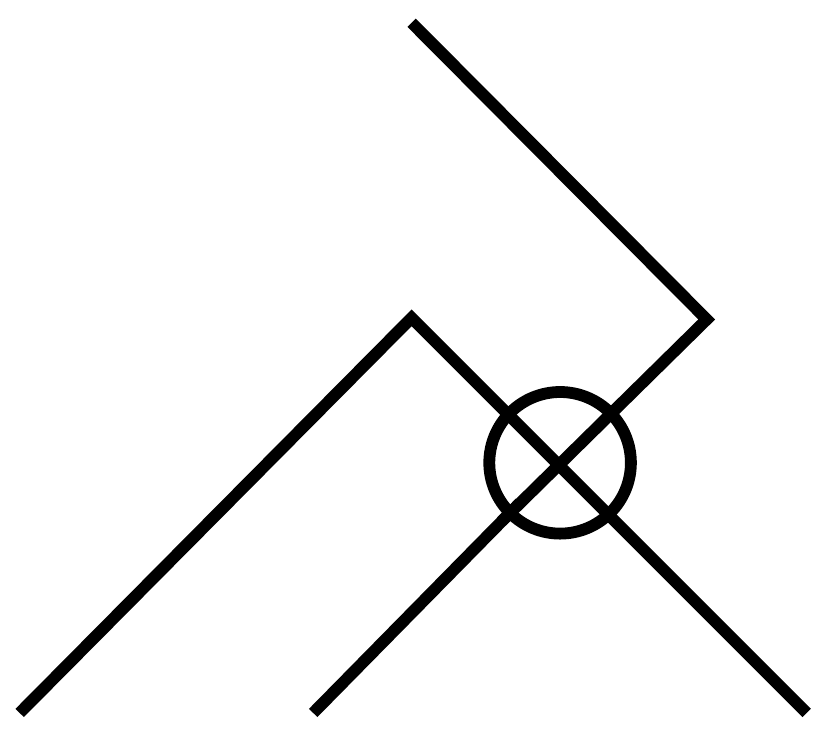}} \hspace{0.3cm} 
 \raisebox{-10pt}{\includegraphics[height=0.4in]{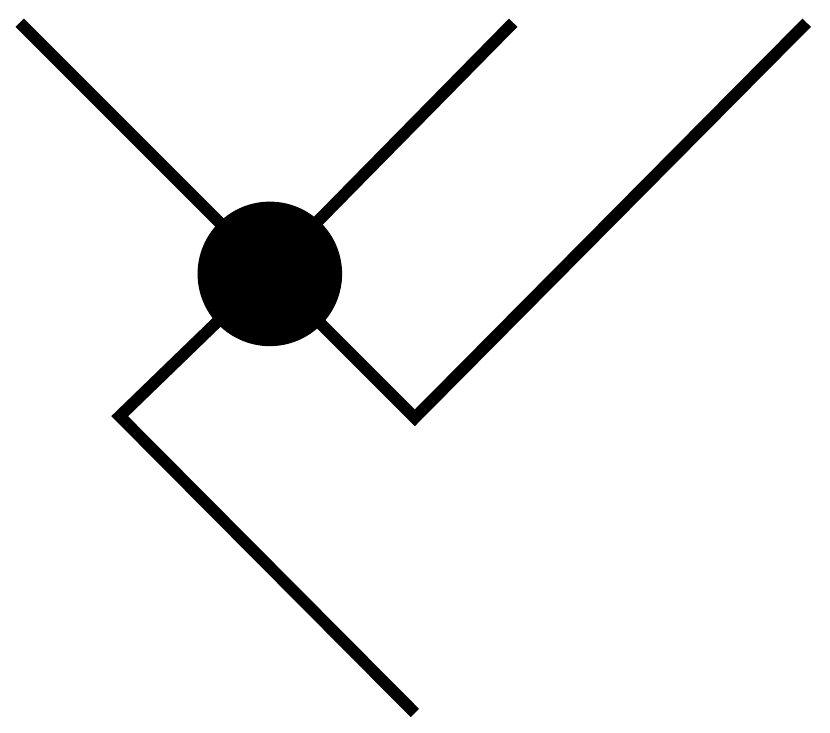}}\,\,\longleftrightarrow \,\, \raisebox{-10pt}{\includegraphics[height=0.4in]{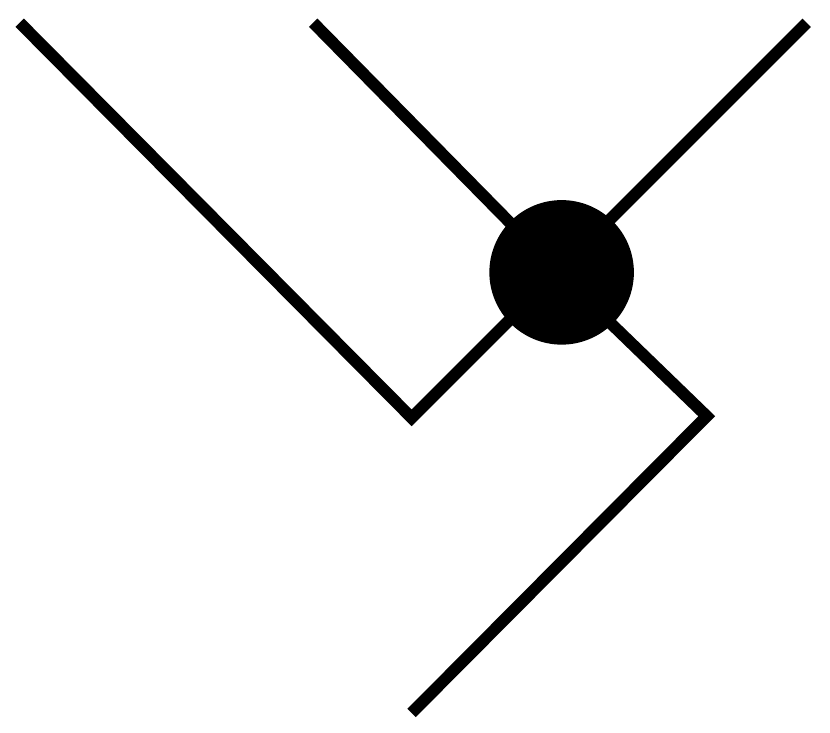}}\]
\caption{The swing moves}\label{fig:swing moves}
\end{figure} 

Two isotopic virtual singular link diagrams in general position differ by the extended virtual Reidemeister moves (provided in Figure~\ref{fig:isotopies}) and the direction sensitive moves. For the remainder of this section, we will work with virtual singular link diagrams in general position.

We describe now the \textit{braiding algorithm} for transforming an oriented virtual singular link diagram (assumed in general position) into the closure of a virtual singular braid.

After placing the subdivision points using the conventions explained above, we apply the braiding algorithm locally, by eliminating each up-arc in the diagram (which can be either an up-arc in a crossing or a free up-arc), one at a time.

We first braid the crossings containing one or two up-arcs. If a crossing has no up-arcs we leave it as it is. We place each crossing that needs to be braided in a narrow rectangular box, called the \textit{braiding box}, with the arcs of the crossing serving as diagonals of the box. A braiding box would have to be sufficiently narrow, so that the region it defines does not intersect the braiding box of another crossing. We braid each crossing, one at a time, according to the \textit{braiding chart} given in Figure~\ref{fig:b-c} (see also~\cite[Figure 7]{KL2}). Any new crossing created between the new braid strands and the rest of the diagram outside the braiding box will be assumed to be virtual; this is indicated abstractly by putting virtual crossings at the ends of the new pair of braid strands. 

\begin{figure}[ht] 
\begin{center}
  \begin{tabular}{ l | l | l }
    \hline
    $\reflectbox{ \raisebox{5pt}{\includegraphics[height=0.4in]{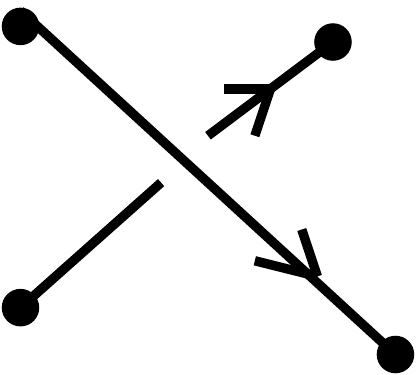}} }\,\, \raisebox{20pt}{$\longrightarrow$} \,\, \reflectbox{\raisebox{-30pt}{\includegraphics[height=1.2in]{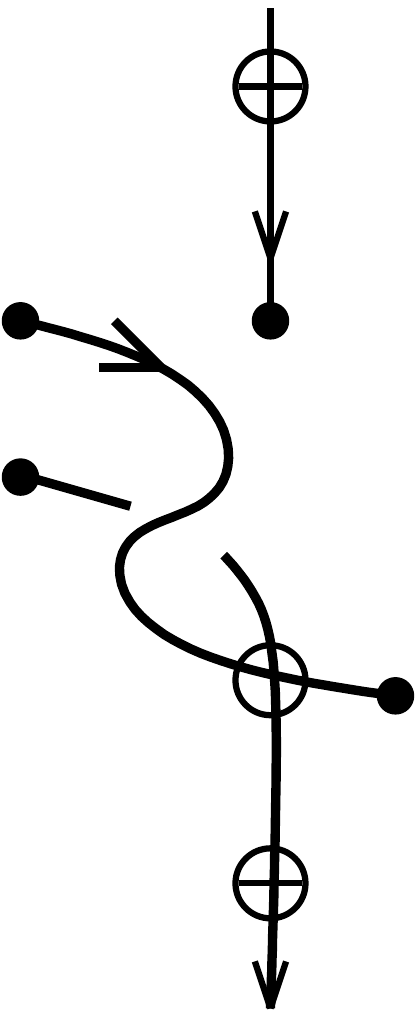}}} $
 &
 $ \reflectbox{\raisebox{5pt}{\includegraphics[height=0.4in]{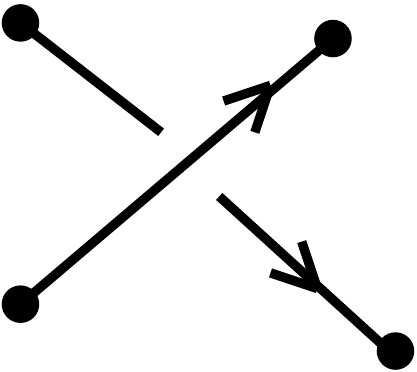}}}  \,\, \raisebox{20pt}{$\longrightarrow$} \,\,  \reflectbox{\raisebox{-30pt}{\includegraphics[height=1.2in]{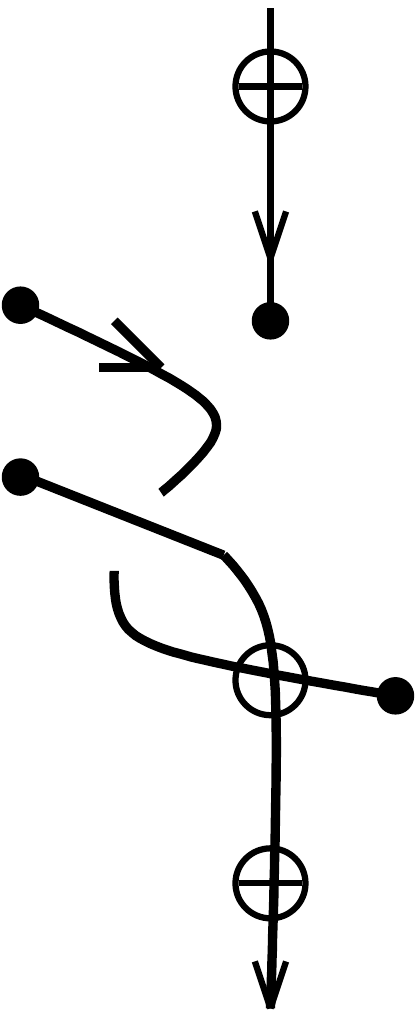}}} $ 
&
 $ \reflectbox{\raisebox{5pt}{\includegraphics[height=0.4in]{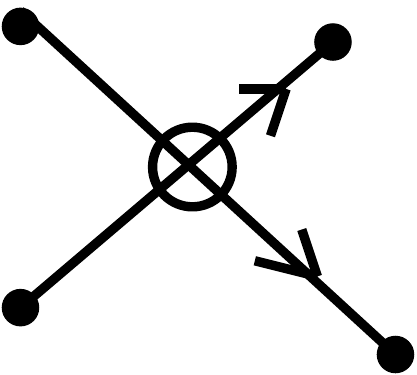}}}  \,\, \raisebox{20pt}{$\longrightarrow$} \,\,  \reflectbox{\raisebox{-25pt}{\includegraphics[height=1.1in]{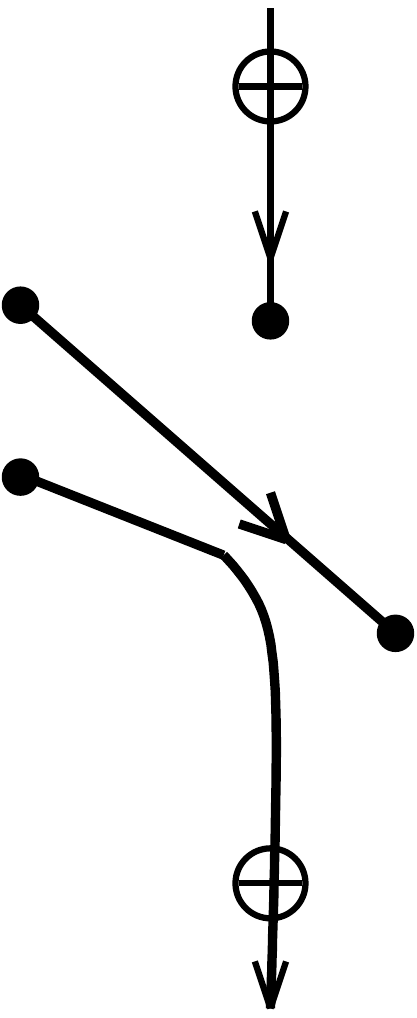}} }$ \\ \hline
    $\reflectbox{\raisebox{5pt}{\includegraphics[height=0.4in]{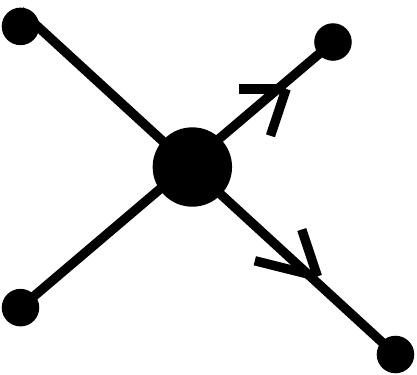}} } \,\, \raisebox{20pt}{$\longrightarrow$} \,\,  \reflectbox{\raisebox{-30pt}{\includegraphics[height=1.2in]{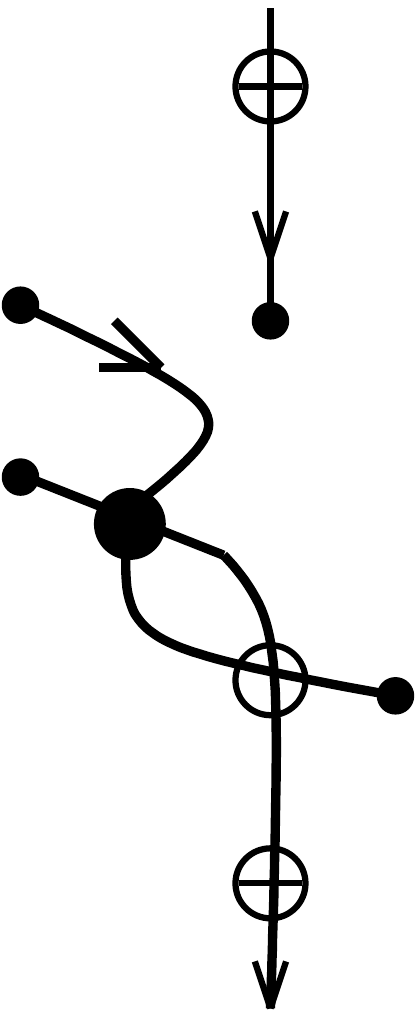}}} $ 
&
 $ \raisebox{5pt}{\includegraphics[height=0.4in]{Classical3A-new}}  \,\, \raisebox{20pt}{$\longrightarrow$} \,\, \raisebox{-30pt}{\includegraphics[height=1.2in]{Classical3B-new}} $ 
&
 $ \raisebox{5pt}{\includegraphics[height=0.4in]{Classical4A-new}} \,\, \raisebox{20pt}{$\longrightarrow$} \,\, \raisebox{-30pt}{\includegraphics[height=1.2in]{Classical4B-new}} $ \\ 
 \hline
 $ \raisebox{5pt}{\includegraphics[height=0.4in]{Virtual2A-new}} \,\, \raisebox{20pt}{$\longrightarrow$} \,\,\raisebox{-25pt}{\includegraphics[height=1.1in]{Virtual2B-new}} $ 
&
 $ \raisebox{5pt}{\includegraphics[height=0.4in]{Singular2A-new}} \,\, \raisebox{20pt}{$\longrightarrow$} \,\, \raisebox{-30pt}{\includegraphics[height=1.2in]{Singular2B-new}} $
 & $ \raisebox{10pt}{\includegraphics[height=0.4in]{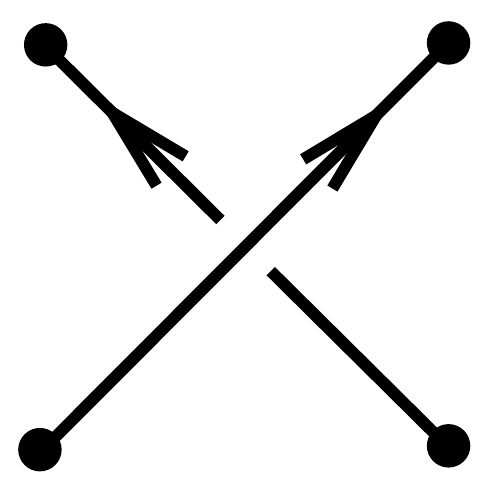}} \,\, \raisebox{25pt}{$\longrightarrow$} \,\, \raisebox{-40pt}{\includegraphics[height=1.4in]{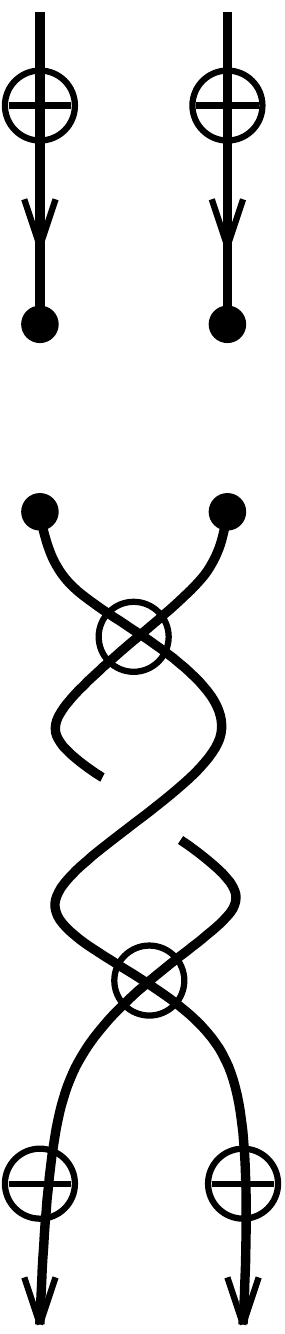}}  $ \\ 
 \hline
  $ \reflectbox{\raisebox{10pt}{\includegraphics[height=0.4in]{Classical5A}}}  \,\, \raisebox{25pt}{$\longrightarrow$} \,\,  \reflectbox{\raisebox{-40pt}{\includegraphics[height=1.4in]{Classical5B-new}} }$  
&
 $ \raisebox{10pt}{\includegraphics[height=0.37in, angle=90]{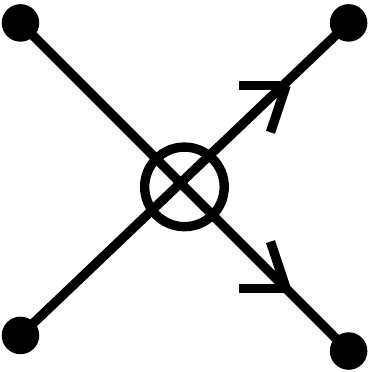}} \,\, \raisebox{20pt}{$\longrightarrow$} \,\, \raisebox{-30pt}{\includegraphics[height=1.2in]{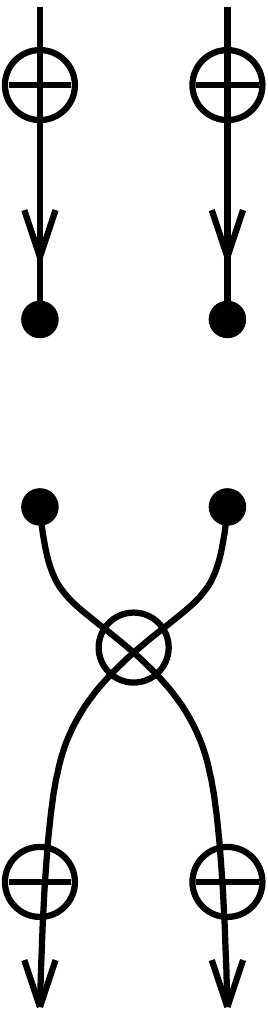}}$ 
  &
  $ \raisebox{10pt}{\includegraphics[height=0.35in, angle = 90]{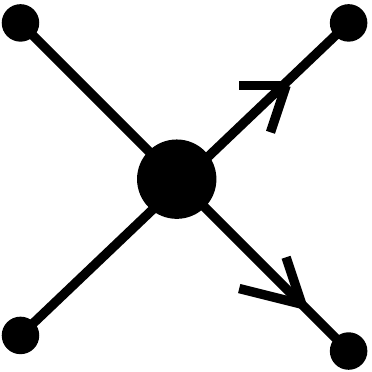}} \,\, \raisebox{20pt}{$\longrightarrow$} \,\,  \raisebox{-45pt}{\includegraphics[height=1.4in]{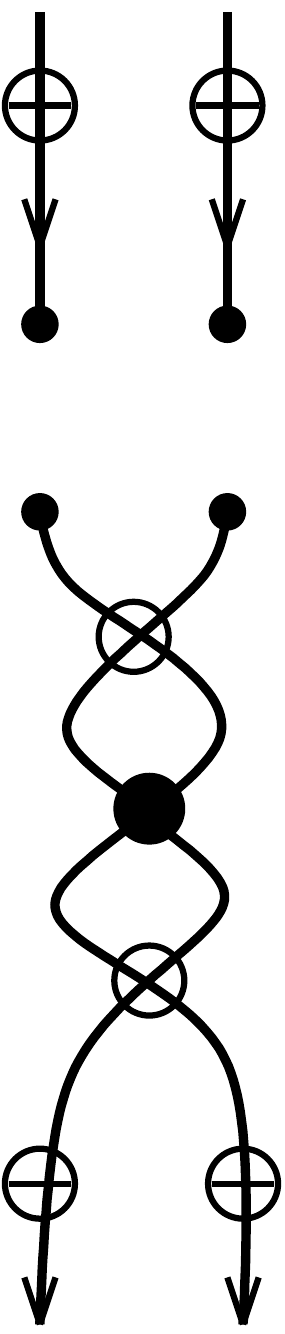}} $
          \\ \hline
\end{tabular}
\end{center}
\caption{The braiding chart for crossings} \label{fig:b-c}
\end{figure} 

Note that, locally speaking, for each crossing that was braided, connecting the corresponding pair of braid strands (outside of the resulting diagram) yields a virtual singular tangle diagram (with four endpoints) which is isotopic to the starting one (the tangle represented by the crossing in the braiding box).

The free up-arcs are arcs joining braiding boxes. Once all crossings have been braided, we braid each of the free up-arcs using the \textit{basic braiding move} depicted in Figure~\ref{fig:basic-braiding} (see~\cite[Figure 9]{KL2}). During this move, we first cut a free up-arc and then extend the upper end upward and the lower end downward, such that the new pair of strands are vertically aligned and such that they cross only virtually any other arcs in the original diagram (which is represented by an abstract virtual crossing on the ends of the new braid strands), as shown in Figure~\ref{fig:basic-braiding}. As in the case of braiding a crossing, by connecting the pair of the new braid strands outside of the original diagram results in a local virtual singular tangle diagram (with two endpoints) which is isotopic to the local tangle before the braiding.
 
  \begin{figure}[ht]
\[ \raisebox{-15pt}{\includegraphics[height=.5in]{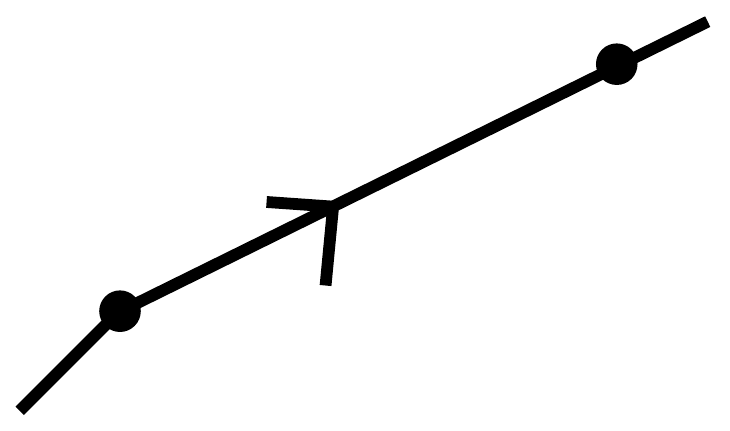}}\,\,{\longrightarrow} \,\,\raisebox{-40pt}{\includegraphics[height=1.2in]{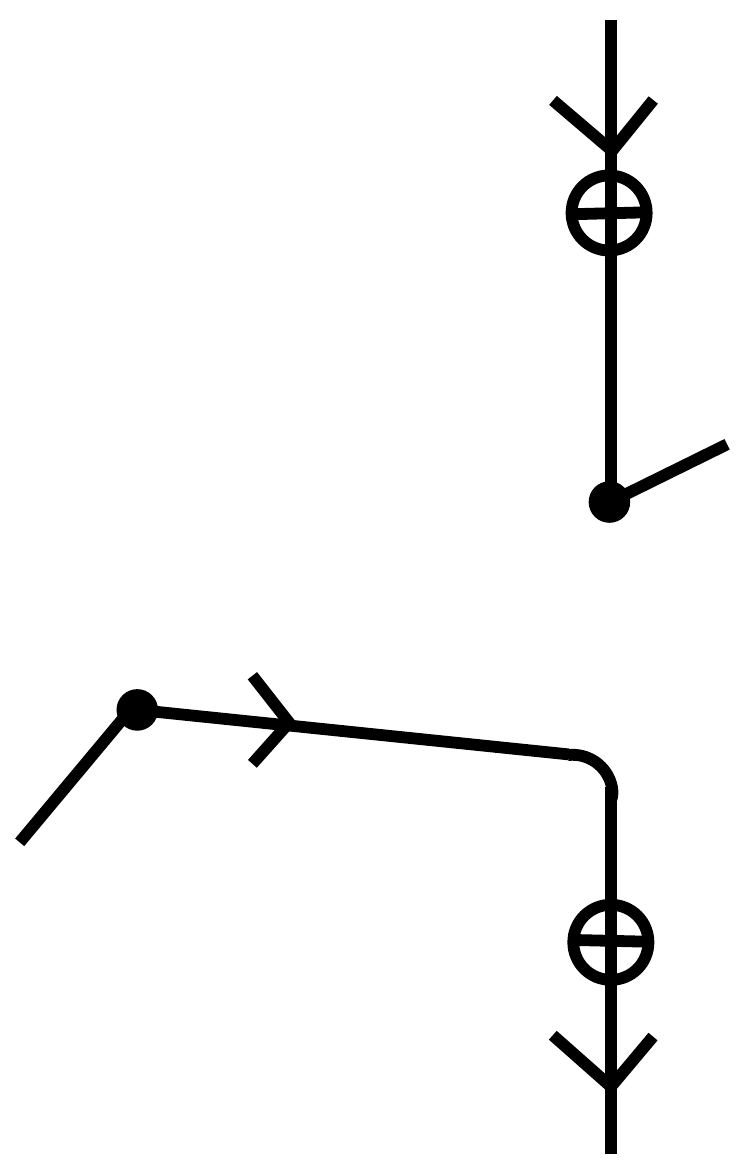}}\,\, \]
\caption{A basic braiding move}\label{fig:basic-braiding}
\end{figure}

The braiding algorithm given above will braid any virtual singular link diagram, creating a virtual singular braid whose closure is isotopic to the original diagram. Indeed, for all braiding moves, even for those that do not contain singular crossings, it is important to observe that there may be singular crossings in the rest of the braid and that upon closure these are detoured freely by the virtual crossings of the new braid strands. Therefore, we have proved the following statement.

\begin{theorem}[\textbf{Alexander-type theorem for virtual singular links}] \label{Alexander thm}
Every oriented virtual singular link can be represented as the closure of a virtual singular braid.
\end{theorem}
%%%%%%%%%%%%%%%%%%%%%%%%%%%%%%%%%%%%%%%%%%%%%%%%%%%
\subsection{L-moves and Markov-type theorems for virtual singular braids}\label{sec:L-moves}

Two virtual singular braids may have isotopic closures, and thus we would like to describe virtual singular braids that result in isotopic virtual singular link diagrams via the closure operation. Therefore, we are interested in Markov-type theorems for virtual singular braids and links. For this purpose, we need to introduce the \textit{singular $L_v$-moves} for virtual singular braids. These moves enlarge the set of the $L_v$-moves for virtual braids, described in~\cite{KL2}. Here, the subscript $v$ stands for `virtual'.  

We remind the reader that the classical $L$-moves were introduced by S. Lambropoulou in~\cite{L} to provide a one-move Markov-type theorem for classical braids and links. We also refer the reader to~\cite{LR}, where the $L$-move equivalence for classical braids is established. 

We recall from~\cite{KL2} that a \textit{basic $L_v$-move} involves cutting a braid strand and pulling the upper endpoint of the cut downward and the lower endpoint upward, and in doing so, creating a pair of new braid strands which cross virtually all of the other strands in the diagram; this is abstractly denoted by a pair of virtual crossings at the points where the two new braid strands cross the box in which the $L_v$-move is applied (see Figure~\ref{fig:basic L_v move}). 
\begin{figure}[ht]
\[\raisebox{-30pt}{\includegraphics[height=1in]{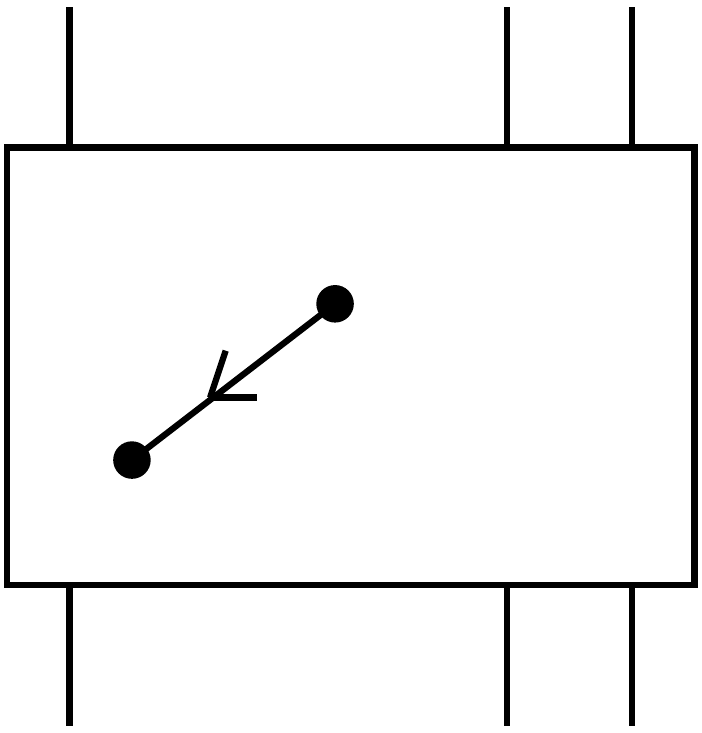}}\,\, \stackrel{\text{basic} \ L_v-\text{move}}{\longrightarrow} \,\, \raisebox{-45pt}{\includegraphics[height=1.4in]{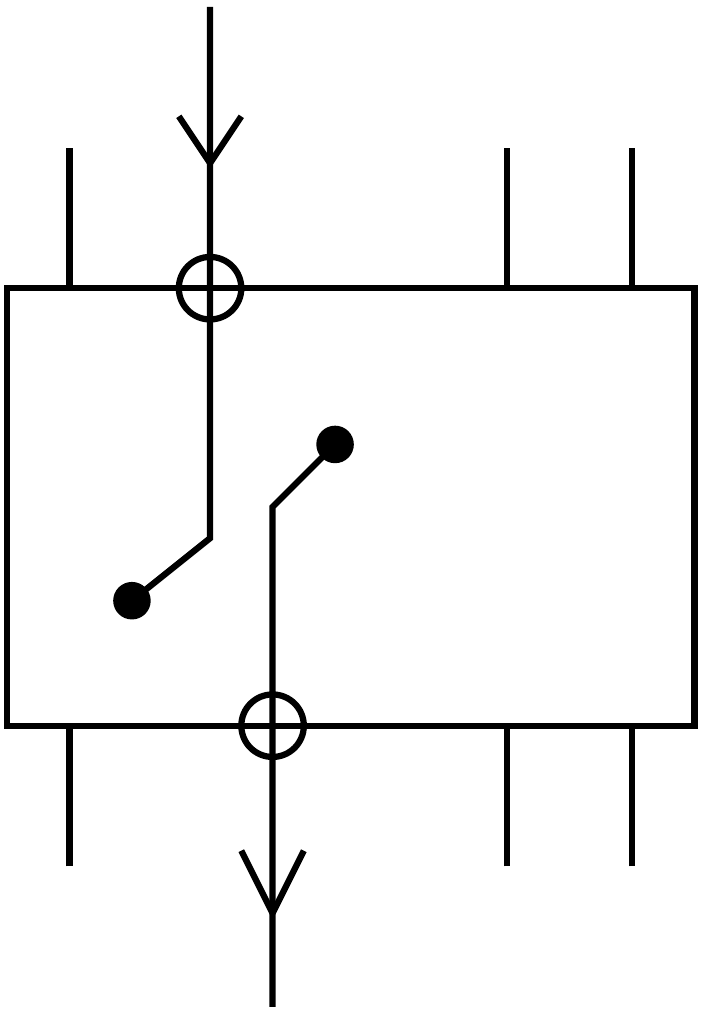}}\]
\caption{A basic $L_v$-move}\label{fig:basic L_v move}
\end{figure}

Note that an $L_v$-move may introduce a crossings, which may be real or virtual, as shown in Figure~\ref{fig:v-r L_v moves} (see~\cite[Figure 11]{KL2}). To stress the existence of the real or virtual crossing, these moves are called the \textit{real $ L_v$-move} or \textit{virtual $ L_v$-move}, respectively (abbreviated to $rL_v$- or $vL_v$-move, respectively), and there are two versions of them, namely \textit{left} or \textit{right} (depending whether the new crossing is on the left or on the right of the arc that was cut during the move). Figure~\ref{fig:v-r L_v moves} displays right virtual and left real $L_v$-moves.
\begin{figure}[ht]
\[   \raisebox{-45pt}{\includegraphics[height=1.4in]{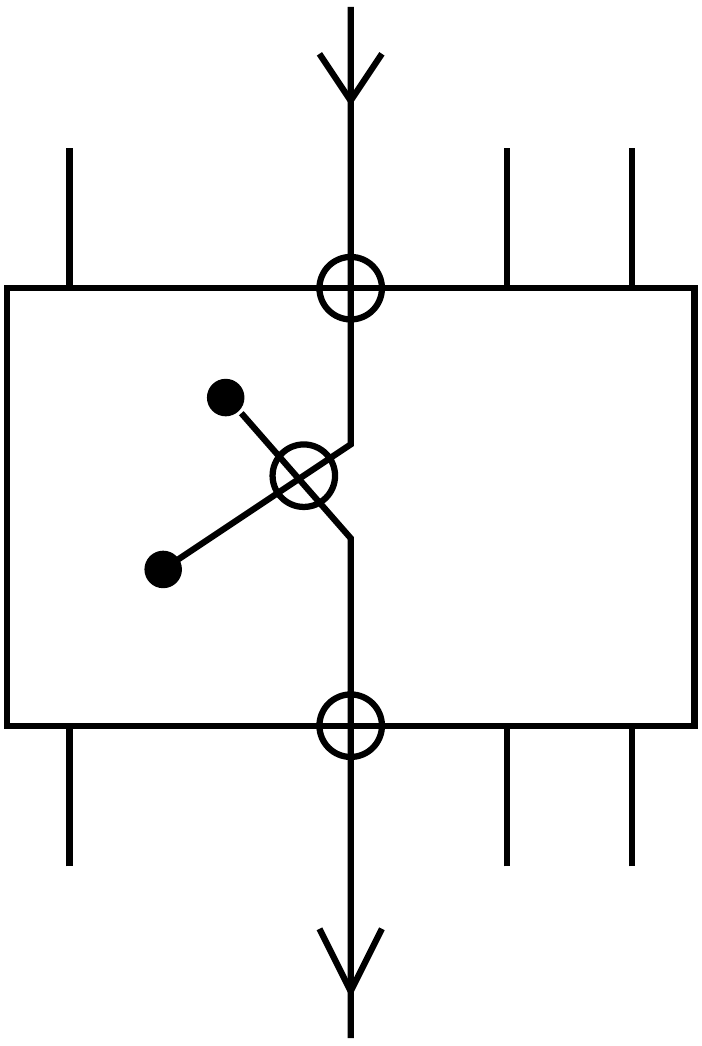}}\,\,\stackrel{\text{right}\ vL_v-\text{move}}{\longleftarrow}
 \,\,\raisebox{-30pt}{\includegraphics[height=1in]{basic-L1}}\,\, \stackrel{\text{left}\  rL_v-\text{move}}{\longrightarrow} \,\, \raisebox{-45pt}{\includegraphics[height=1.4in]{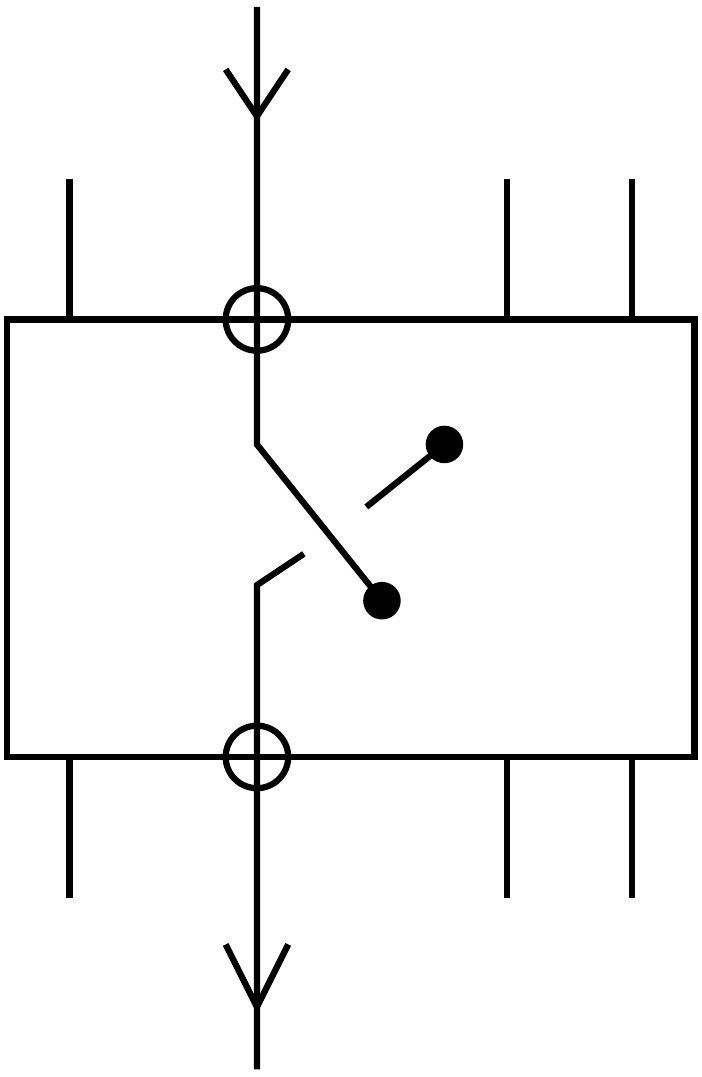}}\]
\caption{Right virtual and left real $L_v$-moves}\label{fig:v-r L_v moves}
\end{figure}

Note that by connecting the pair of the newly created braid strands (outside of the diagram) we obtain a tangle diagram which is isotopic to the tangle diagram we started with (the detoured loop contracts to a kink which involves either a virtual crossing or a real crossing). This is explained in Figure~\ref{fig:close L_v moves}.

\begin{figure}[ht]
\[\raisebox{-13pt}{\includegraphics[height=.4in]{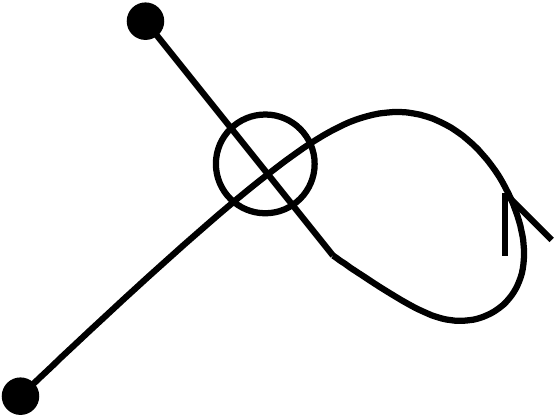}} \stackrel{V1}{\longleftrightarrow}\raisebox{-13pt}{\includegraphics[height=.4in]{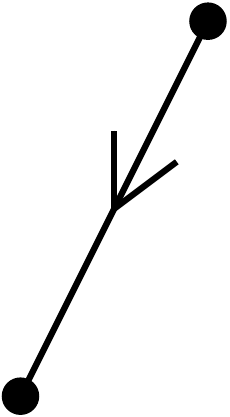}}  \stackrel{R1}{\longleftrightarrow}\raisebox{-13pt}{\includegraphics[height=.4in]{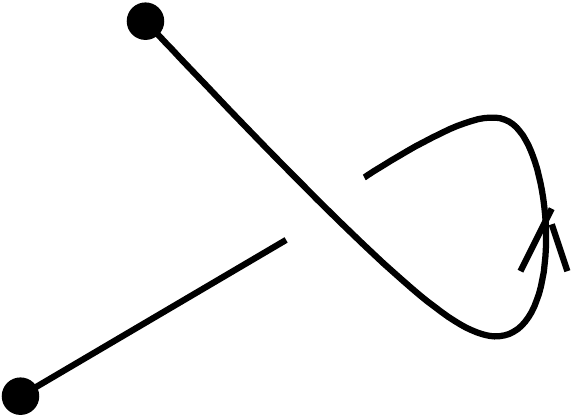}}\]
\caption{Closures of right virtual and right real $L_v$-moves}\label{fig:close L_v moves}
\end{figure}

\begin{definition}
A \textit{threaded $L_v$-move} is an $L_v$-move with a virtual crossing in which, before stretching the arc of the kink, we perform a classical type 2 Reidemeister move using another strand of the braid, called the \textit{thread}. Depending whether we pull the kink over or under the thread, we have an \textit{over-threaded $L_v$-move} or an \textit{under-threaded $L_v$-move}; both of these moves come with the \textit{left} and \textit{right} versions. (We refer the reader to the analogous definition in~\cite[Definition 4]{KL2}.)

\end{definition}

Figure~\ref{fig:threaded-move} shows under-threaded $L_v$-moves, both left and right versions. 
Due to the forbidden moves, a threaded $L_v$-move cannot be simplified on the braid level; that is, the move does not involve isotopic braids but isotopic closures of braids. 

In addition, we can create a \textit{multi-threaded $L_v$-move} by performing two or more classical type 2 Reidemeister moves before pulling open the arc of the kink. (See~\cite[Figure 14]{KL2}.)

\begin{figure}[ht]
\[\reflectbox{\raisebox{-50pt}{\includegraphics[height=1.4in]{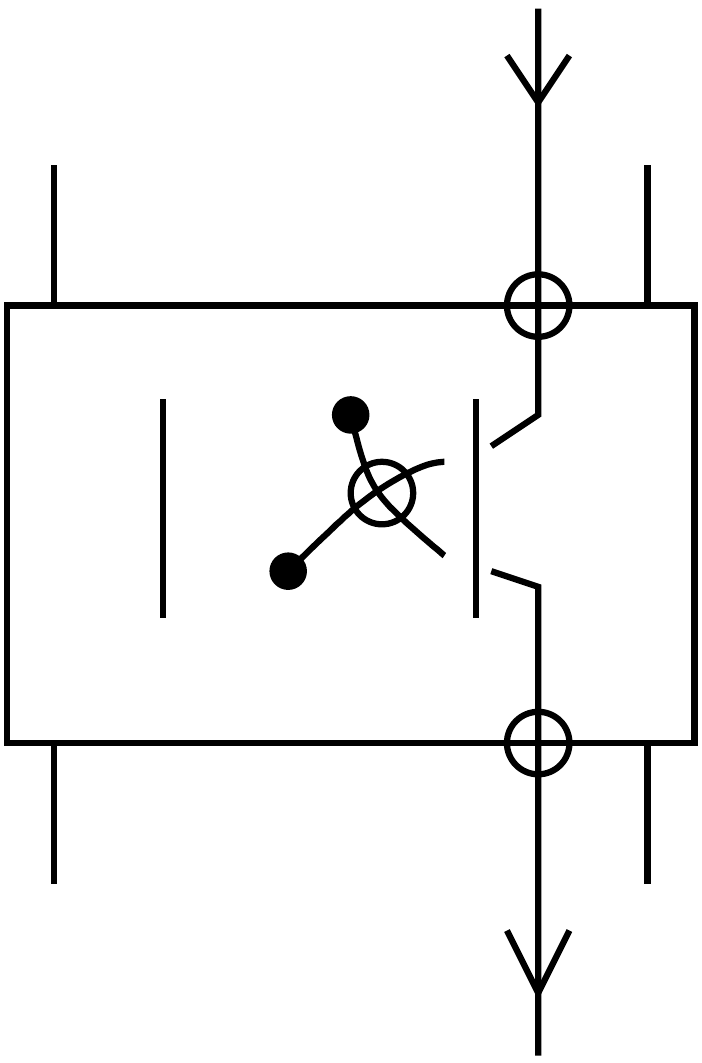}}}\,\,\,\, \longleftarrow \,\,\,\, \raisebox{-35pt}{\includegraphics[height=1in]{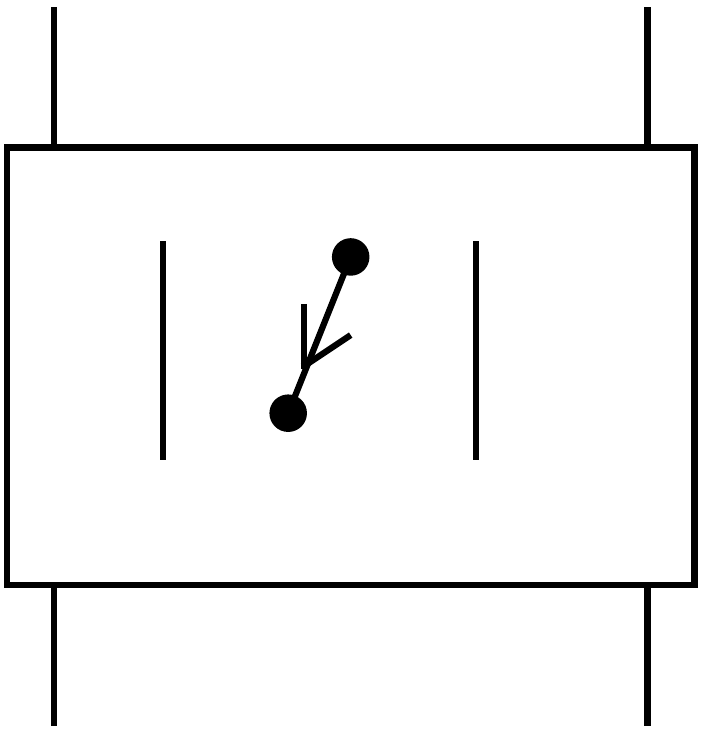}}\,\,\,\, \longrightarrow \,\, \,\, \raisebox{-50pt}{\includegraphics[height=1.4in]{right-threaded-move}}\]
\caption{Left and right under-threaded $L_v$-moves}\label{fig:threaded-move}
\end{figure}

When singular crossings are present, there is another type of threaded move in which the thread `crosses' the detoured loop in a pair of a singular crossing and a real crossing. We call such a move an \textit{$rs$-threaded $L_v$-move}; this move also comes in two variants, namely \textit{left} and \textit{right}. Figure~\ref{fig:rs-threaded move} exemplifies such a move, with only one of the two versions for the real crossing involved in the move. An $rs$-threaded $L_v$-move cannot be applied (simplified) in the braid. However, it is not hard to see that the closures of the two sides of an $rs$-threaded $L_v$-move are isotopic diagrams (via an $RS1$ move), as explained in Figure~\ref{fig:closed-rs-threaded}.

\begin{figure}[ht]
\[ \reflectbox{\raisebox{-49pt}{\includegraphics[height=1.4in]{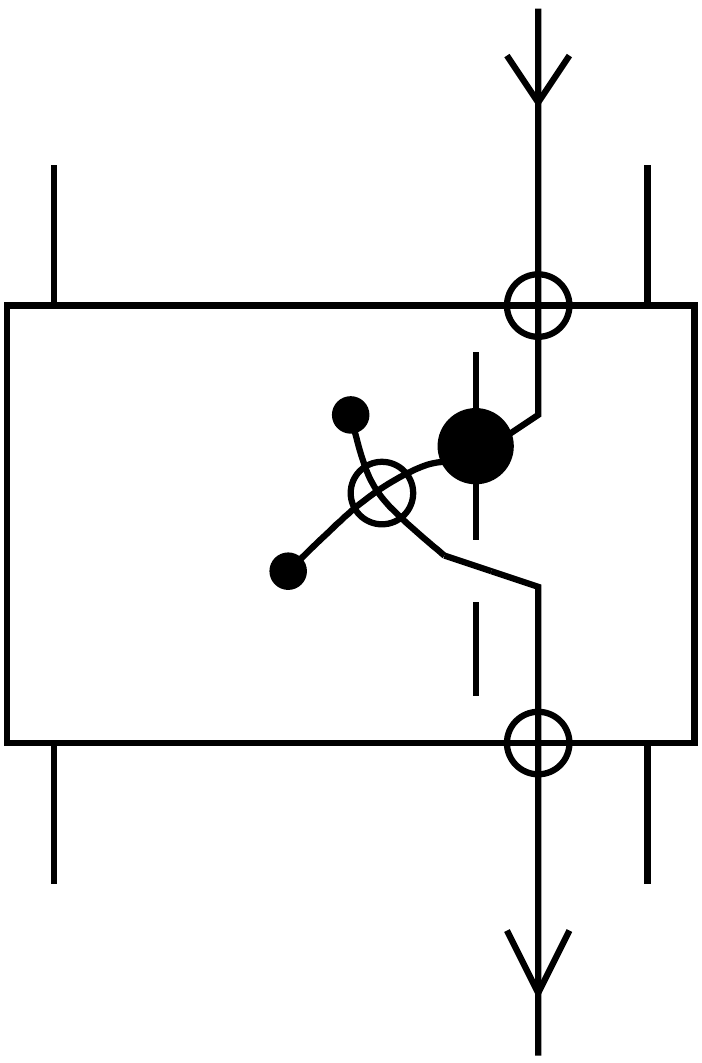}}}\,\, \longleftrightarrow \,\, \reflectbox{\raisebox{-49pt}{\includegraphics[height=1.4in]{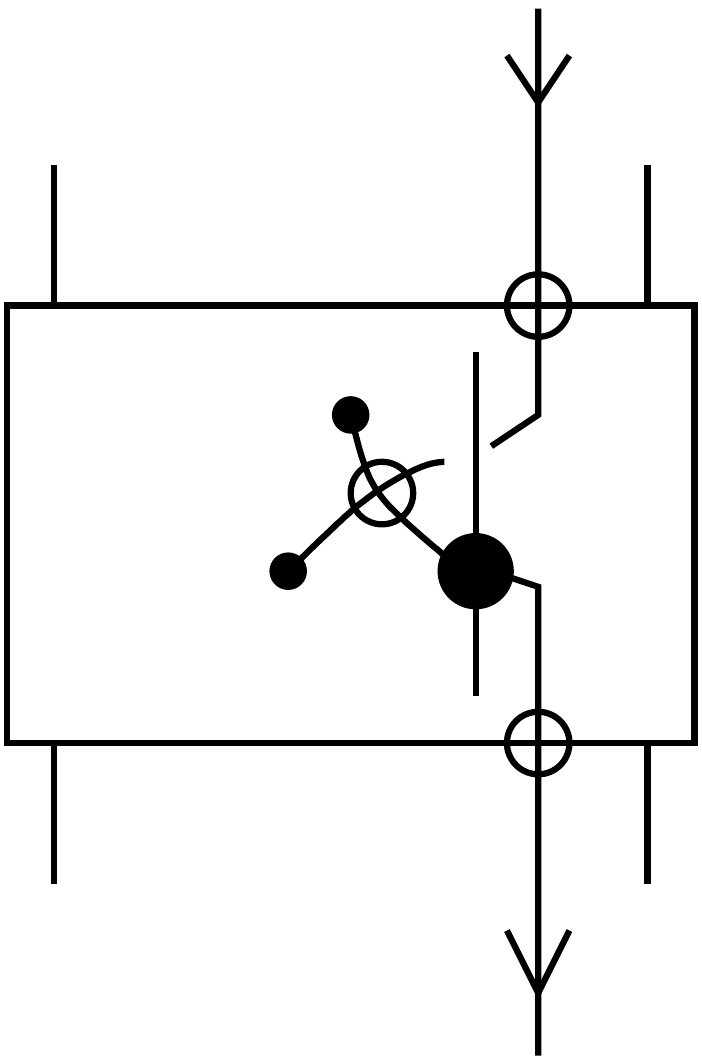}}}
 \hspace{1cm} 
 \raisebox{-49pt}{\includegraphics[height=1.4in]{right-rv-threaded-move1}}\,\,\longleftrightarrow \,\, \raisebox{-49pt}{\includegraphics[height=1.4in]{right-rv-threaded-move2}}\]
\caption{Left and right $rs$-threaded $L_v$-moves}\label{fig:rs-threaded move}
\end{figure}

\begin{figure}[ht]
\[\reflectbox{\raisebox{-19pt}{\includegraphics[height=0.5in]{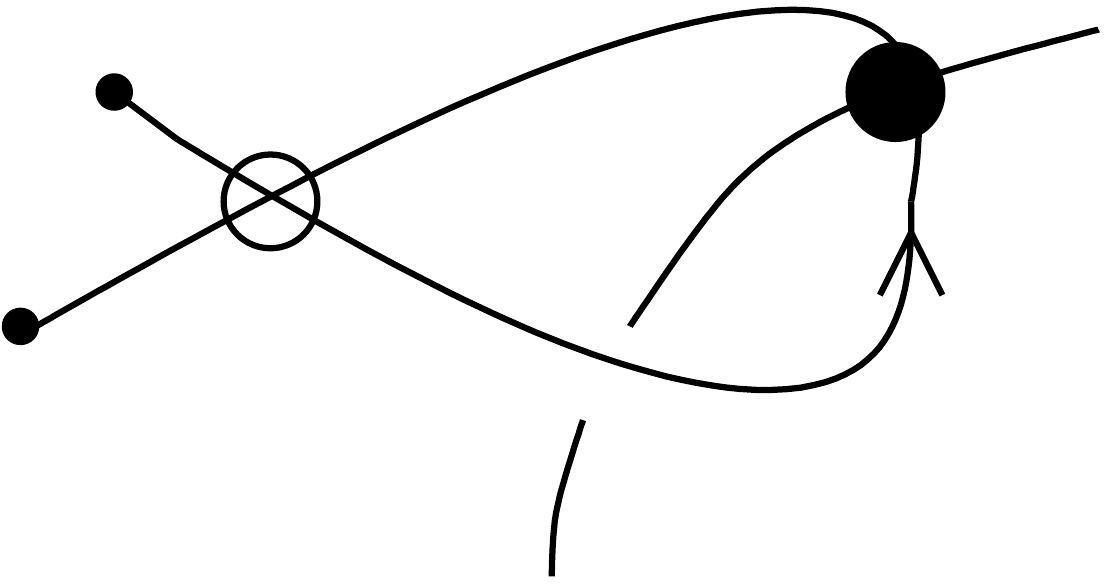}}}\,\,\,\,  \stackrel{RS1}{\longleftrightarrow} \,\, \reflectbox{\raisebox{-19pt}{\includegraphics[height=0.5in]{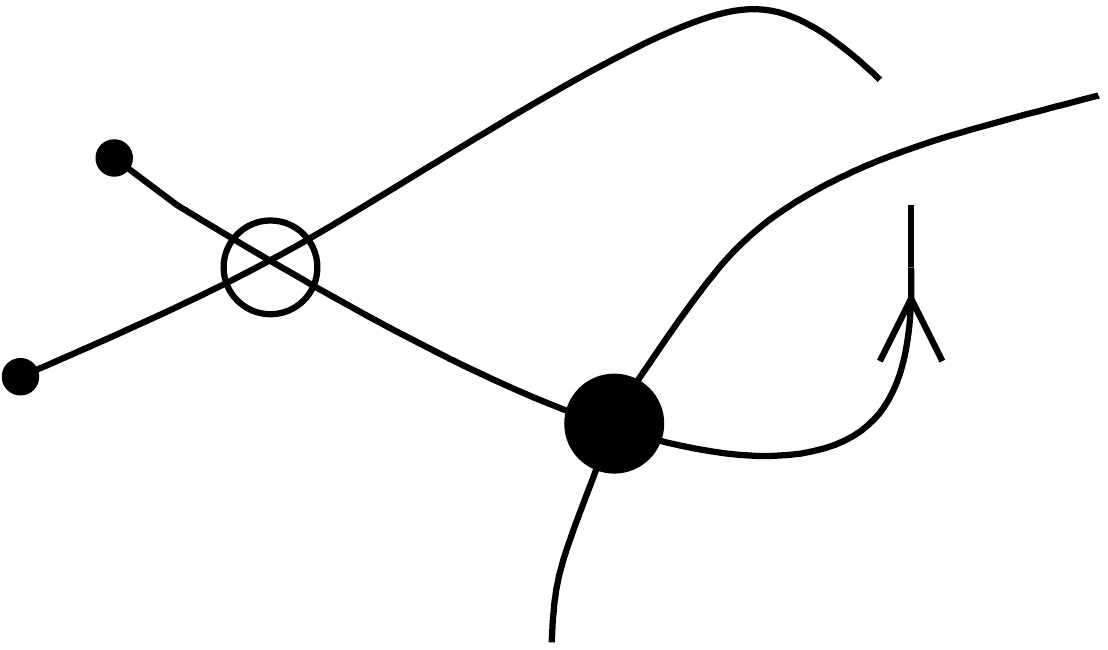}}} \hspace{1.5cm} \raisebox{-19pt}{\includegraphics[height=0.5in]{cl-right-rv1}}\,\,  \stackrel{RS1}{\longleftrightarrow} \,\, \,\, \raisebox{-19pt}{\includegraphics[height=0.5in]{cl-right-rv2}} \]
\caption{Closures of left and right $rs$-threaded $L_v$-moves}\label{fig:closed-rs-threaded}
\end{figure}

Finally, we define the notion of \textit{conjugation} and \textit{commuting} in the virtual singular braid monoid, ${VSB}_n$. Given a virtual singular braid $\omega \in {VSB}_n$, we say that the braids $\omega \sigma _i^{\pm 1} \sim \sigma_i^{\pm 1} \omega$, for $1\leq i \leq n-1$, are related by \textit{real conjugation}. Similarly, we say that the braids $\omega v_i \sim v_i \omega$ are related by \textit{virtual conjugation}, where $1\leq i \leq n-1$. Note that since $v_i$ is its own inverse in ${VSB}_n$, virtual conjugation is equivalent to $\omega \sim v_i \omega v_i$. Similarly, real conjugation can be rewritten in the form $\omega \sim \sigma_i \omega \sigma_i^{-1}$ or $\omega \sim \sigma_i^{-1} \omega \sigma_i$. Finally, we say that $\omega \tau_i \sim \tau_i \omega$ are related by \textit{singular commuting} (note that $\tau_i$ is not invertible in ${VSB}_n$). 

\begin{definition}
We say that two virtual singular braids are \textit{singular $L_v$-equivalent} if they differ by virtual singular braid isotopy and a finite sequence of the following moves or their inverses:
\begin{enumerate}
\item[(i)] Real conjugation and singular commuting
\item[(ii)] Right virtual and right real $L_v$-moves
\item[(iii)] Left and right under-threaded $L_v$-moves
\item[(iv)] Left and right $rs$-threaded $L_v$-moves.
\end{enumerate}
\end{definition}

 We remark that the singular $L_v$-equivalence on virtual singular braids contains as a subset the $L$-equivalence on virtual braids defined in~\cite[Definition 6]{KL2}. We remind the reader that the $L$-equivalence for virtual braids comprises the real conjugation, the right real and right virtual $L_v$-moves, the left and right under-threaded $L_v$-moves, and the virtual braid isotopy.
 
 It was proved in~\cite{KL2} that the virtual conjugation, basic $L_v$-moves, left real and left virtual $L_v$-moves, over-threaded $L_v$-moves, and multi-threaded $L_v$-moves follow from the $L$-equivalence. Therefore, these moves also follow from the singular $L_v$-equivalence, and thus we do not need to include them in our $L$-move Markov-type theorem for virtual singular braids, which we are now ready to state and prove.

\begin{theorem}[\textbf{$L$-move Markov-type theorem for virtual singular braids}] \label{Markov L-moves}
Two virtual singular braids have isotopic closures if and only if they are singular $L_v$-equivalent.
\end{theorem}

\begin{proof}
It is easy to see that singular $L_v$-equivalent virtual singular braids have isotopic closures. 

We will now work on the converse. First, we need to show that different choices made in the braiding process result in braids that are singular $L_v$-equivalent. The choices made during the braiding process are the subdivision points and the order of the braiding moves. The subdivision points are needed for marking the braiding boxes and the up-arcs. Using a similar argument as in \cite[Corollary 2]{KL2}, it is not hard to see that given two subdivisions of a virtual singular diagram, the resulting virtual singular braids obtained by our braiding algorithm are singular $L_v$-equivalent.
Due to the narrow condition for the braiding boxes, the braidings of the crossings are local and independent, so the order in which we braid the crossings has no effect on the final output. Moreover, the order in which we braid the free up-arcs is also irrelevant. Due to the braid detour moves, we can in fact braid first the free up-arcs (or just some of them) and then braid the crossings (and any remaining free up-arcs). 

Second, we need to show that, different choices in bringing a virtual singular diagram to general position result in braids (obtained by our braiding algorithm) that are singular $L_v$-equivalent.
Using a similar argument as in~\cite[Lemma 7]{KL2}, it is easily seen that planar isotopy moves applied away from any of the crossings in a virtual singular link diagram result in braids that are related by braid isotopy and the basic $L_v$-move. Indeed, the addition of singular crossings in the setting does not change the situation. It was also shown in~\cite[Lemma 7]{KL2} that, by applying the braiding algorithm to diagrams that differ by a swing move involving a virtual crossing or a real crossing results in braids that are $L$-equivalent. Therefore, for our case of virtual singular link diagrams, it remains to verify the swing moves containing a singular crossing. These swing moves can be verified in the same manner as the swing moves involving a real crossing, by merely replacing the real crossing in~\cite[Figures 26, 27]{KL2}) with a singular crossing. 

Finally, we need to show that two virtual singular braids with isotopic closures are related by singular $L_v$-equivalence. For that, we need to prove that virtual singular link diagrams (in general position) that differ by the extended virtual Reidemeister moves (recall Figure~\ref{fig:isotopies}) correspond to closures of virtual singular braids that are singular $L_v$-equivalent. By~\cite[Theorem 2]{KL2}, we know that the isotopy moves involving only real and virtual crossings follow from the $L$-equivalence for virtual braids (and thus from singular $L_v$-equivalence). Therefore, we only need to consider the extended virtual Reidemeister moves involving singular crossings, and these moves need to be considered with any given orientation of the strands. 

Note that if all strands involved are oriented downward, the statement follows directly from the relations defined on ${VSB}_n$. We consider all cases of each isotopy move involving singular crossings. We consider diagrams that are identical, except in a small region where they differ as shown in the figures; that is, the isotopy move is applied in that small region.

We start with the move $RS1$ and allow one strand to be oriented upward. If we apply the braiding algorithm to the diagrams on both sides of the move (followed by braiding isotopy), the resulting diagrams differ by a left $rs$-threaded $L_v$-move, as explained in Figure~\ref{RS1-case1}. 
Note that if we reverse the orientations of the two strands in the move, the corresponding braids differ by a right $rs$-threaded $L_v$-move, as explained in Figure~\ref{RS1-case2} (besides reversing the orientations of the two strands, we also changed the sign of the classical crossings from positive to negative, for more variety).
\begin{figure}[ht]
\[   \raisebox{-20pt}{\includegraphics[height=.65 in]{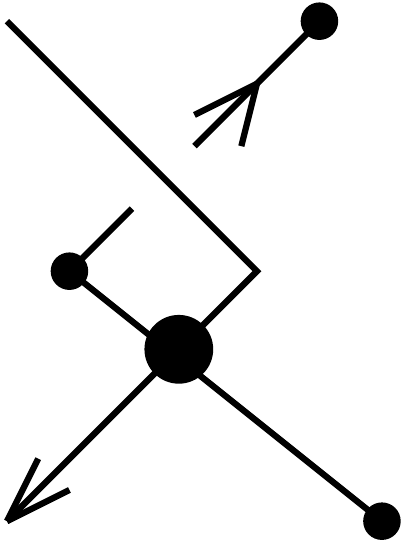}}\,\,\displaystyle \mathop{\longrightarrow}_{\text{braid isotopy}}^{\text{braiding}} \,\,\raisebox{-60pt}{\includegraphics[height=1.8 in]{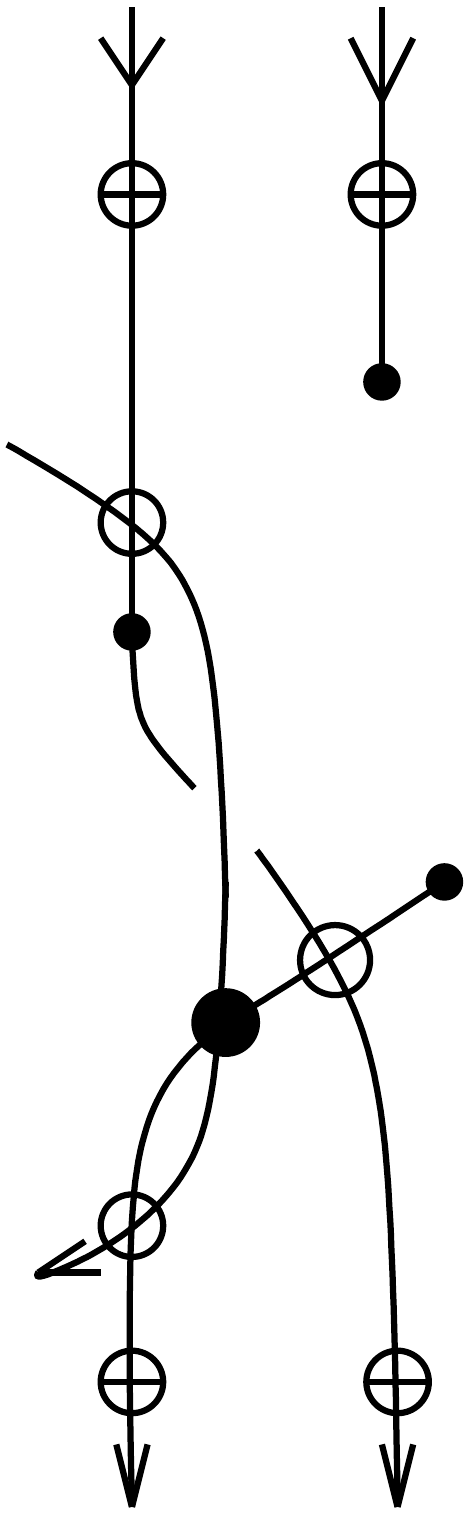}} \stackrel{rs-\text{threaded}}{\stackrel{L_v-\text{move}}{\longleftrightarrow}} \raisebox{-60pt}{\includegraphics[height=1.8 in]{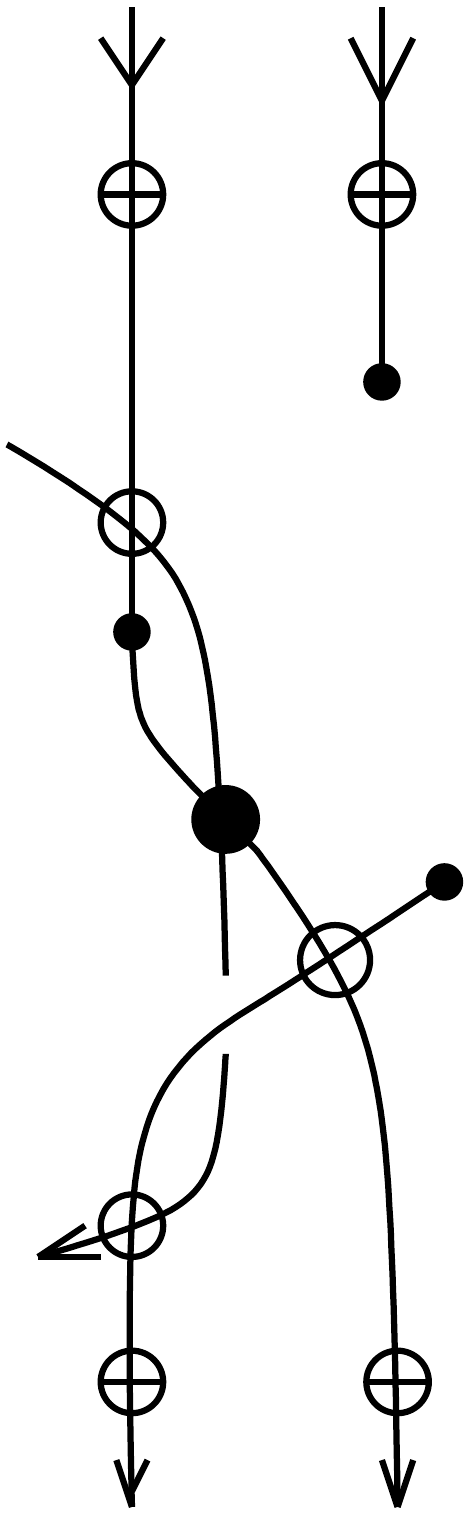}}\,\,\displaystyle \mathop{\longleftarrow}_{\text{braid isotopy}}^{\text{braiding}} \,\,\raisebox{-24pt}{\includegraphics[height=.65 in]{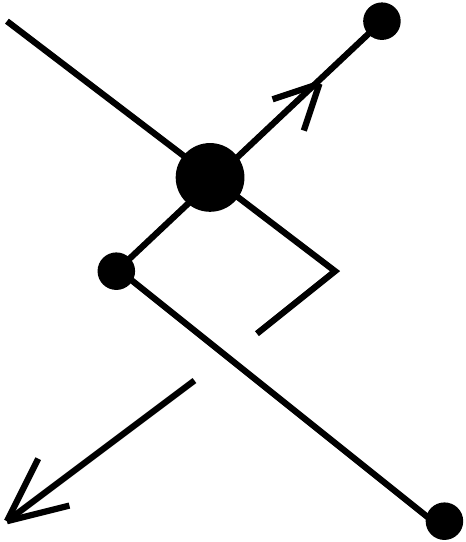}}\,\ \]
\caption{RS1 move--case 1}\label{RS1-case1}
\end{figure}
%%%%%%%%%%%%%%%

\begin{figure}[ht]
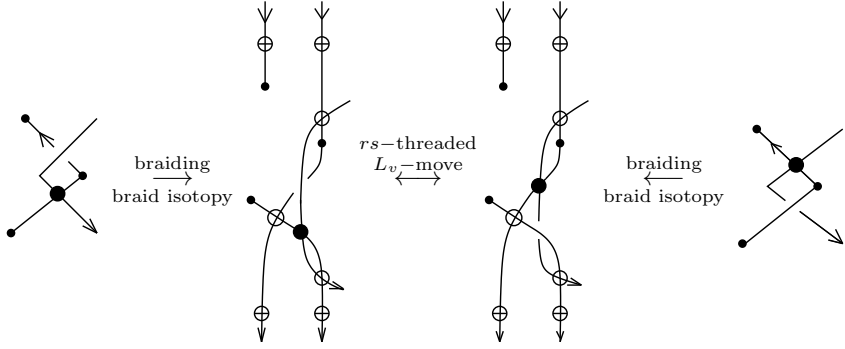

\[   \reflectbox{\raisebox{-20pt}{\includegraphics[height=.65 in]{ba1}}}\,\,\displaystyle \mathop{\longrightarrow}_{\text{braid isotopy}}^{\text{braiding}} \,\, \reflectbox{\raisebox{-60pt}{\includegraphics[height=1.8 in]{ba2}}} \stackrel{rs-\text{threaded}}{\stackrel{L_v-\text{move}}{\longleftrightarrow}} \reflectbox{\raisebox{-60pt}{\includegraphics[height=1.8 in]{ba3}}} \,\,\displaystyle \mathop{\longleftarrow}_{\text{braid isotopy}}^{\text{braiding}} \,\, \reflectbox{\raisebox{-24pt}{\includegraphics[height=.65 in]{ba4}}}\,\ \]
\caption{RS1 move--case 2}\label{RS1-case2}
\end{figure}

%%%%%%%%%%%%%%%%

Figure~\ref{RS1-case3} shows that if we take the isotopy move $RS1$ with both strands  oriented upward and braid the diagrams on each side of the move, we find that the two corresponding braids are related by a series of conjugations and the braid-type $RS1$ move.

\begin{figure}[ht]
\[   \raisebox{-20pt}{\includegraphics[height=.6in]{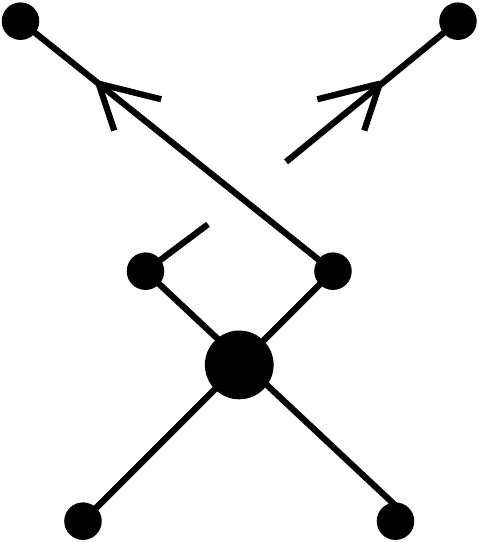}}\,\,\stackrel{\text{braiding}}{\longrightarrow} \,\, \raisebox{-70pt}{\includegraphics[height=1.9in]{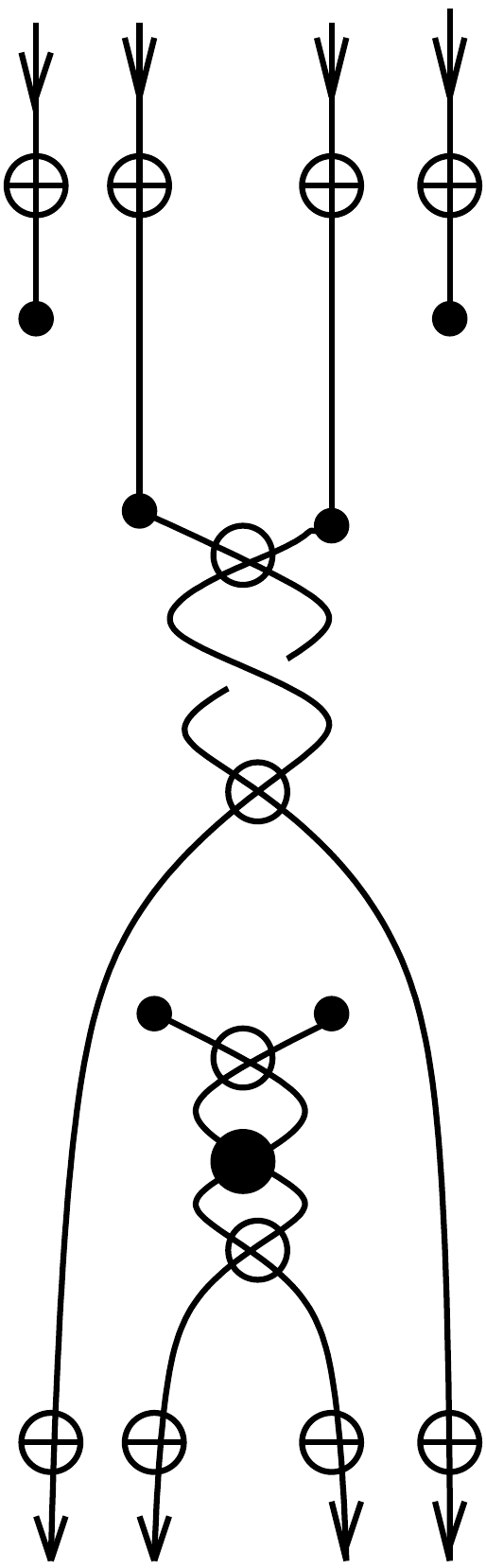}}\,\,\stackrel{\text{virtual}}{\stackrel{\text{conjugation}}{\longleftrightarrow}} \,\, \raisebox{-70pt}{\includegraphics[height=1.9 in]{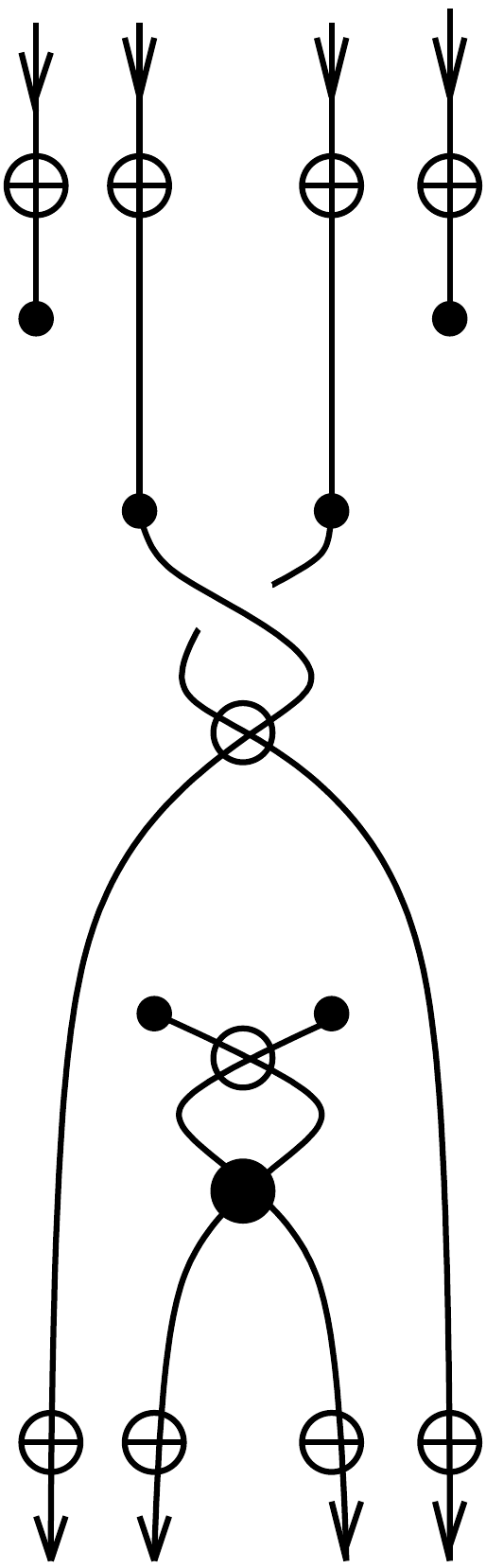}}\,\,\stackrel{\text{singular}}{\stackrel{\text{commuting}}{\longleftrightarrow}}\,\,\raisebox{-70pt}{\includegraphics[height=1.9in]{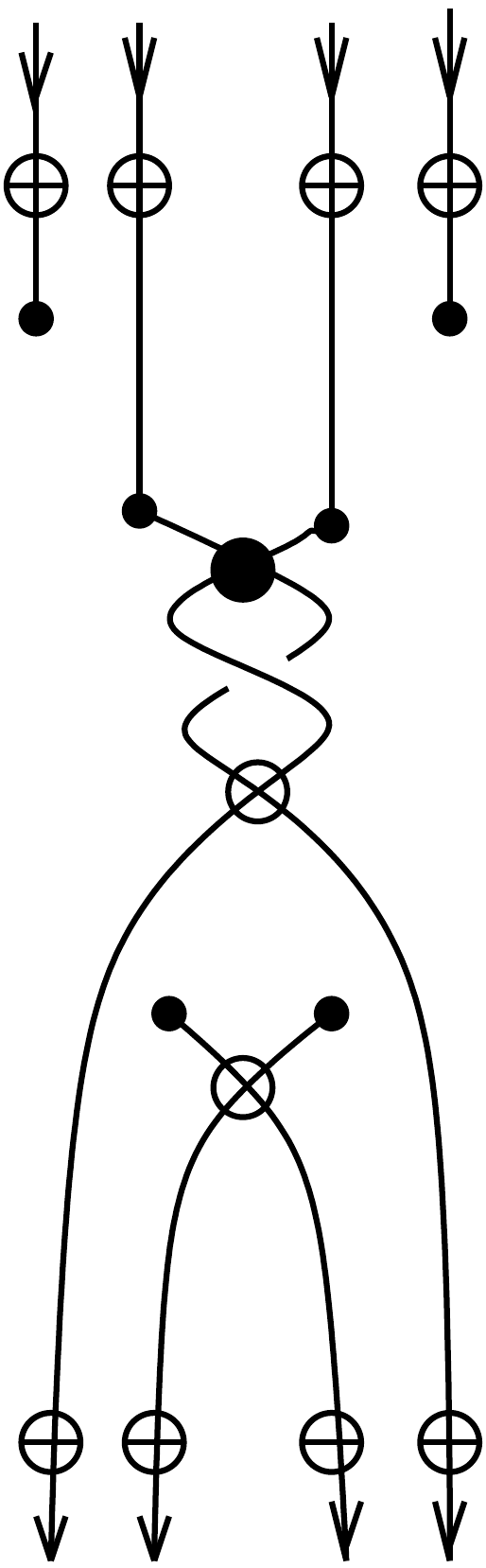}}\]\[\stackrel{\text{braid}}{\stackrel{\text{isotopy}}{\longleftrightarrow}}\,\,\raisebox{-70pt}{\includegraphics[height=1.9in]{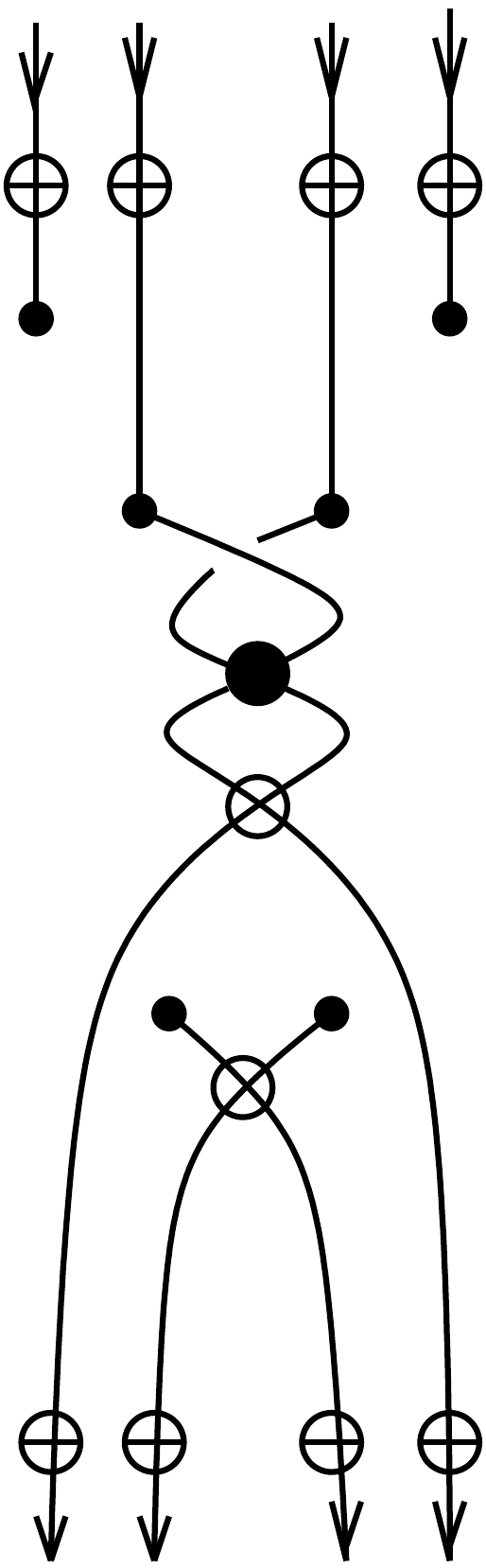}}\,\,\stackrel{\text{real}}{\stackrel{\text{conjugation}}{\longleftrightarrow}} \,\,\raisebox{-70pt}{\includegraphics[height=1.9in]{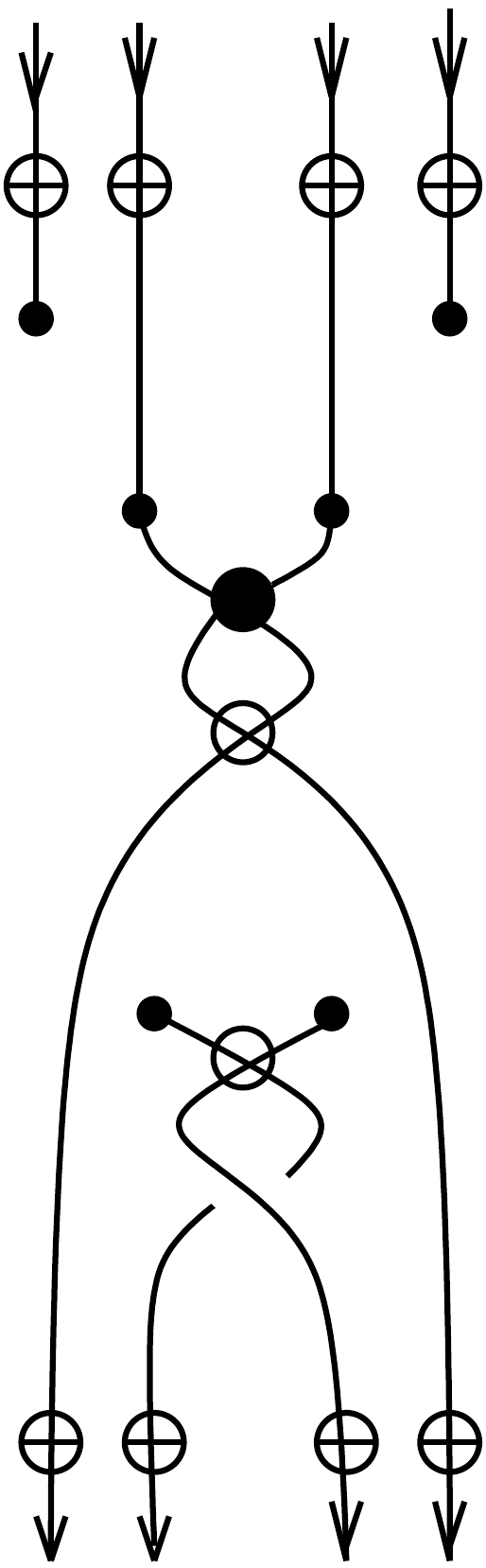}}\,\, \stackrel{\text{virtual}}{\stackrel{\text{conjugation}}{\longleftrightarrow}} \,\, \raisebox{-70pt}{\includegraphics[height=1.9 in]{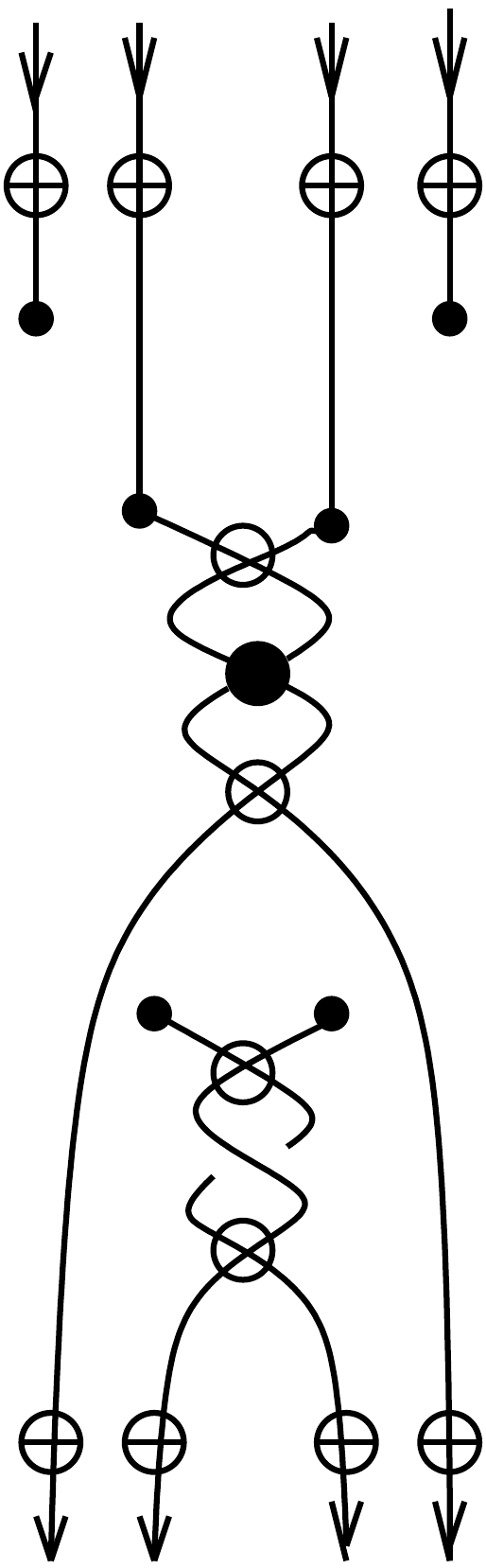}}\,\,\stackrel{\text{braiding}}{\longleftarrow} \raisebox{-20pt}{\includegraphics[height=.6in]{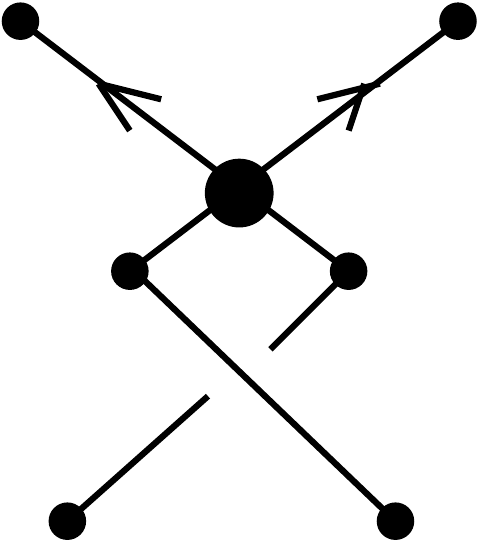}}\]
\caption{RS1 move--case 3}\label{RS1-case3}
\end{figure}

Observe that in Figure~\ref{RS1-case2} we used virtual conjugation, which is not a move of the singular $L_v$-equivalence. However, recall that the virtual conjugation follows from the $L$-equivalence and hence from the singular $L_v$-equivalence (see~\cite[Figures 17, 18]{KL2}).

We will now take a slightly different approach to prove that the moves $RS3$ and $VR3$ hold with any possible orientations on the strands. Again, the case where all strands are oriented downward follow from braid equivalence. We start by considering the $RS3$ move with one strand oriented upward and apply an $R2$ move to create locally three downward oriented strands. After a couple of $RS1$ moves, we apply an $RS3$ move in braid form (see Figure~\ref{RS3-case1}). 

\begin{figure}[ht]
\[   \raisebox{-20pt}{\includegraphics[height=.6in]{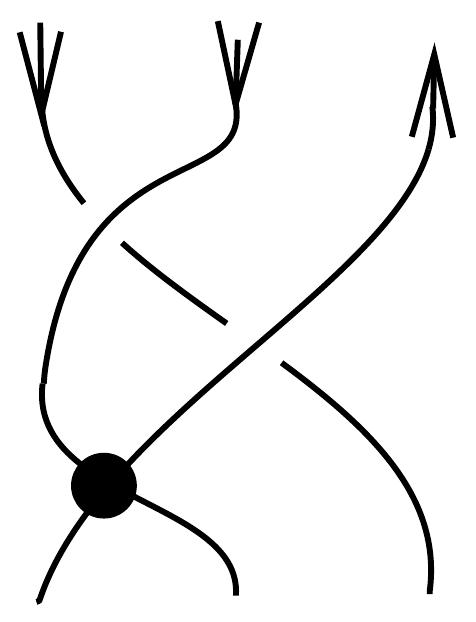}}\,\,\stackrel{R2}{\longleftrightarrow} \,\,\raisebox{-20pt}{\includegraphics[height=.6in]{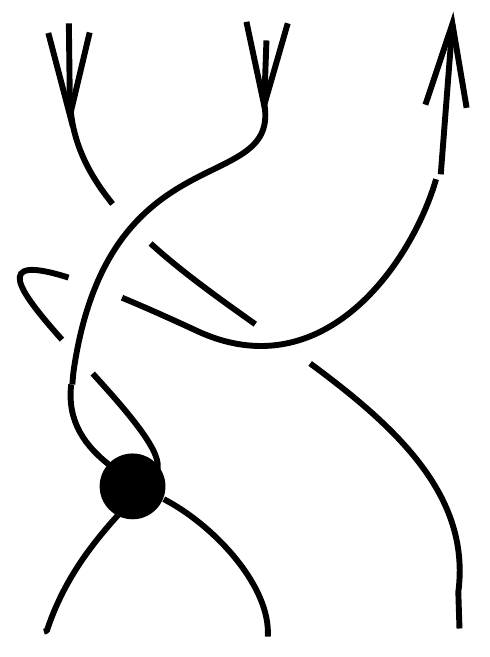}}\,\, \stackrel{RS1}{\longrightarrow} \,\, \raisebox{-20pt}{\includegraphics[height=.6in]{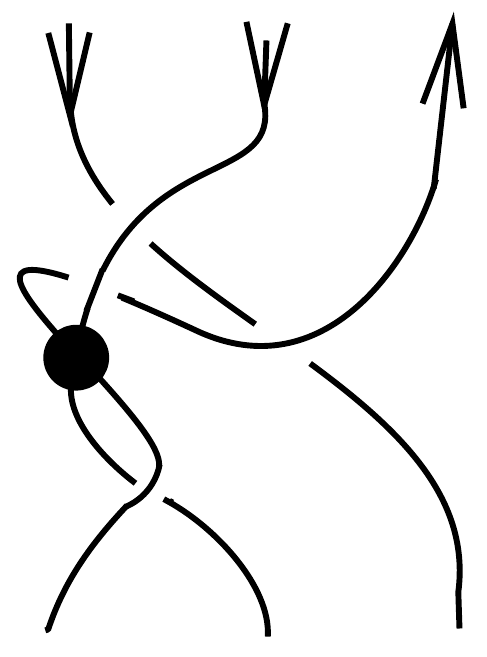}}\,\,\stackrel{RS1}{\longleftrightarrow}\,\,\raisebox{-20pt}{\includegraphics[height=.6in]{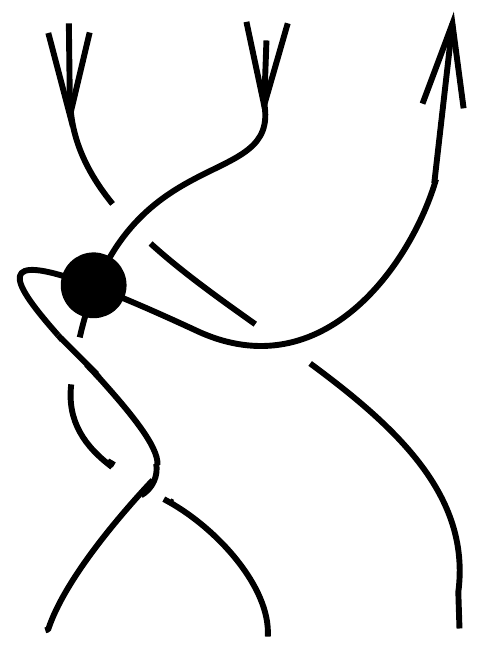}}\]\[\stackrel{R2}{\longleftrightarrow}\,\,\raisebox{-20pt}{\includegraphics[height=.6in]{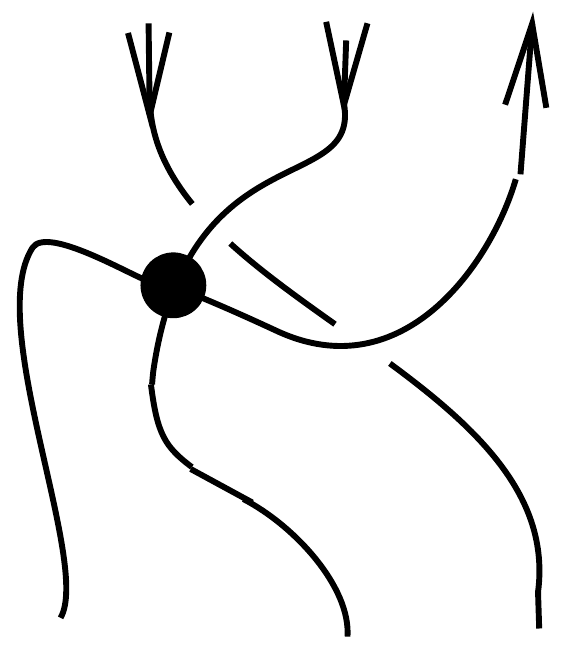}}\,\, \stackrel{\text{braid}}{\stackrel{RS3}{\longleftrightarrow}} \,\,\raisebox{-20pt}{\includegraphics[height=.6in]{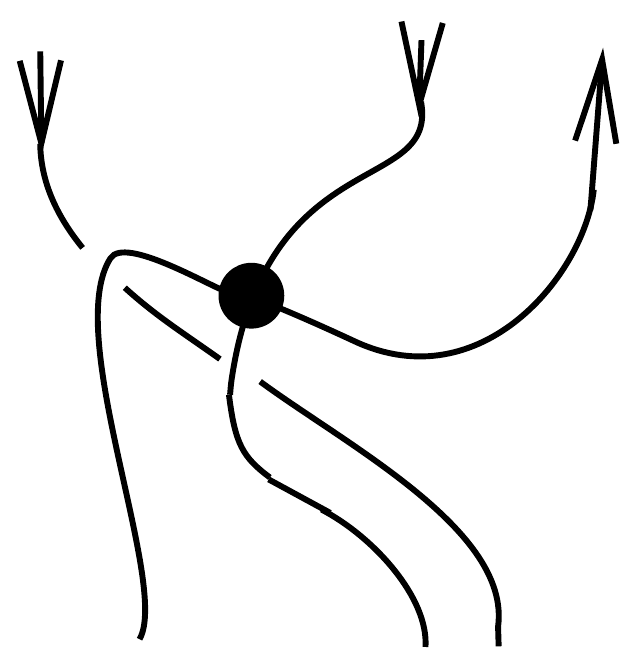}}\,\,\stackrel{\text{swing}}{\stackrel{\text{move}} {\longleftrightarrow}} \,\, \raisebox{-20pt}{\includegraphics[height=.6in]{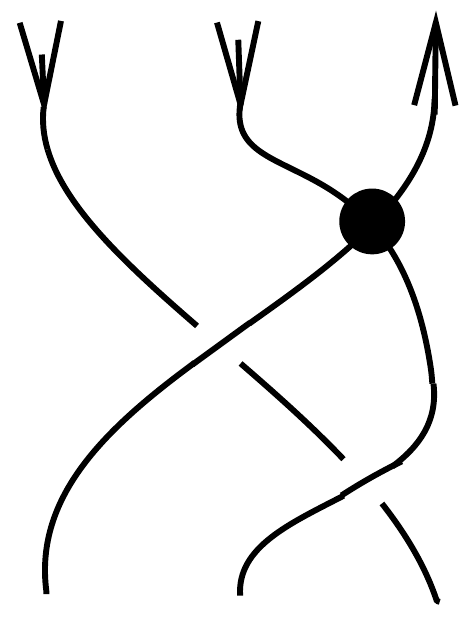}}\]
\caption{RS3 move--case 1}\label{RS3-case1}
\end{figure}

Similarly, if we start with two strands oriented upward, we apply an $R2$ move to reduce to the case with one strand oriented upward, as exemplified in Figure~\ref{RS3-case2}. 

\begin{figure}[ht]
\[   \raisebox{-20pt}{\includegraphics[height=.6in]{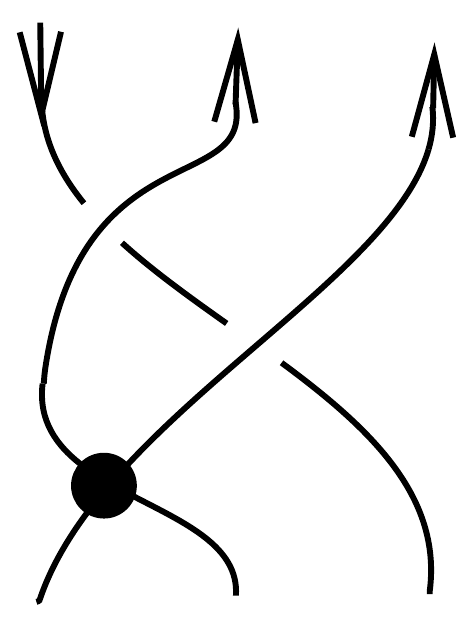}}\,\,\stackrel{\text{R2}}{\longleftrightarrow} \,\,\raisebox{-20pt}{\includegraphics[height=.6in]{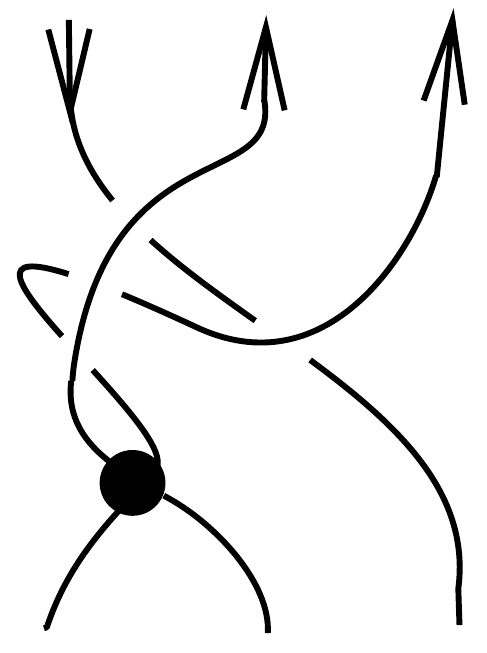}}\,\, \stackrel{\text{RS1}}{\longleftrightarrow} \,\, \raisebox{-20pt}{\includegraphics[height=.6in]{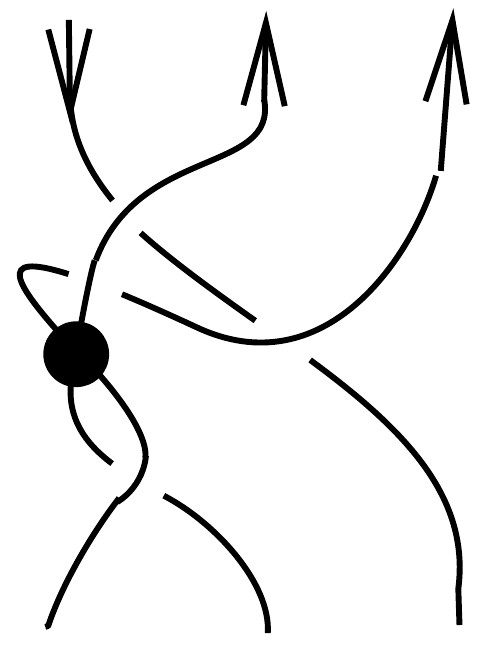}}\,\,\stackrel{\text{RS1}}{\longleftrightarrow}\,\,\raisebox{-20pt}{\includegraphics[height=.6in]{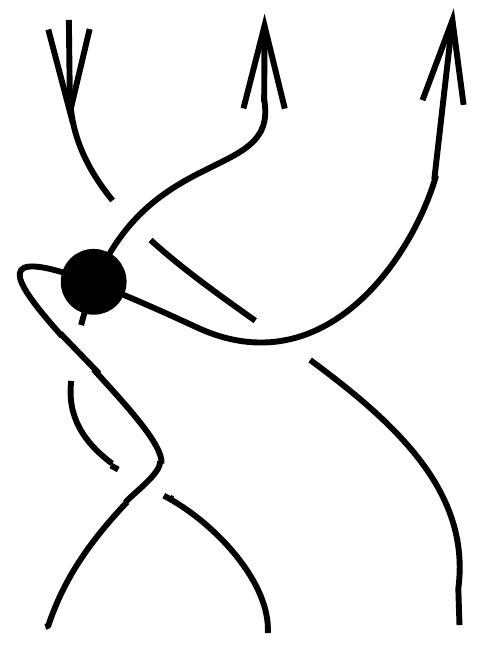}}\]\[\stackrel{\text{R2}}{\longleftrightarrow}\,\,\raisebox{-20pt}{\includegraphics[height=.6in]{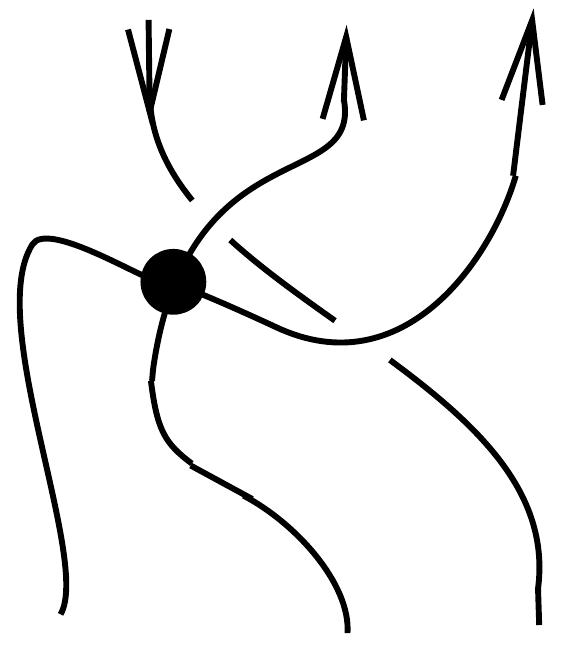}}\,\,\stackrel{\text{RS3 - case 1}}{\longleftrightarrow} \,\,\raisebox{-20pt}{\includegraphics[height=.6in]{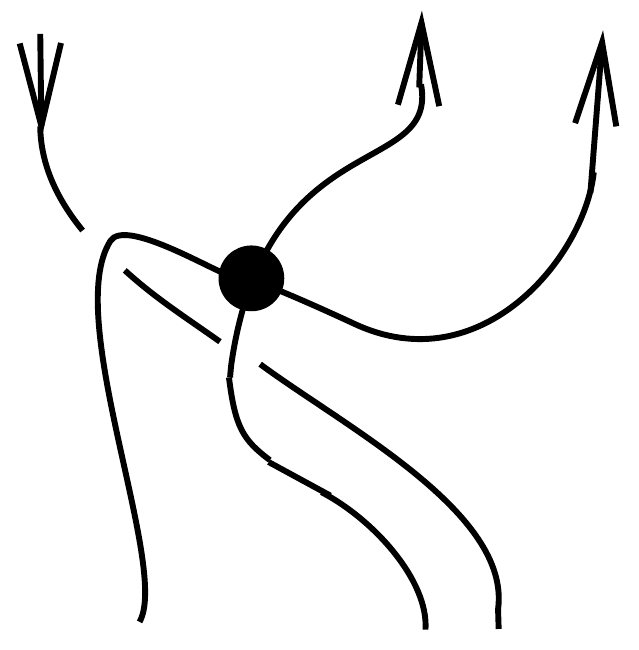}}\,\, \stackrel{\text{swing}}{\stackrel{\text{move}}{\longleftrightarrow}} \,\, \raisebox{-20pt}{\includegraphics[height=.6in]{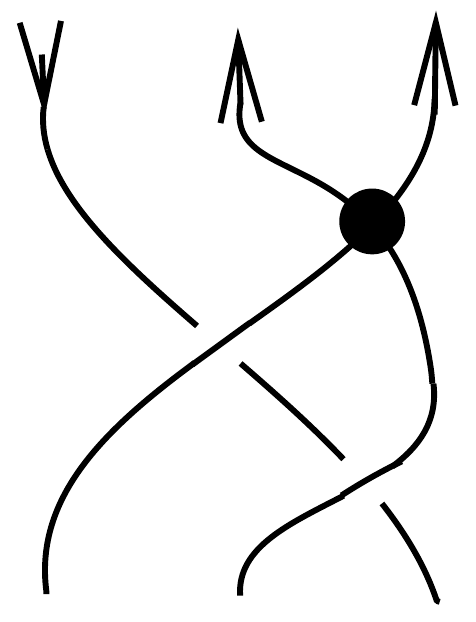}}\]
\caption{RS3 move--case 2}\label{RS3-case2}
\end{figure}

In Figure~\ref{RS3-case3} we consider an $RS3$ move with all three strands oriented upward and show that it can be reduced to the previous case with two strands oriented upward. 

\begin{figure}[ht]
\[   \raisebox{-20pt}{\includegraphics[height=.6in]{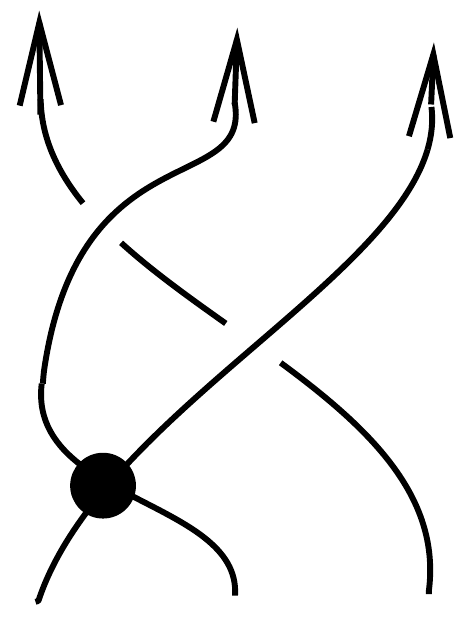}}\,\,\stackrel{\text{R2}}{\longleftrightarrow} \,\,\raisebox{-20pt}{\includegraphics[height=.6in]{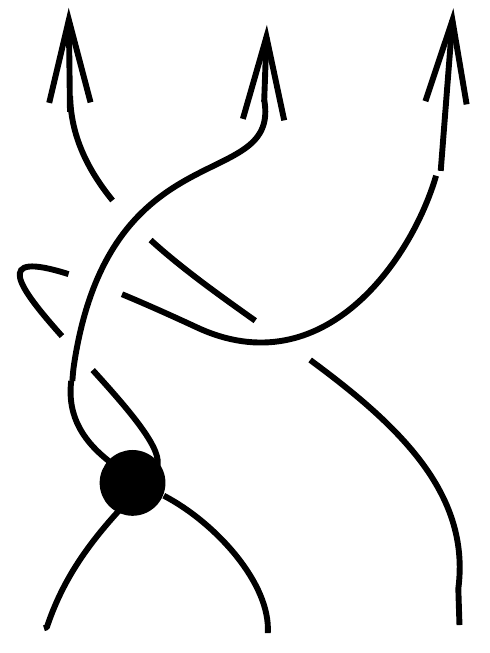}}\,\, \stackrel{\text{RS1}}{\longleftrightarrow} \,\, \raisebox{-20pt}{\includegraphics[height=.6in]{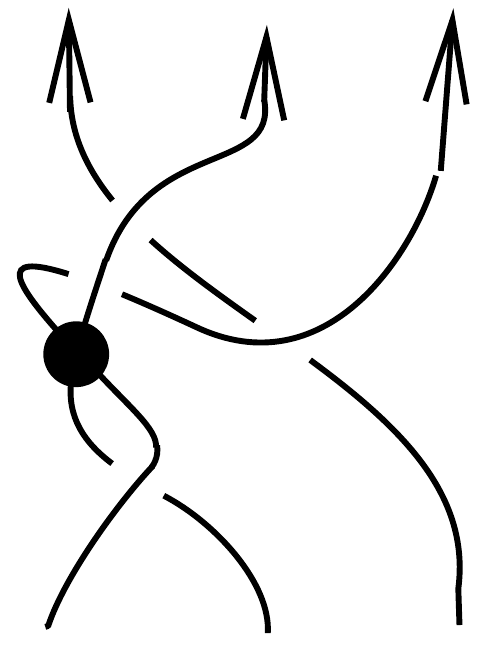}}\,\,\stackrel{\text{RS1}}{\longleftrightarrow}\,\,\raisebox{-20pt}{\includegraphics[height=.6in]{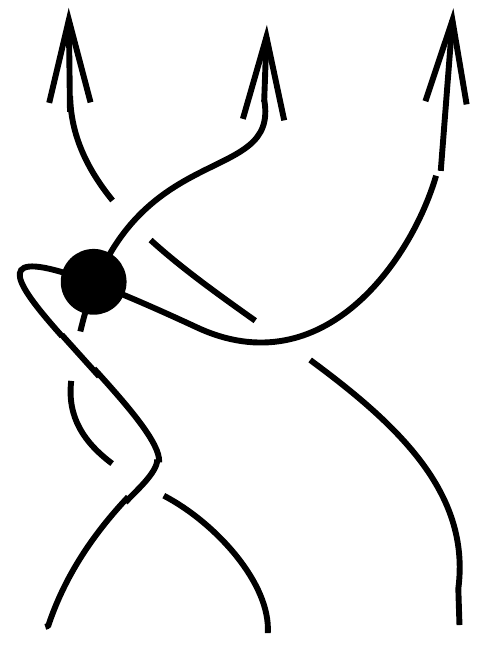}}\]\[\stackrel{\text{R2}}{\longleftrightarrow}\,\,\raisebox{-20pt}{\includegraphics[height=.6in]{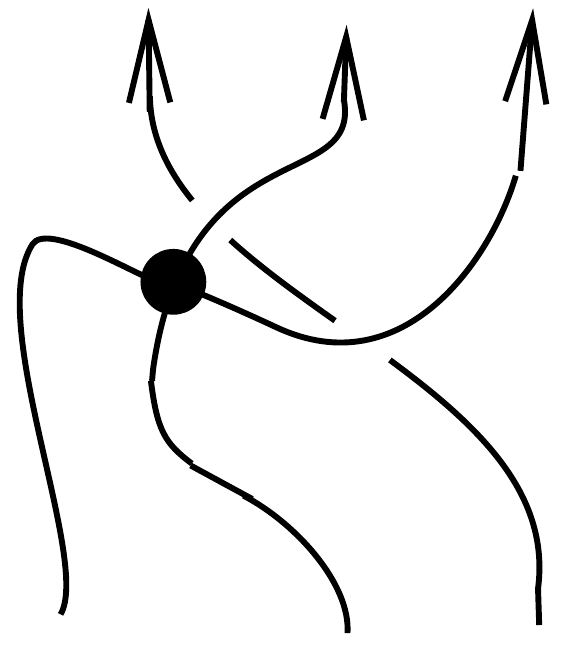}}\,\,\stackrel{\text{RS3 - case 2}}{\longleftrightarrow} \,\,\raisebox{-20pt}{\includegraphics[height=.6in]{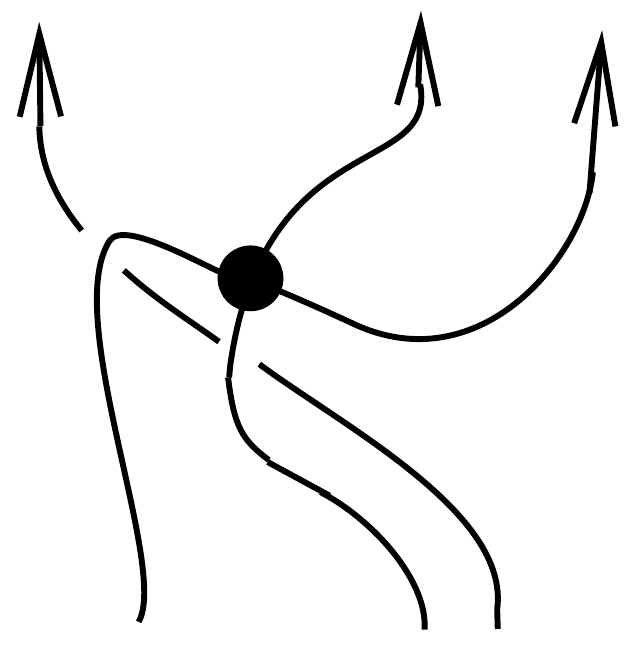}}\,\, \stackrel{\text{swing}}{\stackrel{\text{move}}{\longleftrightarrow}} \,\, \raisebox{-20pt}{\includegraphics[height=.6in]{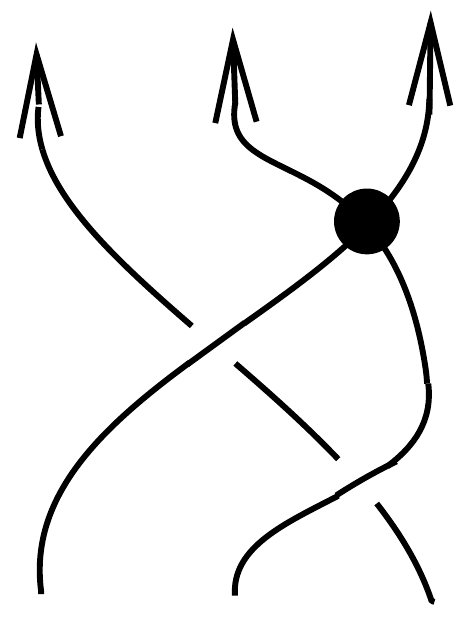}}\]
\caption{RS3 move--case 3}\label{RS3-case3}
\end{figure}

The proof for the $VR3$ moves with various orientations on the strands is done similarly as for the $RS3$ moves, and therefore they are omitted to avoid repetition. This completes the proof.
\end{proof}

In the following theorem we will use $\omega$ to represent an arbitrary virtual singular braid in ${VSB}_n$. We also regard $\omega$ as an element of ${VSB}_{n+1}$ by adding a strand on the right of $\omega$. (We will not use an extra notation when we regard $\omega \in {VSB}_n$ as an element in ${VSB}_{n+1}$.)  Using this operation (of adding a single identity strand on the right of a braid) the monoid  ${VSB}_n$ embeds in ${VSB}_{n+1}$, and we define $VSB_{\infty} := \cup_{n=1}^\infty VSB_n$. In what follows, we also allow adding an identity strand at the left of $\omega \in {VSB}_n$ and we denote by $i(\omega)$ the braid in ${VSB}_{n+1}$ obtained in this way.

%%%%%%%%%%%%%%%%%%%%%%%%%%%%%%%%%%%%%%%%%%%%%%%%%%%
%Algebraic Markov-type theorem
\begin{theorem}[\textbf{Algebraic Markov-type theorem for virtual singular braids}]\label{Markov alg}
Two virtual singular braids have isotopic closures if and only if they differ by a finite sequence of braid relations in ${VSB}_{\infty}$ together with the following moves or their inverses:
\begin{enumerate}
\item[(i)] Real and virtual conjugation, and singular commuting  (see Figure~\ref{fig:conj}):		
\[\sigma_i\omega\sim\omega\sigma_i, \,\,\, \tau_i\omega\sim\omega\tau_i, \,\,\, v_i\omega\sim\omega v_i\]
\item[(ii)] Right real and right virtual stabilization (see Figure~\ref{fig:stab})	:			
\[\omega v_n\sim\omega\sim\omega\sigma_n^{\pm1}\]
\item[(iii)] Right and left algebraic under-threading (see Figure~\ref{fig:thread}):		
\[\omega\sim\omega\sigma_n^{-1}v_{n-1}\sigma_n, \,\,\, \omega\sim i(\omega)\sigma_1v_2\sigma_1^{-1}\]
\item[(iv)] Right and left algebraic $rs$-threading (see Figure~\ref{fig:rvthread}):
\[\omega\tau_n v_{n-1}\sigma_n^{\pm1}\sim\omega\sigma_n^{\pm1}v_{n-1}\tau_n, \,\,\, i (\omega)\tau_1v_{2}\sigma_1^{\pm1}\sim i(\omega)\sigma_1^{\pm1} v_{2}\tau_1\]
\end{enumerate}
where $\omega, v_i,\sigma_i^{\pm1},\tau_ i \in {VSB}_{n}$ and $v_n,\sigma_n^{\pm1},\tau_n \in {VSB}_{n+1}$.
\end{theorem}

\begin{figure}[ht]
\[ \raisebox{-35pt}{\includegraphics[height=1in]{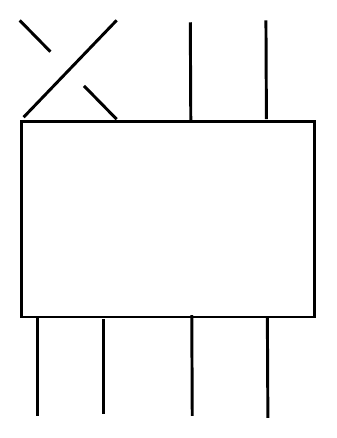}}\,\, \sim \,\, \raisebox{-35pt}{\includegraphics[height=1 in]{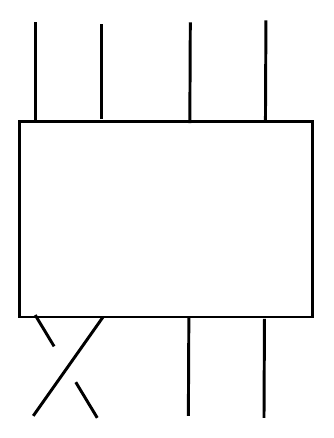}}\hspace{1cm}
 \raisebox{-35pt}{\includegraphics[height=1 in]{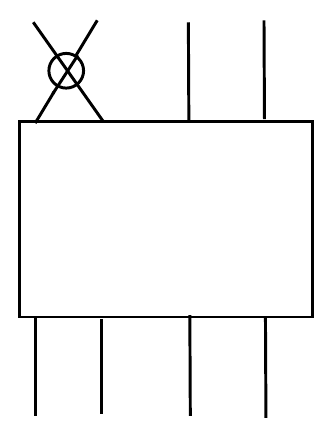}}\,\,\sim \,\, \raisebox{-35pt}{\includegraphics[height=1in]{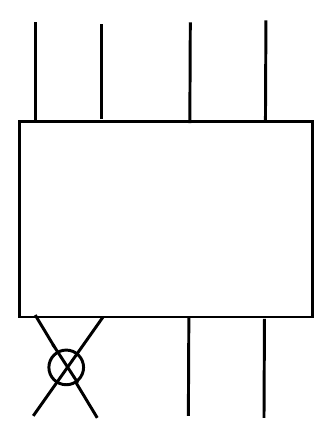}}
 \put(-265, 0){\fontsize{9}{9}$\omega$}
 \put(-190, 0){\fontsize{9}{9}$\omega$}
  \put(-105, 0){\fontsize{9}{9}$\omega$}
   \put(-30, 0){\fontsize{9}{9}$\omega$} \]
 \[
 \raisebox{-35pt}{\includegraphics[height=1in]{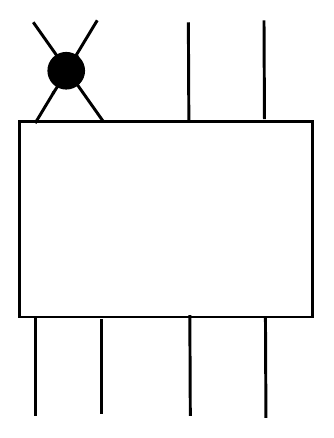}}\,\,\sim \,\, \raisebox{-35pt}{\includegraphics[height=1 in]{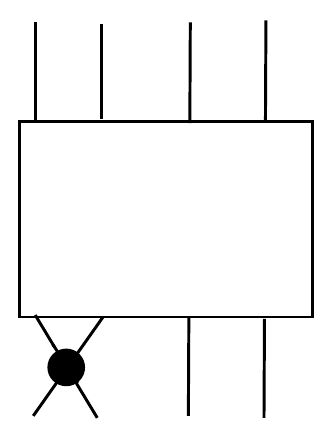}}
   \put(-105, 0){\fontsize{9}{9}$\omega$}
  \put(-30, 0){\fontsize{9}{9}$\omega$}
\]
\caption{Real and virtual conjugation, and singular commuting}\label{fig:conj} 
\end{figure}

\begin{figure}[ht]
\[ \raisebox{-35pt}{\includegraphics[height=1in]{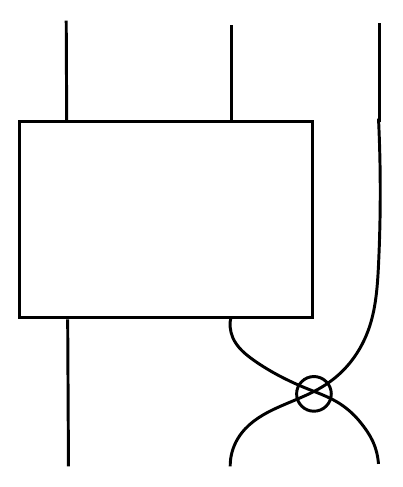}}\,\, \sim \,\, \raisebox{-35pt}{\includegraphics[height=1in]{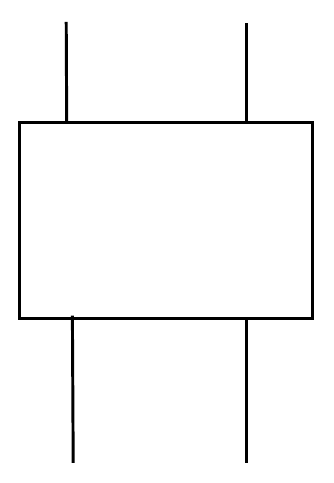}}\sim
 \raisebox{-35pt}{\includegraphics[height=1 in]{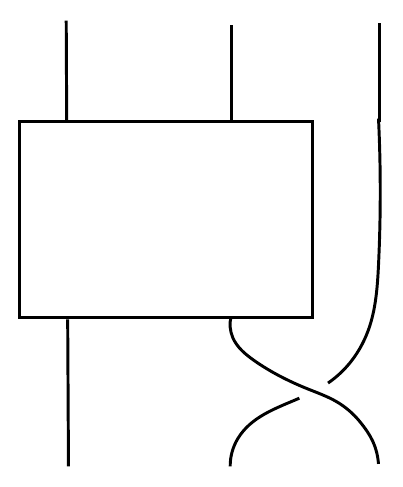}}
 \put(-180, 3){\fontsize{9}{9}$\omega$}
  \put(-100, 3){\fontsize{9}{9}$\omega$}
   \put(-38, 3){\fontsize{9}{9}$\omega$}\]
\caption{Right real and virtual stabilization}\label{fig:stab} 
\end{figure} 

\begin{figure}[ht]
\[ \raisebox{-35pt}{\includegraphics[height=1in]{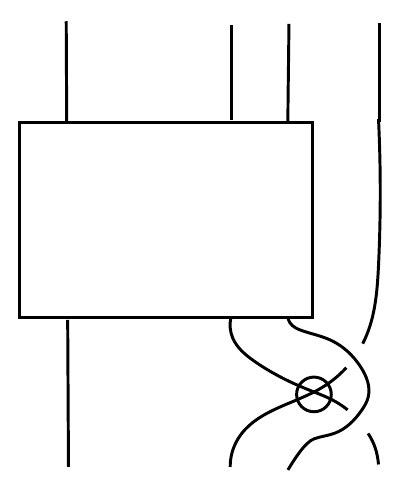}}\,\, \sim \,\, \raisebox{-35pt}{\includegraphics[height=1in]{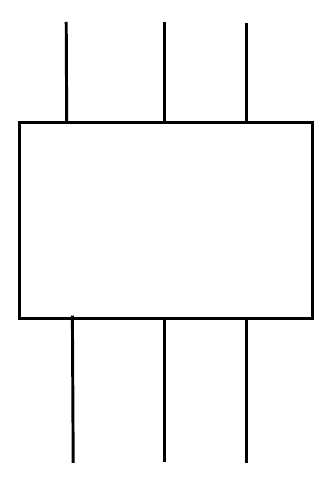}} \,\, \sim \,\, \reflectbox{\raisebox{-35pt}{\includegraphics[height=1in]{th1}}}
 \put(-187, 3){\fontsize{9}{9}$\omega$}
  \put(-107, 3){\fontsize{9}{9}$\omega$}
   \put(-27, 3){\fontsize{9}{9}$\omega$}\]
\caption{Right and left algebraic under-threading}\label{fig:thread} 
\end{figure} 

\begin{figure}[ht]
\[\raisebox{-35pt}{\includegraphics[height=1in]{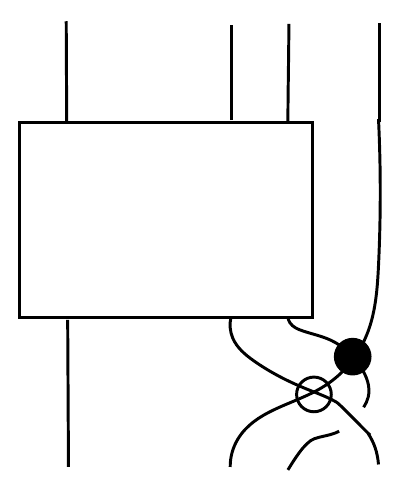}} \,\, \sim \,\, \raisebox{-35pt}{\includegraphics[height=1in]{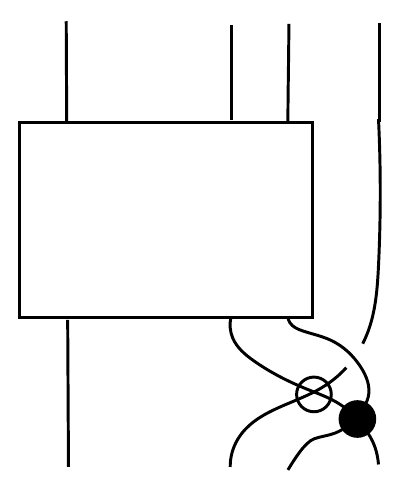}}
 \put(-117, 3){\fontsize{9}{9}$\omega$}
  \put(-37, 3){\fontsize{9}{9}$\omega$} \hspace{1cm}
\reflectbox{\raisebox{-35pt}{\includegraphics[height=1in]{rvth2}}}\,\, \sim \,\, \reflectbox{\raisebox{-35pt}{\includegraphics[height=1in]{rvth1}}}
 \put(-117, 3){\fontsize{9}{9}$\omega$}
  \put(-37, 3){\fontsize{9}{9}$\omega$}\]

\caption{Right and left algebraic $rs$-threading}\label{fig:rvthread} 
\end{figure} 

\begin{proof}
It is easily checked that the closures of two virtual singular braids that are related by virtual singular braid isotopy and a finite sequence of the moves listed in Theorem~\ref{Markov alg} represent isotopic virtual singular links.

For the converse, let $\beta_1$ and $\beta_2$ be virtual singular braids whose closures represent isotopic virtual singular links. By Theorem~\ref{Markov L-moves}, we know that $\beta_1$ and $\beta_2$ are singular $L_v$-equivalent. Therefore, it suffices to show that the four types of moves in Theorem~\ref{Markov alg} follow from the singular $L_v$-equivalence. Clearly, the real, singular, and virtual conjugation follow from the singular $L_v$-equivalence since the first two are part of the singular $L_v$-equivalence and the latter is a consequence of the singular $L_v$-equivalence (as explained in the paragraph before Theorem~\ref{Markov L-moves}).

Right real and right virtual stabilization (the moves in (ii)) follow from right real and right virtual $L_v$-moves, respectively, plus braid detouring and virtual conjugation in $VSB_{\infty}$. Figure~\ref{fig:stabproof} explains the case of the right virtual stabilization; the right real stabilization follows similarly. 

\begin{figure}[ht]
\[ \raisebox{-25pt}{\includegraphics[height=.74in]{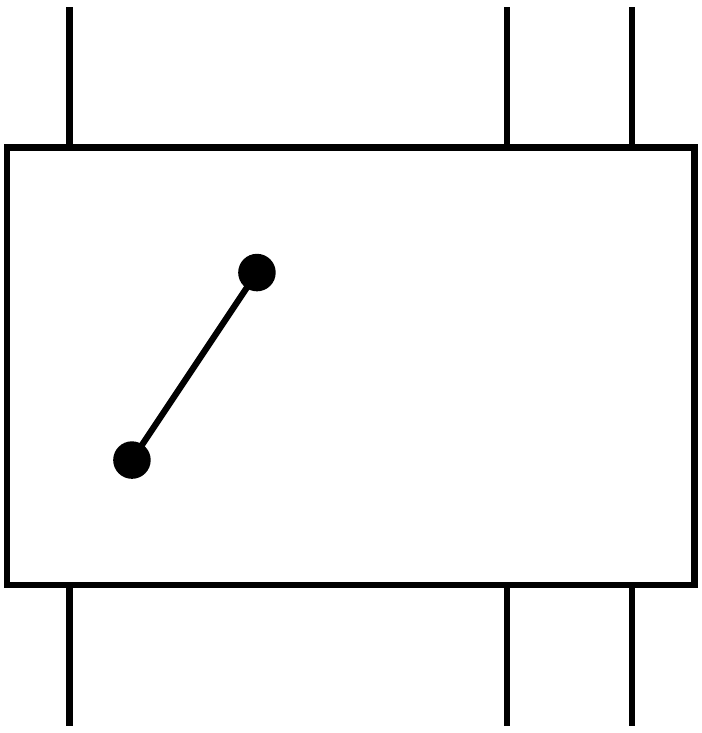}} \,\, \stackrel{\text{right}}{\stackrel{vL_v-move}{\longleftrightarrow}} \,\,\raisebox{-35pt}{\includegraphics[height=1in]{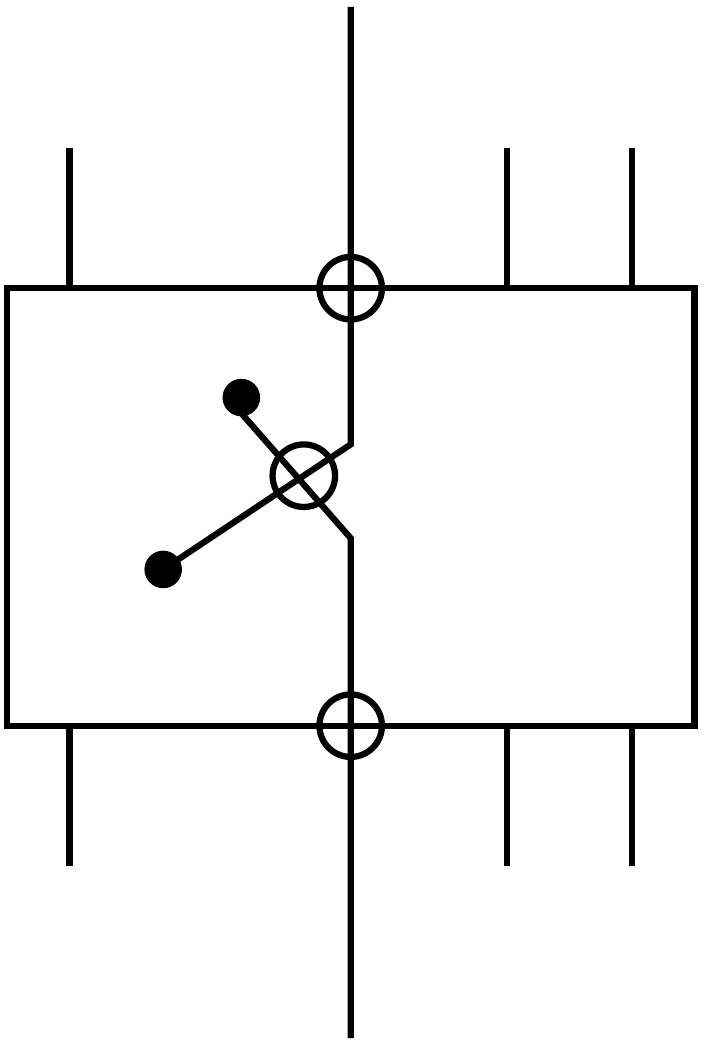}}\,\, \stackrel{\text{braid}}{\stackrel{\text{detours}}{\longleftrightarrow}} \,\, \raisebox{-30pt}{\includegraphics[height=0.92in]{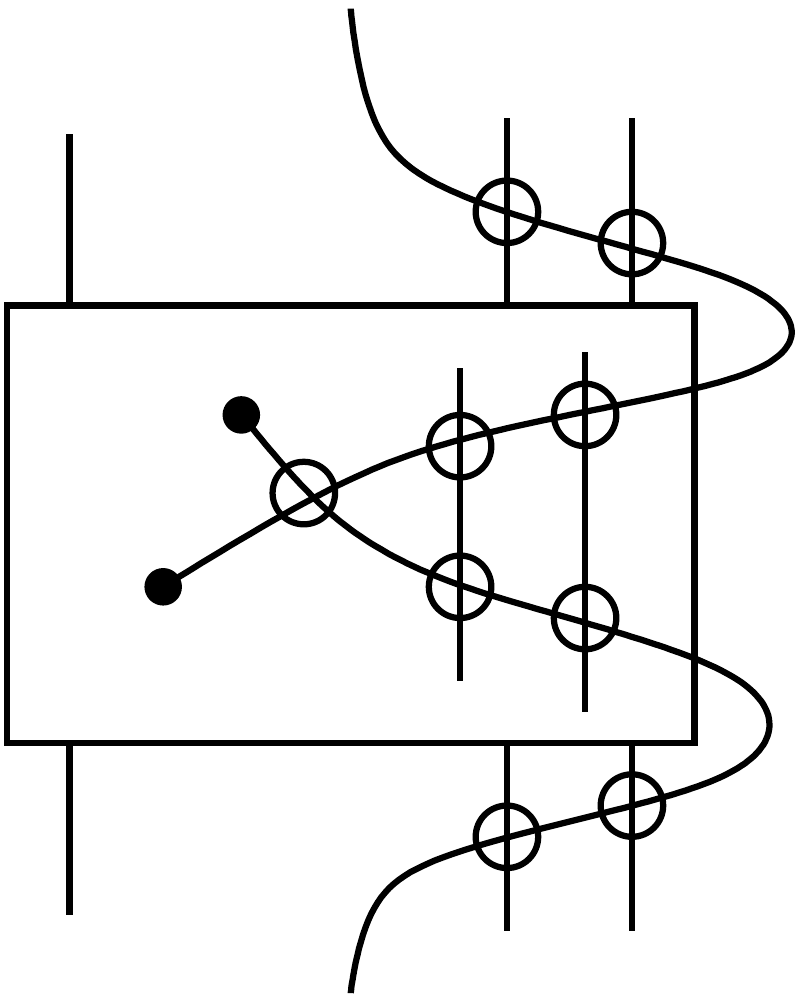}}\,\, \stackrel{\text{virtual}}{\stackrel{\text{conjug.}}{\longleftrightarrow}} \]
\[ \raisebox{-27pt}{\includegraphics[height=0.8in]{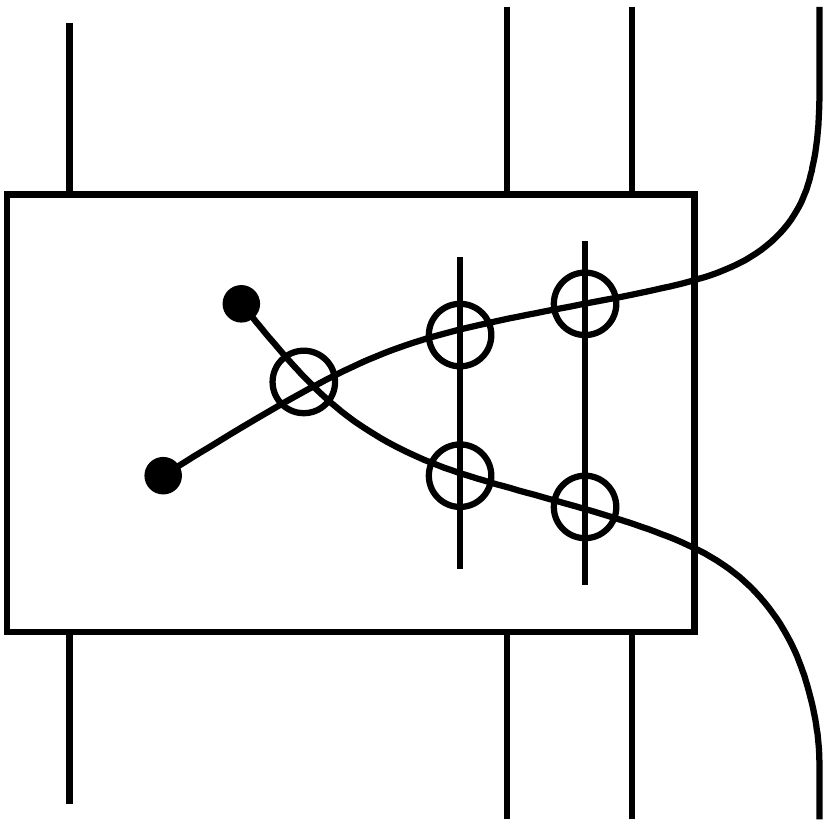}}\,\, \stackrel{\text{detour}}{\stackrel{\text{threads}}{\longleftrightarrow}}\,\,\raisebox{-27pt}{\includegraphics[height=.8in]{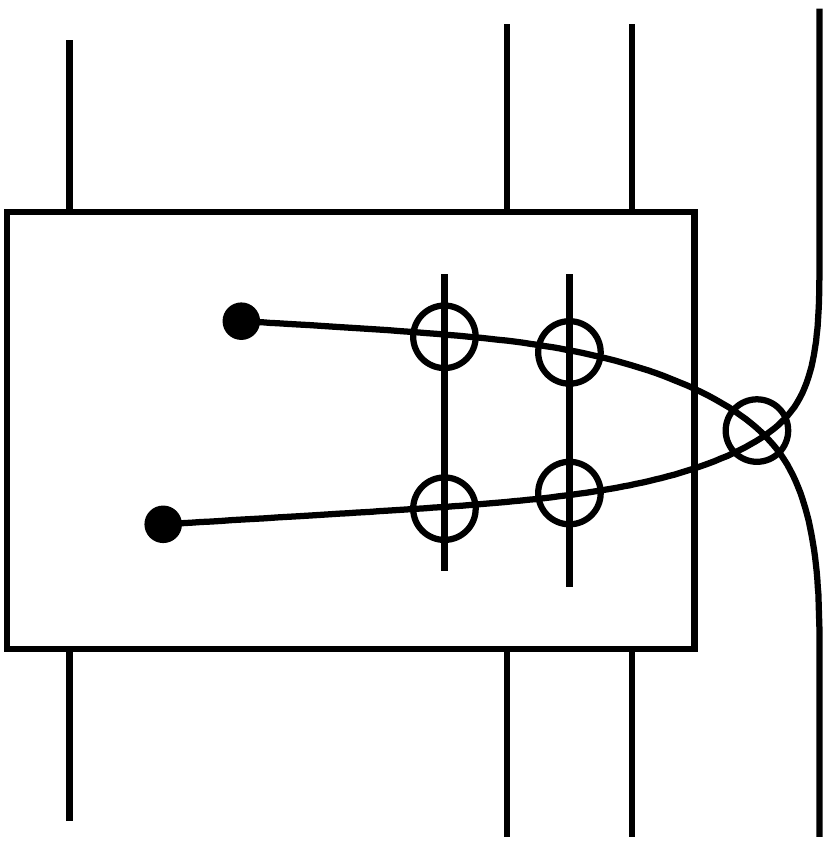}} \,\, \stackrel{\text{virtual}}{\stackrel{\text{conjug.}}{\longleftrightarrow}}\,\,\raisebox{-30pt}{\includegraphics[height=.77in]{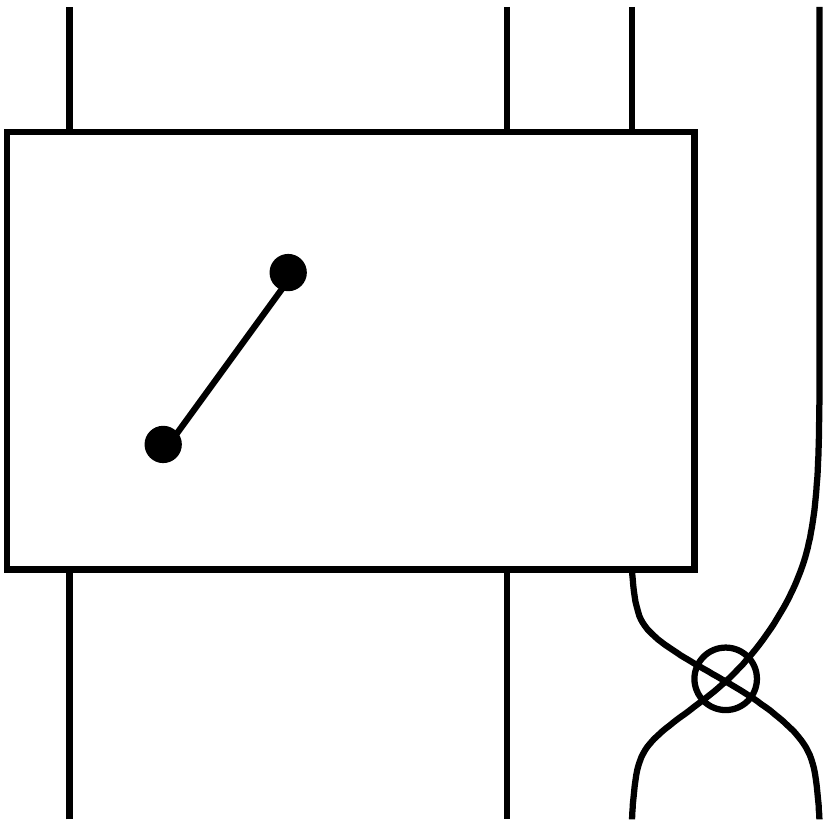}} \]
\caption{Right virtual stabilization is obtained from the right $vL_v$-move, braid detour, and virtual conjugation}\label{fig:stabproof}
\end{figure}

We note that in the last step of Figure~\ref{fig:stabproof}, the virtual conjugation is applied in the smaller braid that contains the threads which cross virtually the pair of braid strands created during the right $vL_v$-move.

The right and left algebraic under-threading (the moves in (iii)) follow from the right and, respectively, the left under-threading $L_v$-moves, braid detour, and virtual conjugation. Figure~\ref{fig:threadingproof} treats the left algebraic under-threading; the right algebraic under-threading is verified in a similar fashion, and therefore is omitted here.

\begin{figure}[ht]
\[ \reflectbox{\raisebox{-25pt}{\includegraphics[height=.74in]{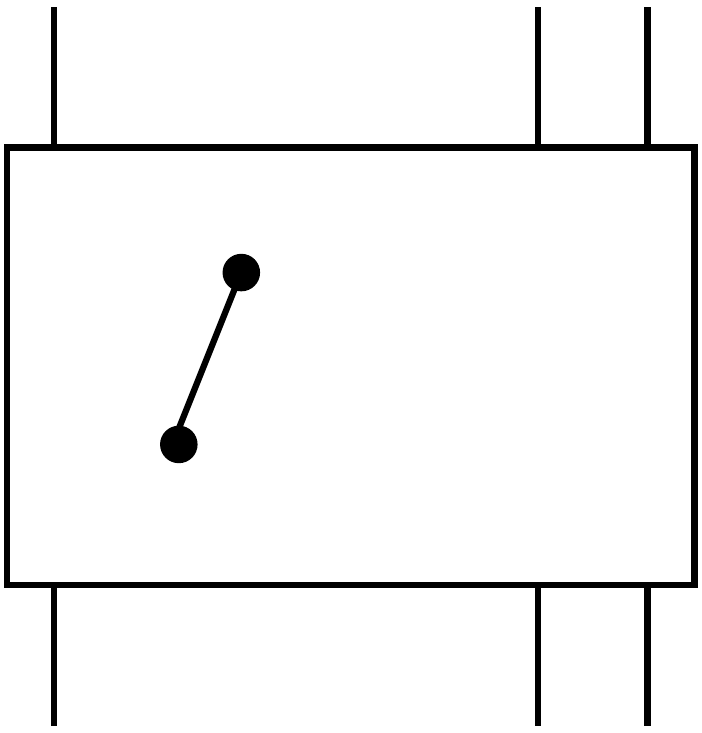}}} \,\,\stackrel{\text{under-threaded}}{\stackrel{L_v -move}{\longleftrightarrow}} \,\, \reflectbox{\raisebox{-27pt}{\includegraphics[height=.8in]{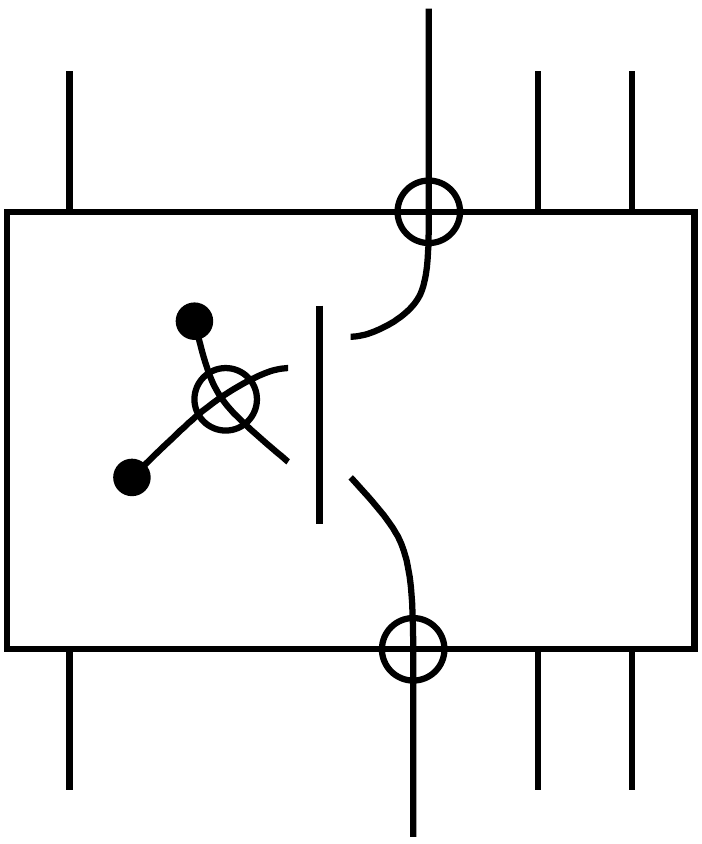}}} \,\, \stackrel{\text{braid}}{\stackrel{\text{detours}}{\longleftrightarrow}} \,\, \reflectbox{\raisebox{-30pt}{\includegraphics[height=.87in]{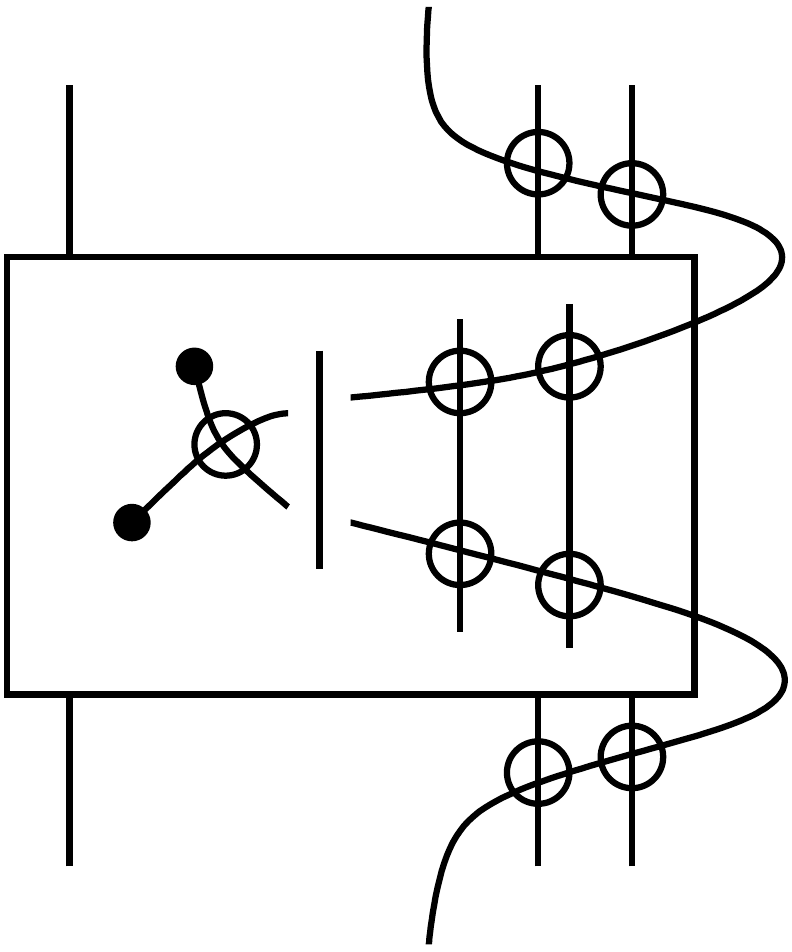}}} \,\, \stackrel{\text{virtual}}{\stackrel{\text{conjug.}}{\longleftrightarrow}} \]
\[ \reflectbox{\raisebox{-25pt}{\includegraphics[height=0.72in]{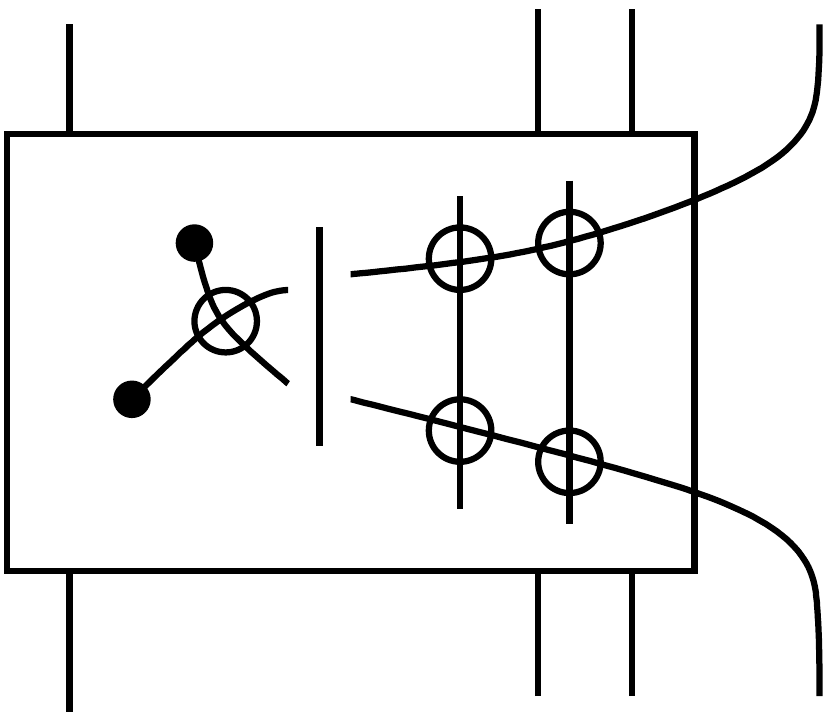}}} \,\, \stackrel{\text{detour}}{\stackrel{\text{virt. threads}} {\longleftrightarrow}}\,  \, \reflectbox{\raisebox{-25pt}{\includegraphics[height=0.72in]{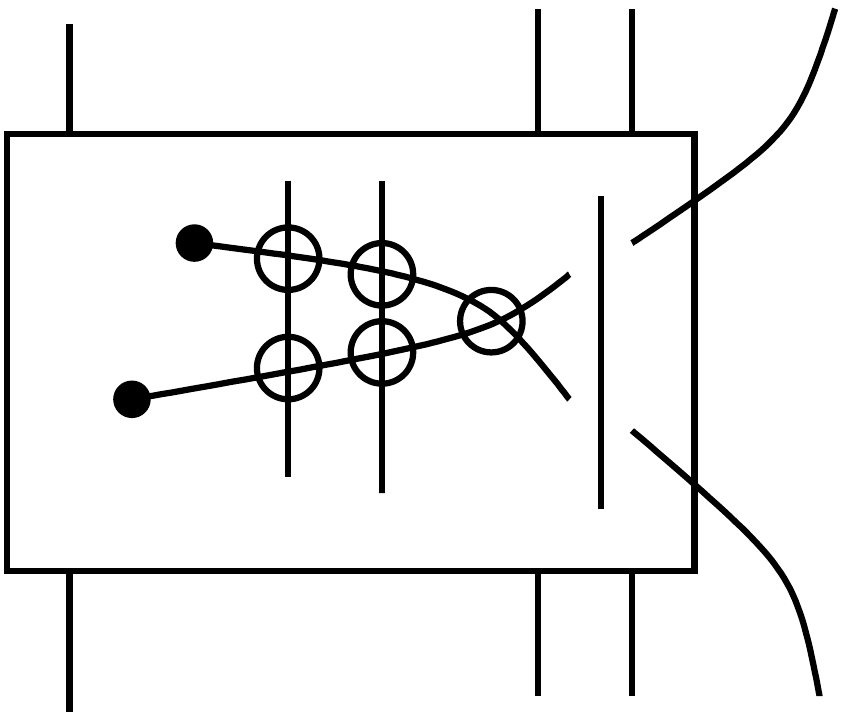}}} \,\, \stackrel{\text{virtual}}{\stackrel{\text{conjug.}}{\longleftrightarrow}}\,\, \reflectbox{\raisebox{-42pt}{\includegraphics[height=.92in]{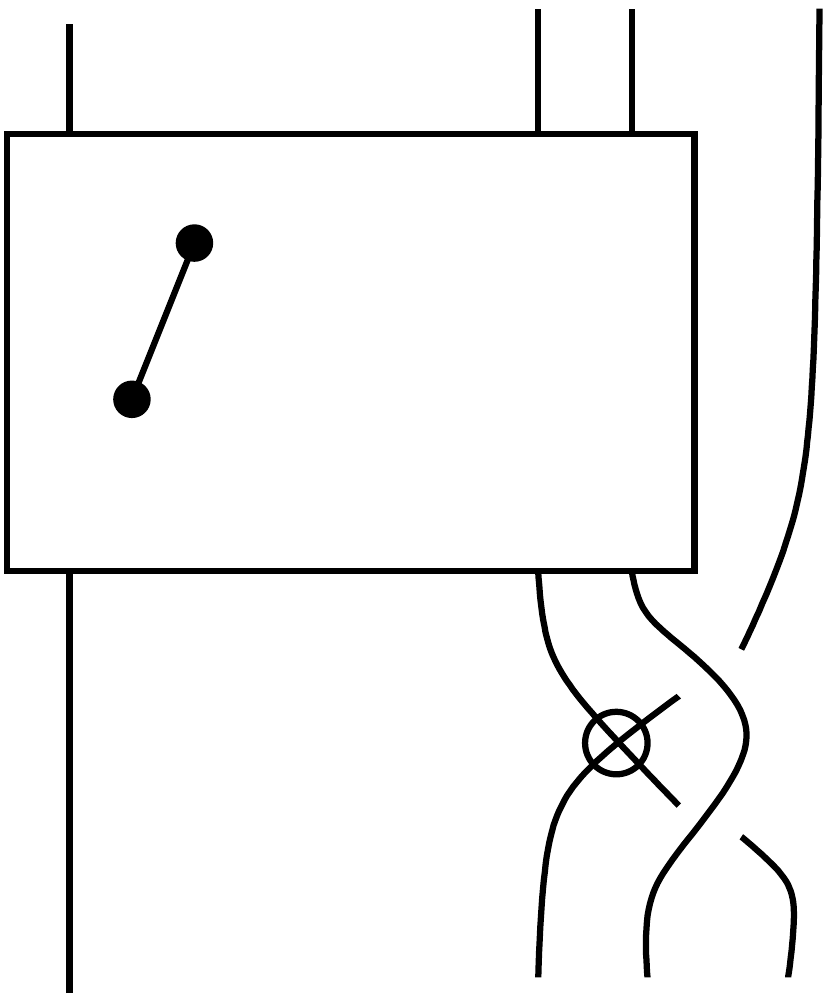}}} \]
\caption{Left algebraic under-threading follows from the left under-threaded $L_v$-move, braid detour, and virtual conjugation}\label{fig:threadingproof}
\end{figure}

The right and left algebraic $rs$-threading (the moves in (iv)) follow similarly. In Figure~\ref{fig:rsproof} we show that the right algebraic $rs$-threading follow from the right $rs$-threaded $L_v$-move, braid detour, and virtual conjugation (the left algebraic $rs$-threading follows similarly, only that it uses instead the left $rs$-threaded $L_v$-move).

\begin{figure}[ht]
\[   \raisebox{-25pt}{\includegraphics[height=.84in]{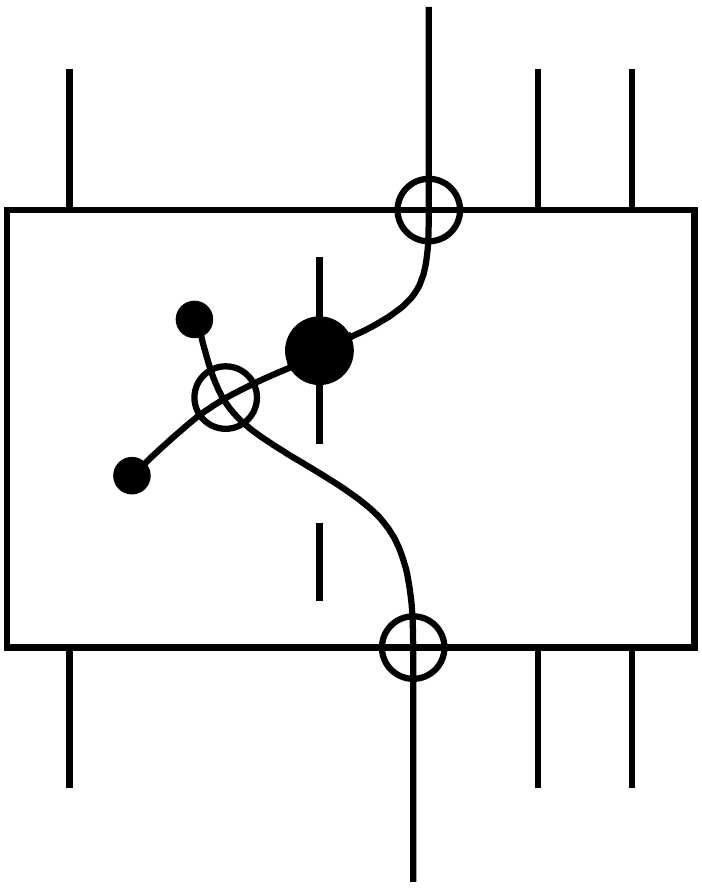}}\stackrel{\text{braid}}{\stackrel{\text{detours}}{\longleftrightarrow}} \,\,\raisebox{-25pt}{\includegraphics[height=.84in]{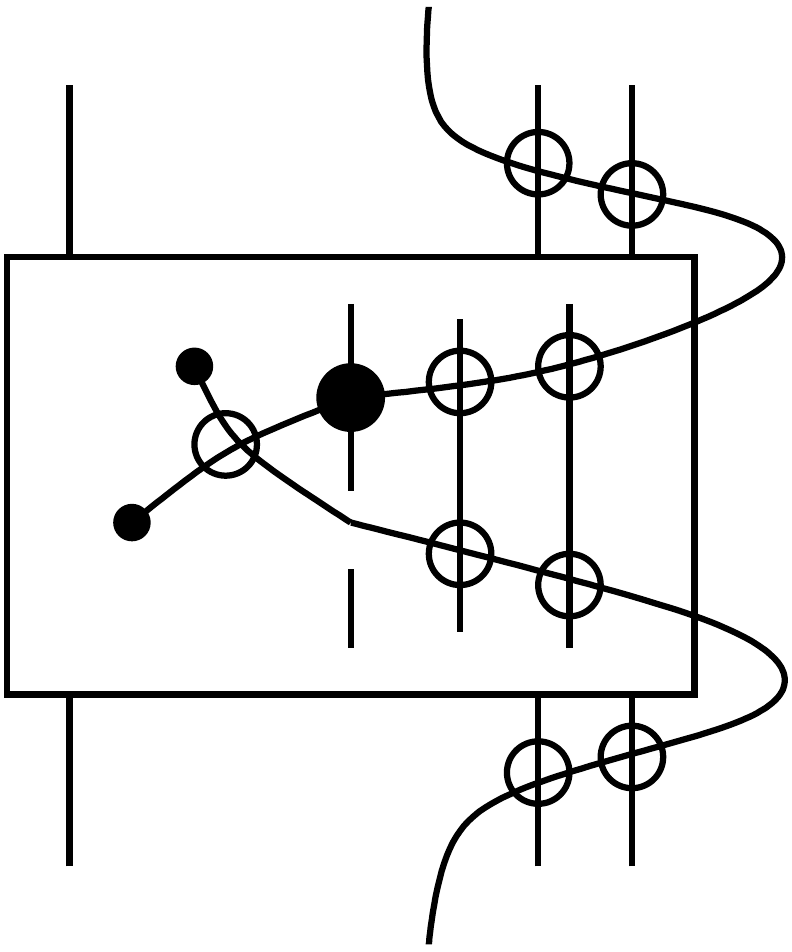}}\,\, \displaystyle \mathop{\longleftrightarrow}^{\text{virtual \, conjug.}}_{\text{detour\, threads}} \, \, \raisebox{-20pt}{\includegraphics[height=.7 in]{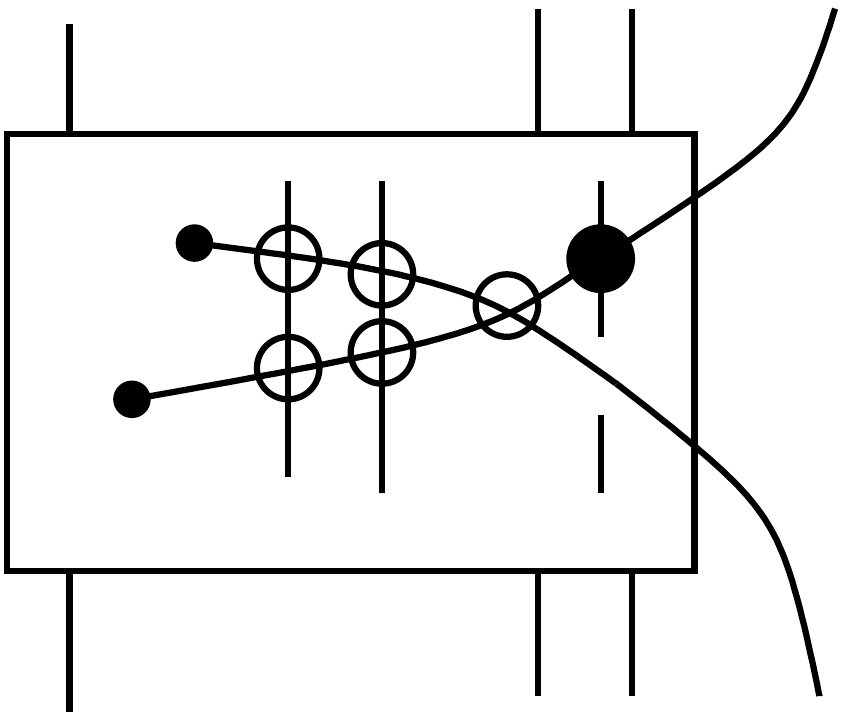}}\,\,\stackrel{\text{virtual}}{\stackrel{\text{conjug.}}{\longleftrightarrow}} \,\,\raisebox{-32pt}{\includegraphics[height=.86in]{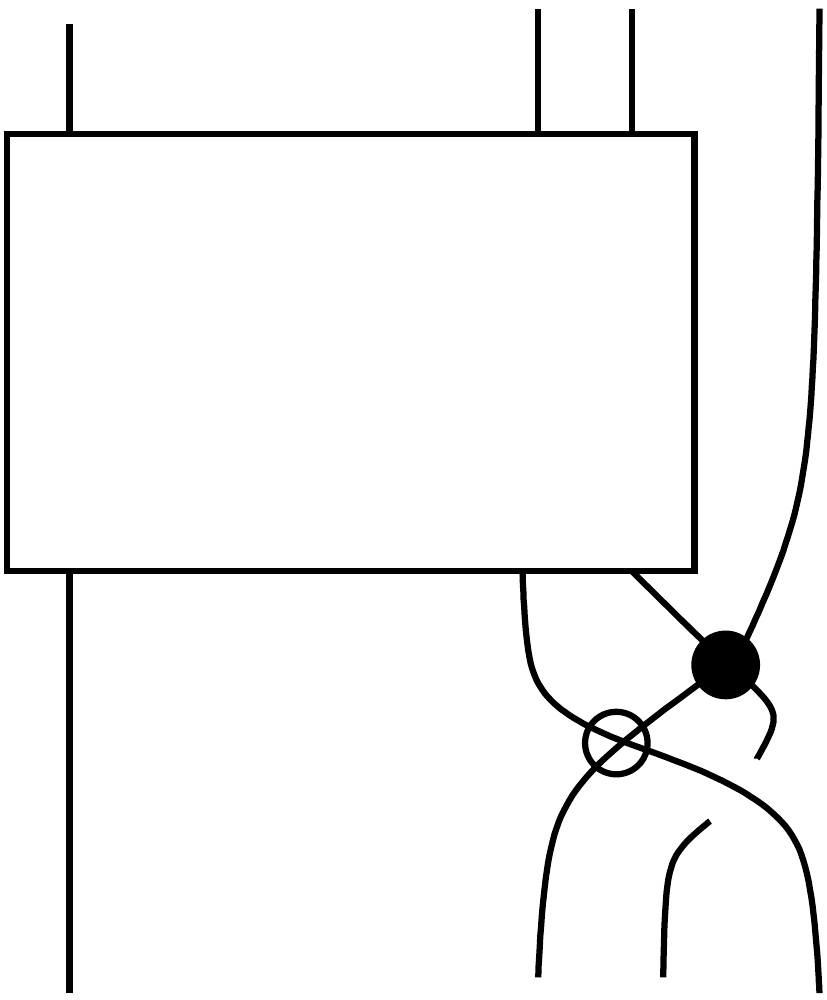}}\]

\[   \raisebox{-25pt}{\includegraphics[height=.84in]{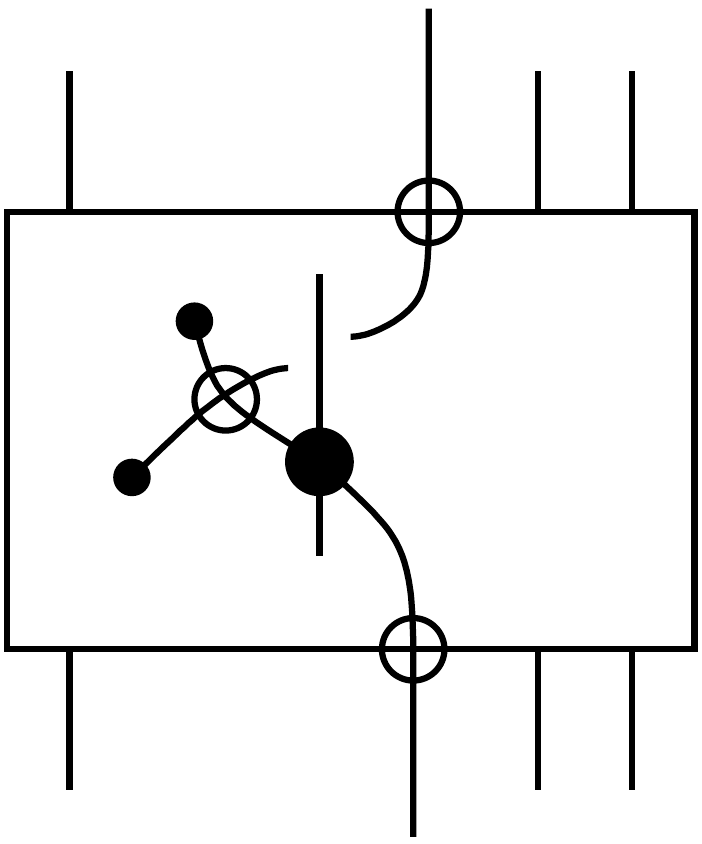}}\stackrel{\text{braid}}{\stackrel{\text{detours}}{\longleftrightarrow}} \,\,\raisebox{-25pt}{\includegraphics[height=.84in]{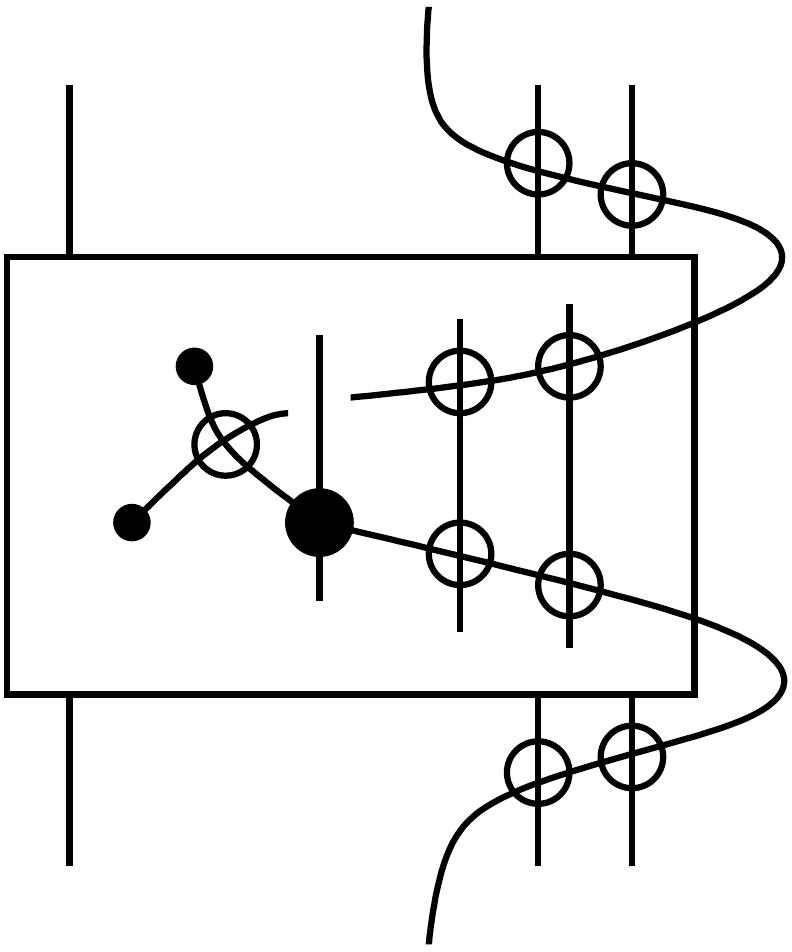}}\,\, \displaystyle \mathop{\longleftrightarrow}^{\text{virtual \, conjug.}}_{\text{detour\, threads}} \, \, \raisebox{-20pt}{\includegraphics[height=.7 in]{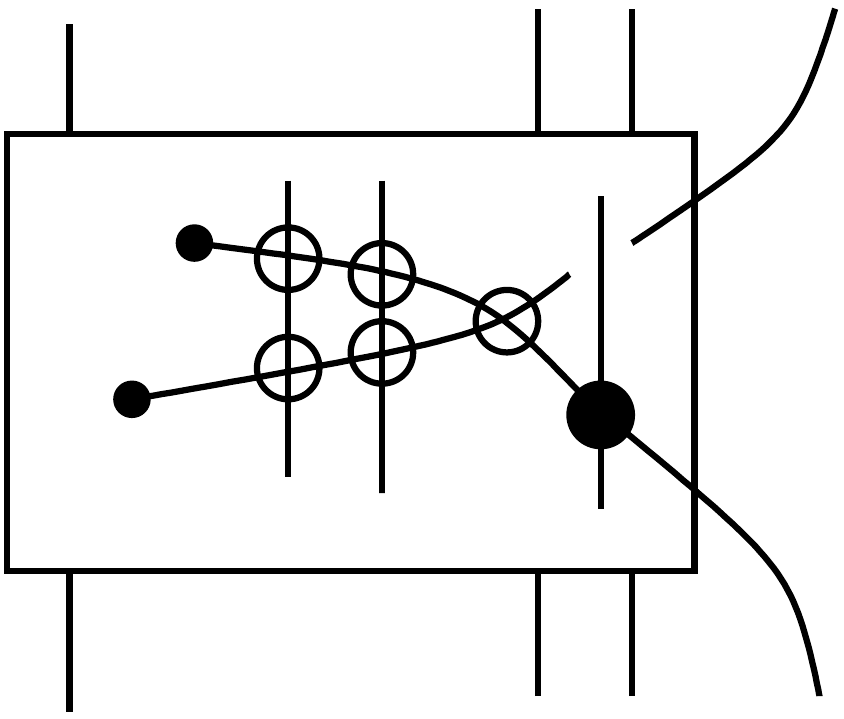}}\,\,\stackrel{\text{virtual}}{\stackrel{\text{conjug.}}{\longleftrightarrow}} \,\,\raisebox{-32pt}{\includegraphics[height=.86in]{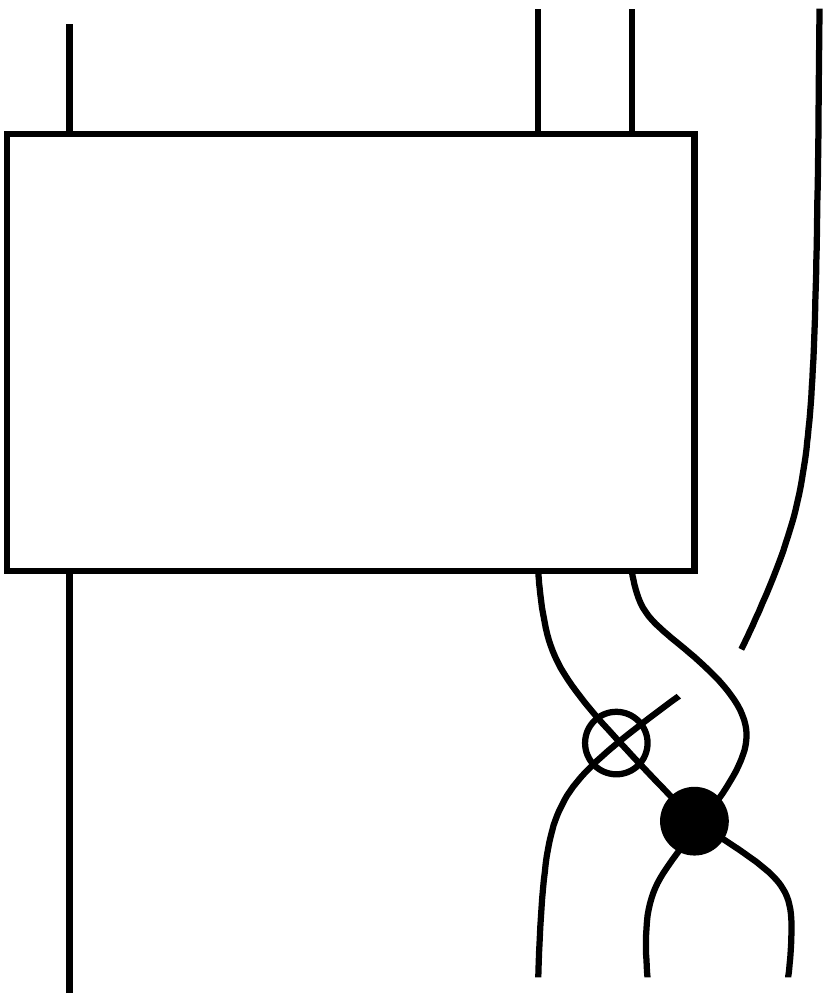}}\]
\caption{Right algebraic $rs$-threading follows from the right  $rs$-threaded $L_v$-move, braid detour, and virtual conjugation}\label{fig:rsproof}
\end{figure}
This completes the proof.
\end{proof}
%%%%%%%%%%%%%%%%%%%%%%%%%%%%%%%%%%%%%%%%%%%%%%%%%

\section{A reduced presentation for ${VSB}_n$}\label{sec:reduced}

In~\cite[Section 3] {KL1}, L.H. Kauffman and S. Lambropoulou provided a reduced presentation for the virtual braid group. Inspired by their work, in this section we give a reduced presentation for the virtual singular braid monoid on $n$ strands, $VSB_n$. This presentation uses fewer generators which listed below: 
\[\{\sigma_1, \sigma_1^{-1}, \tau_1, v_1,\ldots,v_{n-1}\}\]
and assumes the following relations, which we refer to as the \textit{defining relations:}
%defining relations
\begin{eqnarray}
\sigma_{i+1} ^{\pm 1}  & := & (v_i\ldots v_2v_1)(v_{i+1}\ldots v_3v_2)\sigma_1 ^{ \pm 1} (v_2v_3\ldots v_{i+1})(v_1v_2\ldots v_i)  \label{A14} \\  %sigma sigma inverse defining relation
\tau_{i+1} & := & (v_i\ldots v_2v_1)(v_{i+1}\ldots v_3v_2)\tau_1 (v_2v_3\ldots v_{i+1})(v_1v_2\ldots v_i), \label{A15} %tau defining relations 
\end{eqnarray}
where $1\leq i \leq n-2$.  As shown in Figure~\ref{fig:DRtau}, the defining relations are the braid form versions of the detour move. In other words, we detour the real crossings $\sigma_{i+1}^{\pm1}$ and singular crossings $\tau_{i+1}$ to the left side of the braid using the strands $1, 2, \dots, i$.

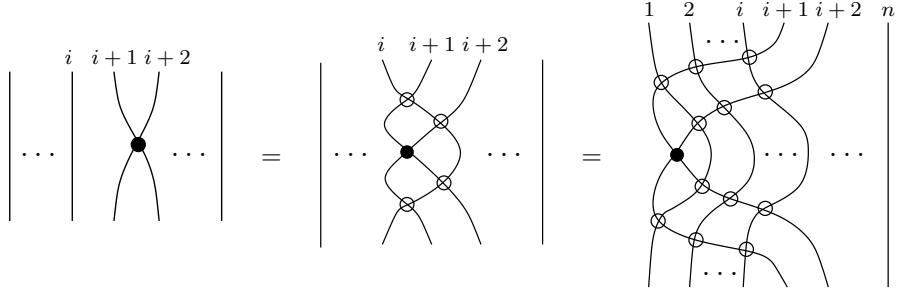
\begin{figure}[ht]
\[ \raisebox{-35pt}{\includegraphics[height=0.8in]{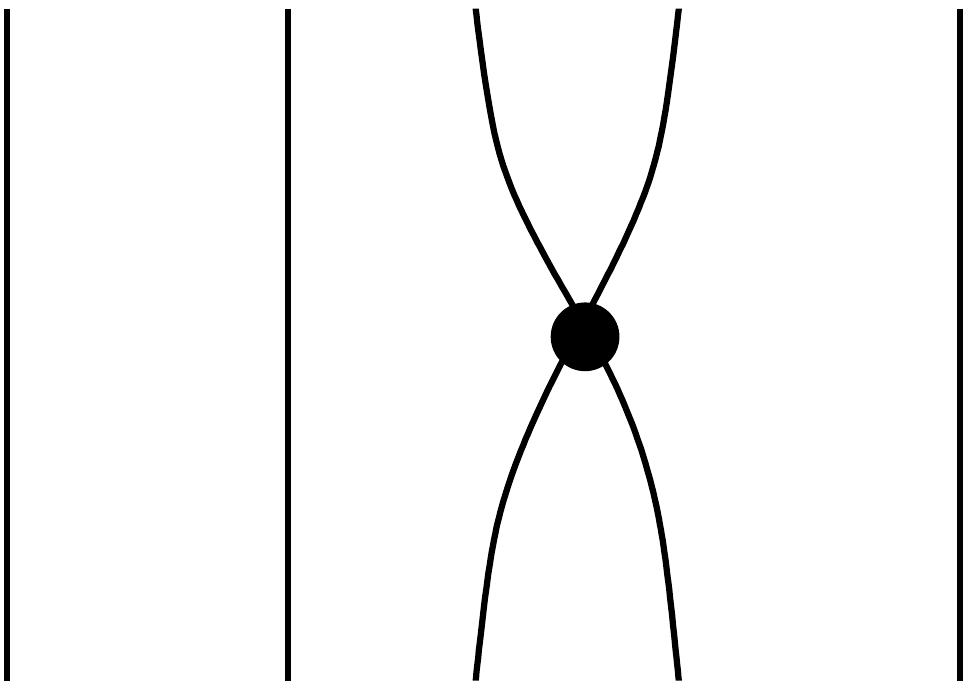}}
 \put(-77, -10){\fontsize{12}{8}$\dots$}
  \put(-20, -10){\fontsize{12}{8}$\dots$}
  \put(-60, 25){\fontsize{8}{8}$i$}
  \put(-50, 25){\fontsize{8}{8}$i+1$}
 \put(-30, 25){\fontsize{8}{8}$i+2$} \hspace{0.5cm}  \raisebox{-13pt}{=}  \hspace{0.5cm} 
\raisebox{-45pt}{\includegraphics[height=1in ]{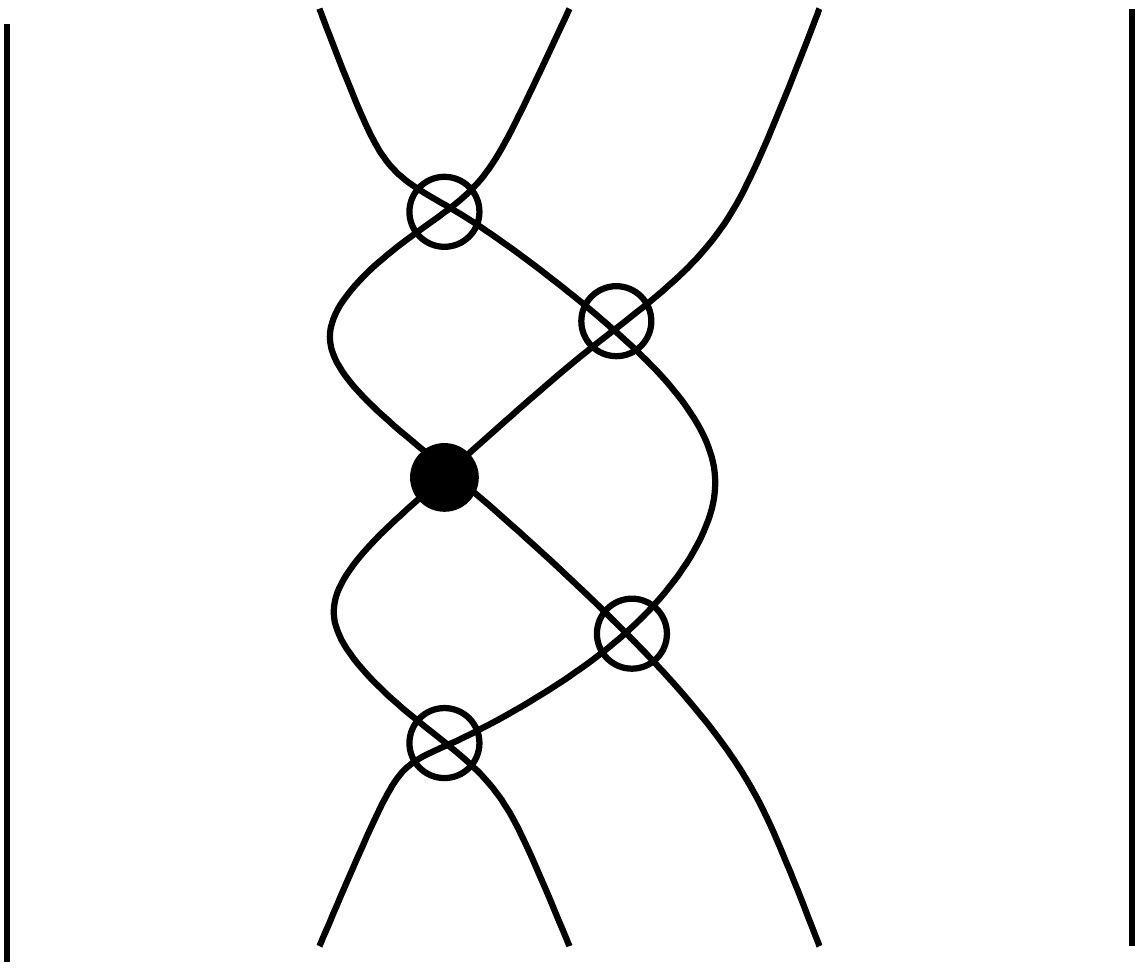}}
 \put(-80, -10){\fontsize{12}{8}$\dots$}
  \put(-22, -10){\fontsize{12}{8}$\dots$}
  \put(-63, 30){\fontsize{8}{8}$i$}
  \put(-51, 30){\fontsize{8}{8}$i+1$}
 \put(-31, 30){\fontsize{8}{8}$i+2$}\hspace{0.5cm} \raisebox{-13pt}{=} \hspace{0.6cm}
 \raisebox{-60pt}{\includegraphics[height=1.4in]{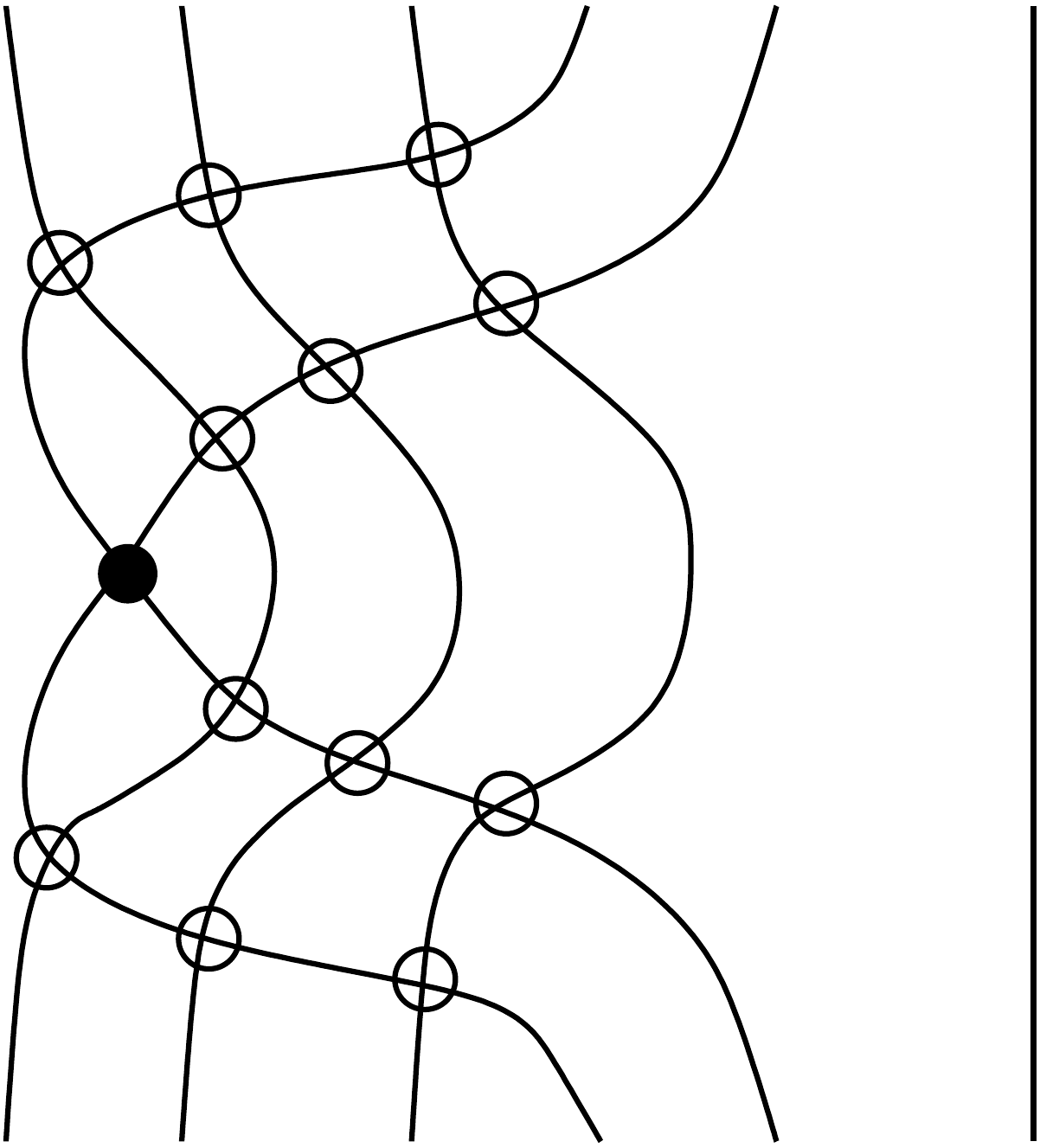}} 
  \put(-23, -10){\fontsize{12}{8}$\dots$}
    \put(-48, -10){\fontsize{12}{8}$\dots$}
     \put(-93, 43){\fontsize{8}{8}$1$}
      \put(-78, 43){\fontsize{8}{8}$2$}
       \put(-71, -55){\fontsize{12}{8}$\dots$}
          \put(-70, 33){\fontsize{12}{8}$\dots$}
    \put(-58, 43){\fontsize{8}{8}$i$}
  \put(-48, 43){\fontsize{8}{8}$i+1$}
 \put(-28, 43){\fontsize{8}{8}$i+2$}
  \put(-3, 43){\fontsize{8}{8}$n$}
 \]
 \caption{Detouring the crossing $\tau_{i+1}$}\label{fig:DRtau} 
\end{figure} 

Any portion of a given virtual singular braid can be detoured to the front of the braid (as shown in Figure~\ref{fig:braid detour-moves}), where all of the new crossings that are created are virtual. For this reason, in the reduced presentation for $VSB_n$, the relations involving real crossings or singular crossings will be imposed to occur between the first strands of a braid. 

\begin{figure}[ht]
\[ \raisebox{-60pt}{\includegraphics[height=1.8in]{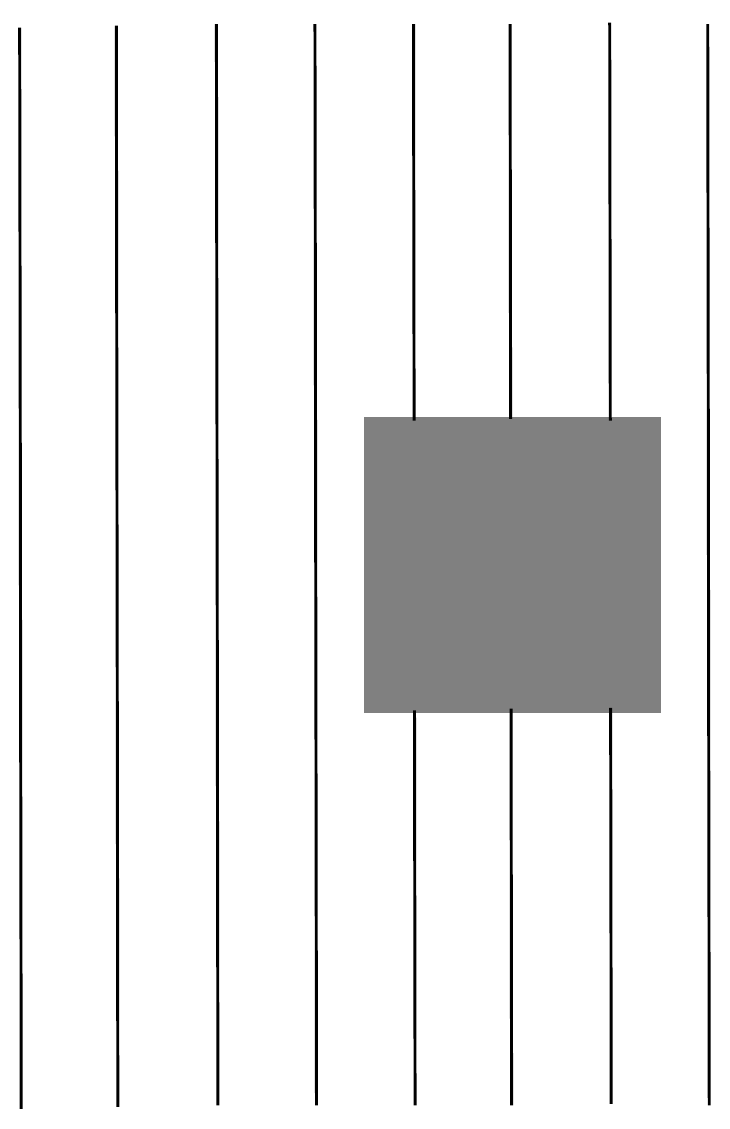}}\,\, \longleftrightarrow \,\, \raisebox{-60pt}{\includegraphics[height=1.8in]{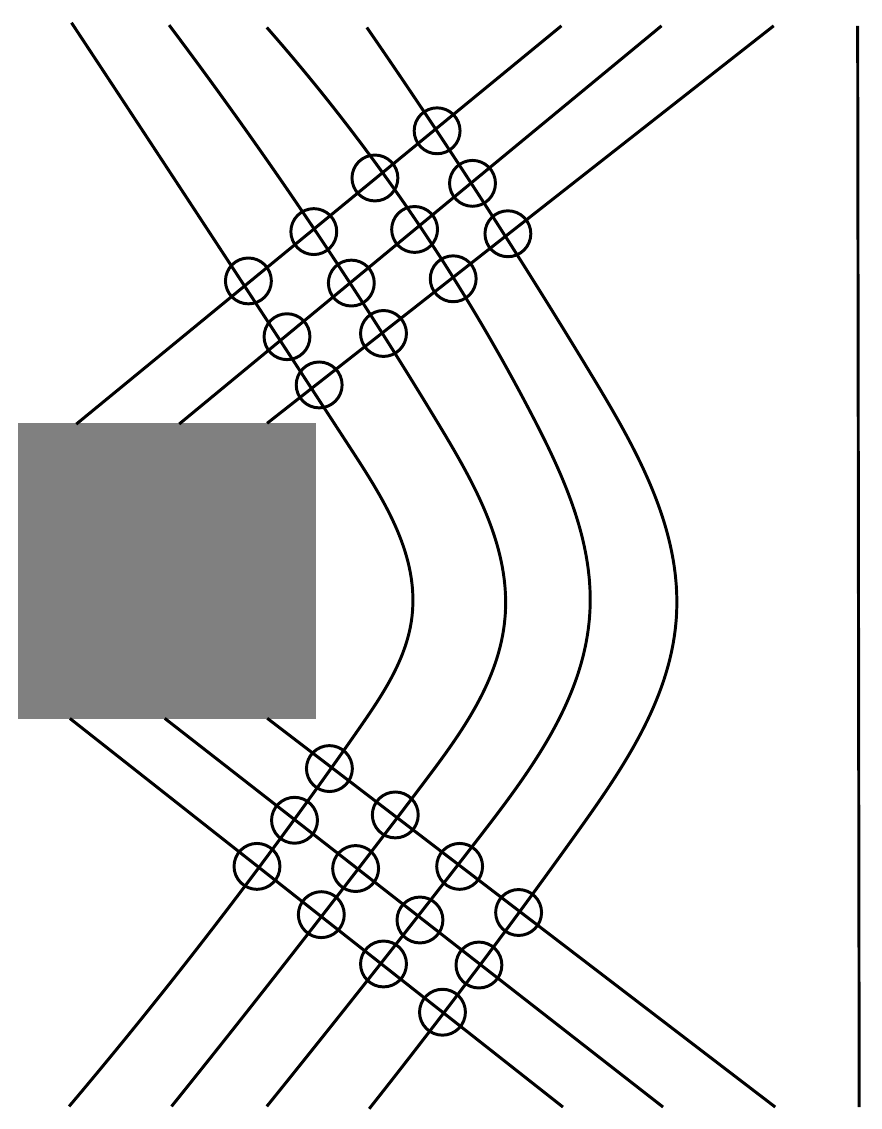}}
\]
\caption{Detouring several strands in a braid}\label{fig:braid detour-moves}
\end{figure} 

We remark that the relations $v_i\sigma_j v_i=v_j\sigma_iv_j$ and $v_i\tau_jv_i=v_j\tau_iv_j$ for $|i-j|=1$ are not needed in the reduced presentation for $VSB_n$, since they were implicitelly used in the defining relations \eqref{A14} and \eqref{A15}.

\begin{theorem} \label{thm:reduced}
The virtual singular braid monoid $VSB_{n}$ has the following reduced presentation with generators $\{\sigma_1 ^{ \pm 1}, \tau_1, v_1,\ldots,v_{n-1}\},$ 
and relations:
\begin{align}
v_i v_j v_i &=  v_j v_i v_j, && |i-j|=1          \label{A18}\\% V3
v_i v_j      &= v_j v_i             &&  |i-j|>1    \label{A19} \\ %far commutativity for virtuals
{v_i}^2      &= 1_n                   && 1 \leq i \leq n-1   \label{A21}\\%virtual identity
\sigma_1\tau_1 &=  \tau_1 \sigma_1  \,\, \, \text{and} \,\,\, \sigma_1 \sigma_1 ^{-1}  = 1_n     \label{A20a} \\ % RS1 base case and sigma identity
\tau_1 v_i  &= v_i \tau_1\,\,\, \text{and}\,\, \,\sigma_1v_i=v_i\sigma_1    &&\, i \geq 3    \label{A23} %far commutativity singular and classical w/virtual base case
\end{align}
\begin{align}
\sigma_1 ( v_1 v_2 \sigma_1 v_2 v_1 ) \sigma_1 &= (v_1 v_2 \sigma_1 v_2 v_1 ) \sigma_1(v_1 v_2 \sigma_1 v_2 v_1 )  \label{A24} \\% R3 base case 
\tau_1( v_1 v_2 \sigma_1 v_2  v_1 ) \sigma_1  &= ( v_1 v_2 \sigma_1 v_2 v_1 )\sigma_1 (v_1 v_2 \tau_1 v_2 v_1 ) \label{A24b} \\ % RS3 base case 
\sigma_1 (v_2 v_3 v_1 v_2 \sigma_1 v_2 v_1 v_3 v_2) &= (v_2 v_3 v_1 v_2 \sigma_1 v_2 v_1 v_3 v_2)  \sigma_1 \label{A27} \\ %far commutativity for classicals base case
\tau_1 (v_2 v_3 v_1 v_2 \sigma_1 v_2 v_1 v_3 v_2) &= (v_2 v_3 v_1 v_2 \sigma_1 v_2 v_1 v_3 v_2) \tau_1 \label{A28}\\%far commutativity classical  w/singular base case
\tau_1 (v_2 v_3 v_1 v_2 \tau_1 v_2 v_1 v_3 v_2) &= (v_2 v_3 v_1 v_2 \tau_1 v_2 v_1 v_3 v_2) \tau_1 \label{A29} %far commutativity for singulars base case
\end{align}
\end{theorem}

We give in Figure \ref{reduced presentations} the diagrammatic representations of relations \eqref{A24} and \eqref{A29}, pictured from left to right, respectively. 

\begin{figure}[ht]
\[   \raisebox{-10pt}{\includegraphics[height=2in]{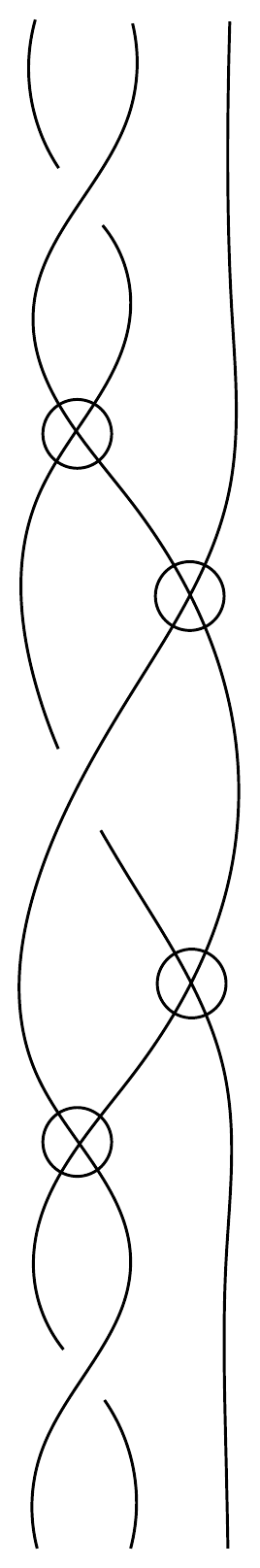}}\hspace{0.5cm} \raisebox{60pt}{=} \hspace{0.5cm}
 \,\,\raisebox{-45pt}{\includegraphics[height=3in]{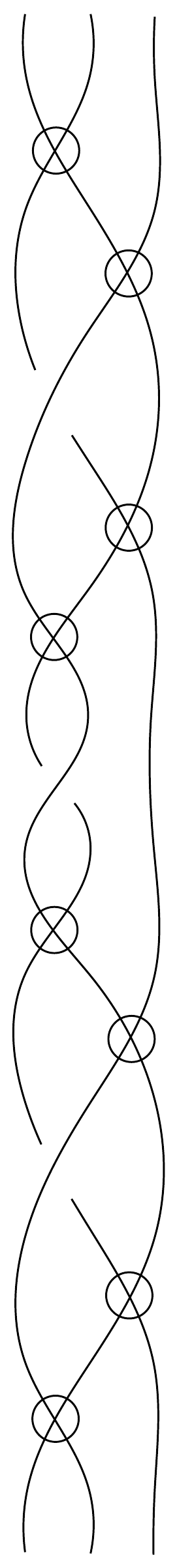}}\,\, \hspace{0.7in} \,\, \raisebox{-20pt}{\includegraphics[height=2.3in]{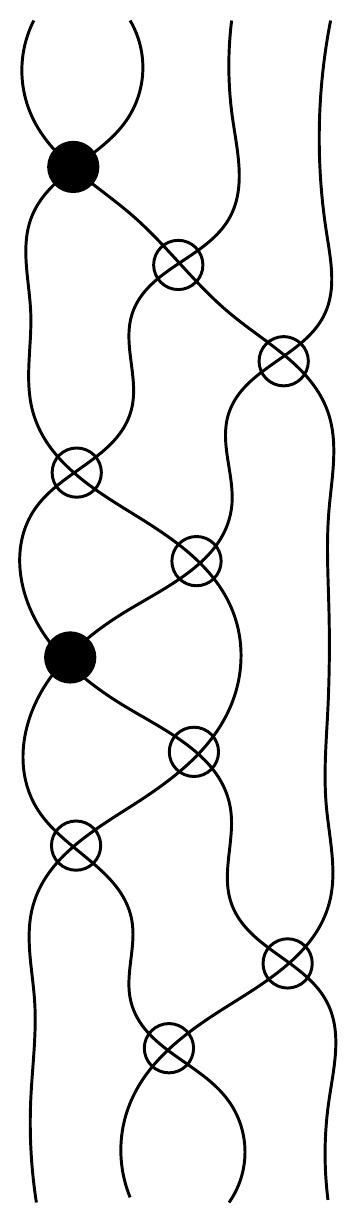}}\hspace{0.5cm} \raisebox{60pt}{=} \hspace{0.5cm}\raisebox{-20pt}{\includegraphics[height=2.3in]{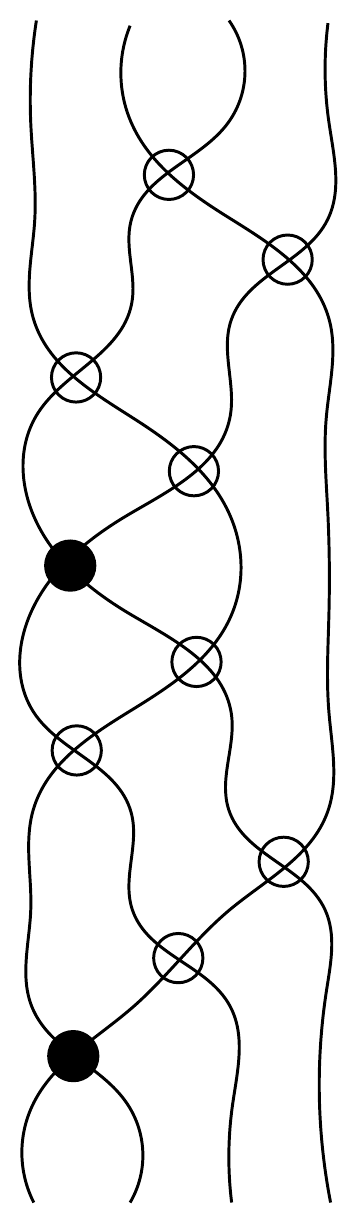}}\]
\caption{Diagrammatic representations of relations \eqref{A24} and \eqref{A29}}
\label{reduced presentations}
\end{figure}

Note that in the reduced presentation for $VSB_n$ we kept all of the original virtual relations (relations involving only virtual crossings/generators). In addition, we have kept the relations involving $\sigma_i$ or $\tau_i$ that can be represented on the left side of the braid. For convenience, we will call these the \textit{base cases} of the original relations. For example, the base case for the commuting relations $\sigma_i \sigma_j = \sigma_j \sigma_i$ is the relation $\sigma_1\sigma_3=\sigma_3\sigma_1$, which by the defining relation~\eqref{A14} is equivalent to the relation~\eqref{A27}. Similarly, from the commuting relations $\tau_i \sigma_j = \sigma_j \tau_i$ and $\tau_i \tau_j = \tau_j \tau_i$ (with $|i -j| >1$) we kept only the relations $\tau_1 \sigma_3 = \sigma_3 \tau_1$ and, respectively, $\tau_1 \tau_3 = \tau _3 \tau_1$, which are represented by the relations~\eqref{A28} and~\eqref{A29}, respectively. We will show that all of the other commuting relations follow from their corresponding base case relations and the virtual relations. 

In addition, we remark that relations~\eqref{A24} and~\eqref{A24b} represent the braid relations $\sigma_1 \sigma_2 \sigma_1 = \sigma_2 \sigma_1 \sigma_2$ and $\tau_1 \sigma_2 \sigma_1 = \sigma_2 \sigma_1 \tau_2$. Thus these two relations are the base cases for $\sigma_i \sigma_j \sigma _i = \sigma_j \sigma_i \sigma_j$ and $\tau_i \sigma_j \sigma _i = \sigma_j \sigma_i \tau_j$, respectively, where $|i-j| = 1$.

In the statements to follow we show that each of the relations in the original presentation for $VSB_n$ hold; therefore, we prove Theorem~\ref{thm:reduced}. The first lemma deals with preparatory identities (this statement was given in~\cite{KL1}, and thus we only provide a sketch of its proof). In each of our proofs below we will underline the portion of the relation that we will work with next. 
 
 %%%% Preparatory lemma
 \begin{lemma} The following equality holds for all $|i-j| \geq 2$:
\begin{eqnarray} v_iv_{i-1}\ldots { v_{j+1}{v_j} v_{j+1}}\ldots v_{i-1} v_i =  v_j v_{j+1}\ldots v_{i-1}{v_i} v_{i-1}\ldots    v_{j+1} v_j. \label{A16} 
 \end{eqnarray}  
 \end{lemma} 
\begin{proof} Let $|i-j| \geq 2$. Then, we have:
\begin{align*}
 v_iv_{i-1}\ldots \underline{ v_{j+1}{v_j} v_{j+1}}\ldots v_{i-1} v_i  \stackrel{(\ref{A18})}{=}&   v_iv_{i-1}\ldots{ \underline{v_{j}}{v_{j+1}} \underline{v_{j}}}\ldots v_{i-1}v_i   \\
    \stackrel{(\ref{A19})}{=} &    {v_j} v_iv_{i-1}\ldots \underline{ {v_{j+2}}  {v_{j+1}} { v_{j+2}}} \ldots v_{i-1} v_i {v_j}                                                                                                                                                  \\    
 \stackrel{(\ref{A18})}{ =}&    v_{j}v_ i v_{i-1}\dots \underline{v_{j+1}}v_{j+2}\underline{v_{j+1}}\ldots  {v_{i-1}} {v_{i}}v_j \\  
\stackrel{(\ref{A19})}{ =}&   v_{j} v_{j+1}v_ i v_{i-1}\dots \underline{{v_{j+3}}v_{j+2}{v_{j+3}}}\ldots  {v_{i-1}} {v_{i}}v_{j+1} v_{j}\\    
 \stackrel{(\ref{A18})}{ =}&    v_{j}v_{j+1}v_ i v_{i-1}\dots \underline{v_{j+2}}v_{j+3}\underline{v_{j+2}}\ldots  {v_{i-1}} {v_{i}}v_{j+1}v_j \\  
    &\,\,\,\,\,\,\,\,\,\,\,\,\,\,\,\,\,\,\,\,\,\,\,\,\,\,\,\,\,\,\,\,\,\,\,\,\,\,\,\,\,\,\,\,\,\,\,\,\,\,\,\,\,\,\,\,\,\,\,\,\,\vdots\\
        = & v_j v_{j+1}\ldots v_{i-1}{v_i} v_{i-1}\ldots v_{j+1} v_j. 
      \end{align*}
\end{proof}

%%%%  Lemma 2
\begin{lemma}\label{sigv-tauv}
The commuting relations $\sigma_i v_j  =  v_j \sigma_i$ and $\tau_i v_j  =  v_j \tau_i$ hold for all $|i-j|>1$.
\end{lemma}

\begin{proof} The first set of relations were proved in~\cite[Lemma 1]{KL1}. We provide here a similar proof for the second set of relations only.
By the defining relation~\eqref{A15}, we have:
  \[\tau_i v_j =(v_{i-1}\ldots
v_2v_1)\,(v_i\ldots v_3v_2)\,\tau_1\, (v_2v_3\ldots v_i)\,(v_1v_2\ldots v_{i-1})\, v_j.\]  
Since $|i-j|>1$, either $j\geq i+2$ or $j\leq i-2$. If $j\geq i+2$, then in the above
expression $v_j$  commutes  with all generators, thus $\tau_i v_j  =  v_j \tau_i$. 
 If $j\leq i-2$ we have:
\begin{align*}
\tau_i v_j  \stackrel{\eqref{A15}}{=}&  (v_{i-1}\ldots v_2v_1)(v_i\ldots v_3v_2)\tau_1 (v_2v_3\ldots v_i)(v_1v_2\ldots v_{i-1})\underline{v_j}  \\
\stackrel{(\ref{A19})}{=}&  (v_{i-1}\ldots v_1)(v_i\ldots v_2)\tau_1 (v_2v_3\ldots v_i)(v_1v_2 \ldots v_{j-1}\underline{v_j v_{j+1} v_j}  v_{j+2} \ldots v_{i-1})  \\ 
\stackrel{(\ref{A18})}{=}&  (v_{i-1}\ldots v_1)(v_i\ldots v_2)\tau_1 (v_2v_3\ldots v_{j+1} v_{j+2} \ldots v_i)(v_1v_2 \ldots v_{j-1} \underline{v_{j+1}} v_j \cdot 
\\& v_{j+1} v_{j+2} \ldots v_{i-1}) \\
\stackrel{(\ref{A19})}{=}&  (v_{i-1}\ldots v_1)(v_i\ldots v_3v_2)\tau_1 (v_2v_3\ldots  v_j \underline{ v_{j+1} v_{j+2} v_{j+1}}  v_{j+3} \ldots v_i) (v_1v_2 \ldots v_{i-1}) \\
\stackrel{(\ref{A18})}{=}&  (v_{i-1}\ldots v_1)(v_i\ldots v_{j+2} v_{j+1} v_j \ldots v_3v_2)\tau_1(v_2v_3\ldots  v_j\underline{ v_{j+2}} v_{j+1} v_{j+2} v_{j+3} \ldots v_i) 
\\& (v_1v_2 \ldots v_{i-1}) \\
 \displaystyle  \mathop{=}_{\eqref{A19}}^{\eqref{A23}}&  (v_{i-1}\ldots v_1)(v_i\ldots v_{j+3} \underline{ v_{j+2}  v_{j+1} v_{j+2}} v_j\ldots v_2) \tau_1 (v_2\ldots  v_i) (v_1 \ldots v_{i-1}) \\
\stackrel{(\ref{A18})}{=}&  (v_{i-1}\ldots  v_{j+1} v_j v_{j-1} \ldots v_2v_1)(v_i\ldots  v_{j+3} \underline{ v_{j+1}} v_{j+2} v_{j+1} v_j\ldots v_2)\tau_1 (v_2\ldots  v_i) \\& (v_1 \ldots v_{i-1}) \\
\stackrel{(\ref{A19})}{=}&  (v_{i-1}\ldots v_{j+2} \underline{v_{j+1} v_j v_{j+1} } v_{j-1} \ldots v_2v_1)  (v_i\ldots v_3v_2)\tau_1 (v_2\ldots  v_i) (v_1 \ldots v_{i-1}) \\
\stackrel{(\ref{A18})}{=}& (v_{i-1}\ldots v_{j+2} \underline{v_j }  v_{j+1}  v_j  v_{j-1} \ldots v_2v_1)  (v_i\ldots v_3v_2) \tau_1(v_2\ldots  v_i) (v_1 \ldots v_{i-1}) \\
\stackrel{(\ref{A19})}{=}&  v_j (v_{i-1}\ldots  v_1)  (v_i\ldots v_2) \tau_1 (v_2\ldots  v_i) (v_1 \ldots v_{i-1}) \\
 \displaystyle  \stackrel{\eqref{A14}}{=} &  v_j \tau_i.  
\end{align*} 
Hence, the statement holds.\end{proof}

%Lemma 3
\begin{lemma}\label{sig-tau}
The commuting braid relations below hold for all $|i-j|>1$:
\[ \sigma_i  \sigma_j = \sigma_j \sigma_i, \,\, \tau_i  \tau_j = \tau_j \tau_i\,\, \, \text{and} \, \,\, \sigma_i  \tau_j = \tau_j \sigma_i.\]
\end{lemma}

\begin{proof} The same type of proof can be used to show that the given three sets of commuting braid relations hold. For brevity, we will prove here only the last set of relations. Our proof is somewhat different than the one given in~\cite[Lemma 3]{KL1} for $\sigma_i  \sigma_j = \sigma_j \sigma_i$.

Without loss of generality, assume $j>i$. Specifically, suppose that $j \geq i+2$. Then,
\begin{align*}
1_n = \,\,\, &  1_n 1_n \\ 
\stackrel{\eqref{A21}}{=}& (v_{i-1} \dots v_1)(v_1 \dots v_{i-1}) \underline{(v_{j-1} \dots v_{i+1})} (v_i \dots v_1)(v_1 \dots v_i)(v_{i+1} \dots v_{j-1}) \\
\stackrel{\eqref{A19}}{=} &  (v_{i-1} \dots v_1) (v_{j-1} \dots v_{i+1}) \underline{(v_1 \dots v_{i-1} v_i v_{i-1} \dots v_1) }(v_1 \dots v_i)(v_{i+1} \dots v_{j-1}) \\
\stackrel{\eqref{A16}}{=}& (v_{i-1} \dots v_1) (v_{j-1} \dots v_{i+1}) (v_i \dots v_1 \dots v_i) (v_1 \dots v_i)(v_{i+1} \dots v_{j-1}) \\
= \,\,\, & (v_{i-1} \dots v_1) (v_{j-1} \dots v_2)( v_1v_2 \dots v_i)(v_1 \dots v_{j-1}) \\
=\,\,\, &  (v_{i-1} \dots v_1)(v_{j-1} \dots v_{i+2}) \underline{(1_n)} (v_{i+1} \dots v_2) (v_1 \dots v_i)(v_1 \dots v_{j-1}) \\
\stackrel{\eqref{A21}}{=}& (v_{i-1} \dots v_1)(v_{j-1} \dots v_{i+2}) \underline{(v_i \dots v_2) } (v_2 \dots v_i) (v_{i+1} \dots v_2) (v_1 \dots v_i)(v_1 \dots v_{j-1}) \\
\stackrel{\eqref{A19}}{=}&  (v_{i-1} \dots v_1)(v_i \dots v_2)(v_{j-1} \dots v_{i+2}) (v_2 \dots v_i) (v_{i+1} \dots v_2) (v_1 \dots v_i)(v_1 \dots v_{j-1}) \\
= \,\,\, &  (v_{i-1} \dots v_1)(v_i \dots v_2)(v_{j-1} \dots v_{i+2}) \underline{(v_2 \dots v_i v_{i+1} v_i \dots v_2) } (\underline{v_1} \dots v_i)(v_1 \dots v_{j-1}) \\
\stackrel{\eqref{A16}}{=}&  (v_{i-1} \dots v_1)(v_i \dots v_2)(v_{j-1} \dots v_{i+2}) (v_{i+1} \dots  v_{2}  \dots v_{i+1} )  (v_1 \dots v_i)(v_1 \dots v_{j-1}) \\
= \,\,\, & (v_{i-1} \dots v_1)(v_i \dots v_2)(v_{j-1} \dots  v_2 ) (v_3  \dots v_{i+1} )  (\underline{v_1} \dots v_i)(v_1 \dots v_{j-1}) \\
\stackrel{\eqref{A19}}{=}& (v_{i-1} \dots v_1)(v_i \dots v_2)(v_{j-1} \dots  v_2 v_1) \underline{(1_n)} (v_3  \dots v_{i+1} )  (v_2 \dots v_i)(v_1 \dots v_{j-1}) \\
\stackrel{\eqref{A21}}{=}&  (v_{i-1} \dots v_1)(v_i \dots v_2)(v_{j-1} \dots  v_2 v_1) (v_j \dots v_{i+3})\underline{(v_{i+3} \dots v_j)} \cdot\\
& (v_3  \dots v_{i+1} )  (v_2 \dots v_i)(v_1 \dots v_{j-1})
\end{align*}
If $j = i+2$ then $v_{i+3}\ldots v_j = \emptyset$. If $j \geq i+3$, then $(v_{i+3}\ldots v_j)$ commutes with $v_1, v_2, \dots, v_{i+1}$. In either case, we have the following:
\begin{align*}
1_n = \,\,\, &  (v_{i-1} \dots v_1)(v_i \dots v_2)(v_{j-1} \dots  v_2 v_1) (v_j \dots v_{i+3}) (v_3  \dots v_{i+1} )  (v_2 \dots v_i) \cdot\\
&(v_{i+3} \dots v_j)(v_1 \dots v_{j-1}) \\
=\,\,\, & (v_{i-1} \dots v_1)(v_i \dots v_2)(v_{j-1} \dots  v_2 v_1) (v_j \dots v_{i+3}) (v_3  \dots v_{i+1} ) \underline{(1_n)} (v_2 \dots v_i) \cdot\\
&(v_{i+3} \dots v_j)(v_1 \dots v_{j-1}) \\
 \stackrel{\eqref{A21}}{=}&  (v_{i-1} \dots v_1)(v_i \dots v_2)(v_{j-1} \dots  v_2 v_1) (v_j \dots v_{i+3}) (v_3  \dots v_{i+1} ) \cdot \\ 
& (v_{i+2}  \underline{v_{i+2}}) (v_2 \dots v_i)(v_{i+3} \dots v_j)(v_1 \dots v_{j-1}) \\
\stackrel{\eqref{A19}}{=}& (v_{i-1} \dots v_1)(v_i \dots v_2)(v_{j-1} \dots  v_2 v_1) (v_j \dots v_{i+3}) (v_3  \dots v_{i+1} )\cdot \\
& (v_{i+2}) (v_2 \dots v_i)\underline{(1_n)}(v_{i+2} \dots v_j)(v_1 \dots v_{j-1}) \\
\stackrel{\eqref{A21}}{=}& (v_{i-1} \dots v_1)(v_i \dots v_2)(v_{j-1} \dots  v_2 v_1) (v_j \dots v_{i+3}) (v_3  \dots v_{i+1} ) \cdot\\
& (v_{i+2}) \underline{ (v_2 \dots v_i)(v_{i+1} v_i \dots v_2) }  (v_2  \dots v_{i+1})(v_{i+2} \dots v_j)(v_1 \dots v_{j-1}) \\
\stackrel{\eqref{A16}}{=}& (v_{i-1} \dots v_1)(v_i \dots v_2)(v_{j-1} \dots  v_2 v_1) (v_j \dots v_{i+3}) (v_3  \dots v_{i+1} ) \cdot \\
& (v_{i+2}) (v_{i+1} \dots v_2  \dots v_{i+1})  (v_2  \dots v_j)  (v_1 \dots v_{j-1}) \\ 
=& (v_{i-1} \dots v_1)(v_i \dots v_2)(v_{j-1} \dots  v_2 v_1) (v_j \dots v_{i+3}) \underline{ (v_3  \dots v_{i+1} v_{i+2} v_{i+1} \dots v_3 } v_2  \dots v_{i+1}) \cdot \\
& (v_2  \dots v_j)(v_1 \dots v_{j-1}) \\ 
\stackrel{\eqref{A16}}{=}& (v_{i-1} \dots v_1)(v_i \dots v_2)(v_{j-1} \dots  v_2 v_1) (v_j \dots v_{i+3}) (v_{i+2}  \dots  v_3  \dots v_{i+2}) \cdot \\
&( v_2  \dots v_{i+1}) (v_2 \dots v_j)(v_1 \dots  v_{j-1}) \\
=\,\,\, & (v_{i-1} \dots v_1)(v_i \dots v_2)(v_{j-1} \dots  \underline{v_2 v_1}) (v_j \dots v_4) (v_3  \dots v_{i+2}) \cdot \\
&( v_2  \dots v_{i+1}) (v_2 \dots v_j)(v_1 \dots  v_{j-1}) \\
\stackrel{\eqref{A19}}{=}&  (v_{i-1} \dots v_1)(v_i \dots v_2)(v_{j-1} \dots  v_3) (v_j \dots v_4) (v_2 v_3  \dots v_{i+2}) ( v_1 v_2  \dots v_{i+1})\cdot
 \\& (v_2 \dots v_j)(v_1 \dots  v_{j-1}). 
 \end{align*}
 
 Then, we have:
 \begin{align*}
 \tau_j \underline{\sigma_i} \stackrel{\eqref{A14}}{=}& \tau_j (v_{i-1} \dots v_1)(v_i \dots v_2)\sigma_1 \underline{(v_2 \dots v_i)(v_1 \dots v_{i-1}) (v_{i-1} \dots v_1)(v_i \dots v_2)} \cdot \\
 &(v_{j-1} \dots  v_3) (v_j \dots v_4) (v_2 v_3  \dots v_{i+2}) ( v_1 v_2   \dots v_{i+1}) (v_2  \dots v_j)(v_1 \dots v_{j-1}) \\
\stackrel{\eqref{A21}}{=}& \underline{ \tau_j} (v_{i-1} \dots v_1)(v_i \dots v_2)\underline{\sigma_1} (v_{j-1} \dots  v_3) (v_j \dots v_4) (v_2 v_3  \dots v_{i+2}) \cdot
 \\& ( v_1 v_2  \dots v_{i+1}) (v_2 \dots v_j)  (v_1 \dots v_{j-1}) \\
\stackrel{\eqref{A23}}{=}&  (v_{i-1} \dots v_1)(v_i \dots v_2) \tau_j  (v_{j-1} \dots  v_3)  (1_n) (v_j \dots v_4) \sigma_1 (v_2 v_3  \dots v_{i+2}) \cdot \\
&( v_1 v_2  \dots v_{i+1}) (v_2  \dots v_j)(v_1 \dots v_{j-1}) \\
= \,\,\, &  (v_{i-1} \dots v_1)(v_i \dots v_2) \tau_j  (v_{j-1} \dots  v_3)  (v_2 \underline{v_2}) (v_j \dots v_4) \sigma_1 (v_2 v_3  \dots v_{i+2})  \cdot
\\&( v_1 v_2  \dots v_{i+1}) (v_2  \dots v_j) (v_1 \dots v_{j-1}) \\ 
\stackrel{\eqref{A19}}{=}&  (v_{i-1} \dots v_1)(v_i \dots v_2) \tau_j  (v_{j-1} \dots  v_3)  (v_2) (v_j \dots v_4)\underline{ (1_n)} v_2 \sigma_1 (v_2 v_3  \dots v_{i+2}) \cdot \\
&( v_1 v_2  \dots v_{i+1})  (v_2 \dots  v_j)(v_1 \dots v_{j-1}) \\ 
= \,\,\, &  (v_{i-1} \dots v_1)(v_i \dots v_2) \tau_j  (v_{j-1} \dots  v_3)  (v_2) (v_j \dots v_4) (v_3 v_3) v_2 \sigma_1 (v_2 v_3  \dots v_{i+2})  \cdot
 \\& ( v_1 v_2  \dots v_{i+1})(v_2 \dots v_j)  (v_1 \dots v_{j-1}) \\
= \,\,\, &  (v_{i-1} \dots v_1)(v_i \dots v_2) \tau_j  (v_{j-1} \dots  v_3)  (v_2)\underline{(1_n)} (v_j \dots v_3) v_3 v_2 \sigma_1 (v_2 v_3  \dots v_{i+2})  \cdot
 \\& ( v_1 v_2  \dots v_{i+1})(v_2  \dots v_j)(v_1 \dots v_{j-1}) \\
= \,\,\, &  (v_{i-1} \dots v_1)(v_i \dots v_2) \tau_j  (v_{j-1} \dots  v_3)  (v_2) (v_1 \underline{v_1}) (v_j \dots v_3) v_3 v_2 \sigma_1 (v_2 v_3  \dots v_{i+2}) \cdot \\
 &( v_1 v_2  \dots v_{i+1})(v_2  \dots v_j)(v_1 \dots v_{j-1}) \\
\stackrel{\eqref{A19}}{=}&  (v_{i-1} \dots v_1)(v_i \dots v_2) \tau_j  (v_{j-1} \dots v_2 v_1) (v_j \dots v_3)  \underline{(1_n)}v_1 v_3 v_2 \sigma_1 (v_2 v_3  \dots v_{i+2}) \cdot
 \\& ( v_1 v_2  \dots v_{i+1})  (v_2  \dots v_j)(v_1 \dots v_{j-1}) \\
= \,\,\,&  (v_{i-1} \dots v_1)(v_i \dots v_2) \tau_j  (v_{j-1} \dots v_2 v_1) (v_j \dots v_3) (v_2 v_2)v_1 v_3 v_2 \sigma_1 (v_2 v_3  \dots v_{i+2}) \cdot \\
& ( v_1 v_2  \dots v_{i+1}) (v_2 \dots v_j)(v_1 \dots v_{j-1}) \\
=\,\,\, &  (v_{i-1} \dots v_1)(v_i \dots v_2) \underline{\tau_j}  (v_{j-1} \dots v_2 v_1) (v_j \dots v_3 v_2 ) v_2 v_1 v_3 v_2 \sigma_1 (v_2 v_3  \dots v_{i+2}) \cdot \\
&( v_1 v_2  \dots v_{i+1}) (v_2 \dots v_j)(v_1 \dots v_{j-1}) \\
\stackrel{\eqref{A15}}{=}&  (v_{i-1} \dots v_1)(v_i \dots v_2) (v_{j-1}\ldots v_1)(v_{j}\ldots v_2)\tau_1 \underline{ (v_2 \ldots v_{j})(v_1 \ldots v_{j-1})} \cdot \\
&\underline{(v_{j-1} \dots  v_1) (v_j \dots  v_2 ) }v_2 v_1 v_3 v_2 \sigma_1 (v_2 v_3  \dots v_{i+2}) ( v_1 v_2  \dots v_{i+1}) (v_2  \dots v_j)(v_1 \dots v_{j-1}) \\ 
\stackrel{\eqref{A21}}{=}& \underline{(1_n)} (v_{i-1} \dots v_1)(v_i \dots v_2) (v_{j-1}\ldots v_1)(v_{j}\ldots v_2)\tau_1 v_2 v_1 v_3 v_2 \sigma_1 \cdot
 \\& (v_2 v_3  \dots v_{i+2})  ( v_1 v_2  \dots v_{i+1}) (v_2  \dots v_j)(v_1 \dots v_{j-1}) \\
\stackrel{\eqref{A21}}{=}& (v_{j-1} \dots v_{i+2}) \underline{(v_{i+2} \dots v_{j-1})} (v_{i-1} \dots v_1)(v_i \dots v_2) (v_{j-1}\ldots v_1)(v_{j}\ldots v_2)\cdot \\
& \tau_1 v_2 v_1 v_3 v_2 \sigma_1(v_2 v_3  \dots v_{i+2}) ( v_1 v_2  \dots v_{i+1}) (v_2 \dots v_j)(v_1 \dots v_{j-1}) \\
\stackrel{\eqref{A19}}{=}& (v_{j-1} \dots v_{i+2})  (v_{i-1} \dots v_1)(v_i \dots v_2) (v_{i+2} \dots v_{j-1}) (v_{j-1}\ldots v_1)(v_{j}\ldots v_2)\cdot \\
&\tau_1 v_2 v_1 v_3 v_2 \sigma_1  (v_2 v_3  \dots v_{i+2}) ( v_1 v_2  \dots  v_{i+1}) (v_2  \dots v_j)(v_1 \dots v_{j-1}) \\
= \,\,\, & (v_{j-1} \dots v_{i+2})  (v_{i-1} \dots v_1)(v_i \dots v_2) \underline{(v_{i+2} \dots v_{j-1}) (v_{j-1}\ldots v_{i+2}) }(v_{i+1} \dots v_1)\cdot \\
&(v_{j}\ldots v_2)  \tau_1 v_2 v_1 v_3 v_2 \sigma_1 (v_2 v_3 \dots v_{i+2}) ( v_1 v_2  \dots v_{i+1}) (v_2 \dots v_j)(v_1 \dots v_{j-1}) \\ 
\stackrel{\eqref{A19}}{=}& (v_{j-1} \dots v_{i+2}) \underline{(1_n)}  (v_{i-1} \dots v_1)(v_i \dots v_2) (v_{i+1} \dots v_1)(v_{j}\ldots v_2) \cdot \\
&\tau_1 v_2 v_1 v_3 v_2 \sigma_1 (v_2 v_3  \dots v_{i+2}) ( v_1 v_2  \dots v_{i+1}) (v_2  \dots v_j)(v_1 \dots v_{j-1}) \\
\stackrel{\eqref{A19}}{=}& (v_{j-1} \dots v_{i+2}) (v_j \dots v_{i+3})\underline{(v_{i+3} \dots v_j ) }  (v_{i-1} \dots v_1)(v_i \dots v_2) (v_{i+1} \dots v_1)\cdot \\
&(v_{j}\ldots v_2) \tau_1 v_2 v_1 v_3 v_2 \sigma_1 (v_2 v_3  \dots v_{i+2}) ( v_1 v_2  \dots v_{i+1}) (v_2  \dots v_j)(v_1 \dots v_{j-1}) \\
\stackrel{\eqref{A19}}{=}& (v_{j-1} \dots v_{i+2}) (v_j \dots v_{i+3})  (v_{i-1} \dots v_1)(v_i \dots v_2) (v_{i+1} \dots v_1)(v_{i+3}  \dots v_j ) \cdot \\
&(v_{j}\ldots v_2)\tau_1 v_2 v_1 v_3 v_2 \sigma_1 (v_2 v_3 v_4  \dots v_{i+2}) ( \underline{v_1 v_2 } \dots v_{i+1}) (v_2  \dots v_j)(v_1 \dots v_{j-1}) \\
\stackrel{\eqref{A19}}{=}& (v_{j-1} \dots v_{i+2}) (v_j \dots v_{i+3})  (v_{i-1} \dots v_1)(v_i \dots v_2) (v_{i+1} \dots v_1)\underline{(v_{i+3} \dots v_j )} \cdot \\
& \underline{(v_{j}\ldots v_{i+3} } \dots v_2)\tau_1 v_2 v_1 v_3 v_2  \sigma_1 v_2 v_3 v_1 v_2(v_4  \dots v_{i+2}) (  v_3  \dots v_{i+1})\cdot \\
& (v_2 \dots v_j)(v_1 \dots v_{j-1}) \\
\stackrel{\eqref{A21}}{=}& (v_{j-1} \dots v_{i+2}) (v_j \dots v_{i+3})  (v_{i-1} \dots v_1)(v_i \dots v_2) (v_{i+1} \dots v_1) (v_{i+2} \dots v_2) \cdot \\
& \tau_1 v_2 v_1 v_3 v_2 \sigma_1 v_2 v_3 v_1 v_2  (v_4  \dots v_{i+2}) (  v_3  \dots v_{i+1}) (v_2  \dots v_j)(v_1 \dots v_{j-1}) \\
=\,\,\, & (v_{j-1} \dots v_{i+2}) (v_j \dots v_{i+3})  (v_{i-1} \dots v_1)(v_i \dots v_2) (v_{i+1} \dots v_2 \underline{v_1}) (v_{i+2} \dots v_4) \cdot \\
& v_3 v_2\tau_1 v_2 v_1 v_3 v_2 \sigma_1 v_2 v_3 v_1 v_2(v_4   \dots v_{i+2}) (  v_3  \dots v_{i+1}) (v_2 \dots v_j)(v_1 \dots v_{j-1}) \\
\stackrel{\eqref{A19}}{=}& (v_{j-1} \dots v_{i+2}) (v_j \dots v_{i+3})  (v_{i-1} \dots v_1)(v_i \dots v_2) (v_{i+1} \dots v_3 \underline{ v_2 } ) (v_{i+2} \dots v_4 ) \cdot \\
& v_3 v_1 v_2\tau_1 v_2 v_1 v_3 v_2 \sigma_1 v_2 v_3 v_1  v_2(v_4  \dots v_{i+2}) (  v_3  \dots v_{i+1}) (v_2  \dots v_j)(v_1 \dots v_{j-1}) \\
\stackrel{\eqref{A19}}{=}& (v_{j-1} \dots v_{i+2}) (v_j \dots v_{i+3})  (v_{i-1} \dots v_1)(v_i \dots v_2) (v_{i+1} \dots v_3 ) (v_{i+2} \dots v_4 ) \cdot\\
&\underline{v_2 v_3 v_1 v_2\tau_1 v_2 v_1 v_3 v_2  \sigma_1} v_2 v_3 v_1 v_2  (v_4  \dots  v_{i+2}) (  v_3  \dots v_{i+1}) (v_2  \dots v_j)(v_1 \dots v_{j-1}). \\
\end{align*}
Recall now the relation \eqref{A28}: $\tau_1 (v_2 v_3 v_1 v_2 \sigma_1 v_2 v_1 v_3 v_2) = (v_2 v_3 v_1 v_2 \sigma_1 v_2 v_1 v_3 v_2) \tau_1$.
Multiplying this relation on the left and on the right by $v_2 v_3 v_1 v_2$ and using that $v_i^2 = 1$ for $i = 1,2,3$ and that $v_1v_3 = v_3v_1$, we obtain:
\[(v_2 v_3 v_1 v_2 \tau_1 v_2 v_1v_3 v_2) \sigma_1 = \sigma_1 (v_2 v_3 v_1 v_2 \tau_1 v_2 v_1v_3 v_2).  \]

Returning to the computations above and using the latter equality to replace the underlined product, we arrive at:
\begin{align*}
\tau_j \sigma_i =& (v_{j-1} \dots v_{i+2}) (v_j \dots v_{i+3})  (v_{i-1} \dots v_1)(v_i \dots v_2) (v_{i+1} \dots v_3 ) (v_{i+2} \dots v_4 )  \cdot
\\& \sigma_1 v_2v_3 v_1 v_2 \tau_1\underline{ v_2 v_1v_3 v_2v_2 v_3v_1 v_2}(v_4  \dots v_{i+2}) (  v_3  \dots v_{i+1}) (v_2  \dots v_j)(v_1 \dots v_{j-1}) \\
\stackrel{\eqref{A21}}{=}& (v_{j-1} \dots v_{i+2}) (v_j \dots v_{i+3})  (v_{i-1} \dots v_1)(v_i \dots v_2) (v_{i+1} \dots v_3 ) (v_{i+2} \dots v_4 ) \cdot
 \\& \sigma_1 v_2 v_3 v_1 v_2 \tau_1 (v_4  \dots v_{i+2}) (  v_3  \dots v_{i+1})  (v_2  \dots v_j)(v_1 \dots v_{j-1}).
\end{align*}

On the other hand, using the relations \eqref{A14}, \eqref{A15}, \eqref{A19}, \eqref{A23}, and \eqref{A16}, one can show the following equality:
 \begin{align*}
  \sigma_i \tau_j = &
 (v_{j-1} \dots v_{i+2}) (v_j \dots v_{i+3})  (v_{i-1} \dots v_1)(v_i \dots v_2) (v_{i+1} \dots v_3 ) (v_{i+2} \dots v_4 ) \cdot \\& \sigma_1 v_2 v_3 v_1 v_2 \tau_1 (v_4  \dots v_{i+2}) (  v_3  \dots v_{i+1})  (v_2  \dots v_j)(v_1 \dots v_{j-1}).
\end{align*}
Comparing the two results, we obtain the desired equality: $\tau_j \sigma_i = \sigma_i \tau_j$.
 \end{proof}
%%%%%%%%%%%%%%%%%%%%%%%%%%%%%%%%%%%%%%%%%%%

%Lemma 4 
\begin{lemma}
 The braid relations $\sigma_i \sigma_j \sigma_i = \sigma_j \sigma_i \sigma_i$ hold for all $|i-j|=1$.
 \end{lemma}

\begin{proof} We will show that the relation holds for $j = i+1$ and $i \geq $2 (recall that the base case relation corresponds to $i = 1$ and $j = 2$, which is represented by the relation~\eqref{A24}). 

Starting with the left hand side of the desired identity and using the relations \eqref{A14}, \eqref{A21}, and \eqref{A16}, we obtain (see the beginning of the proof for Lemma 2 in ~\cite{KL1}):
\begin{align*}
\sigma_i  \sigma_{i+1} \sigma_i = & (v_{i-1}\ldots v_1)(v_{i}\ldots v_2) ( v_{i+1} \ldots v_3) ( \sigma_1   v_1 v_2  \sigma_1 v_2 v_1  \sigma_1)( v_3  \ldots  v_{i+1} ) \cdot  
\\&  (v_2\ldots v_{i})(v_1\ldots v_{i-1}).
\end{align*}

For the right hand side of the identity, we have:
\begin{align*}
\sigma_{i+1} \underline{ \sigma_i} \sigma_{i+1} \stackrel{\eqref{A14}}{=}& \sigma_{i+1} (v_{i-1}\ldots v_1)  (v_i \ldots  v_2 \underline{(1_n)}) \sigma_1 \underline{(1_n)(1_n)}(v_2 \ldots v_{i}) (v_1\ldots v_{i-1})  \sigma_{i+1} \\
=\,\,\, & \underline{ \sigma_{i+1} } (v_{i-1}\ldots v_1)  (v_i \ldots v_2 v_1 v_1) \sigma_1 \underline{(v_{i+1}\ldots v_3)}   (v_3 \ldots  v_{i+1})  (\underline{v_1} v_1 )\cdot 
\\&( v_2\ldots v_{i}) (v_1\ldots v_{i-1})\underline{\sigma_{i+1}}.
\end{align*}
Making use of Lemmas~\ref{sigv-tauv} and ~\ref{sig-tau}, and the relations~\eqref{A19} and~\eqref{A23}, we arrive at:
\begin{align*}
\sigma_{i+1}  \sigma_i \sigma_{i+1}
=\,\,\,&  (v_{i-1}\ldots v_1) \sigma_{i+1}  (v_i \ldots v_{1})(v_{i+1}\ldots v_3)\underline{(1_n)}  v_1  \sigma_1 \,v_1 \underline{(1_n)} (v_3 \ldots  v_{i+1} ) \cdot 
\\& (v_1\ldots v_{i})\sigma_{i+1} (v_1\ldots v_{i-1}) \\
=\,\,\,&  (v_{i-1}\ldots v_1) \underline{\sigma_{i+1}} (v_i \ldots v_{1})(v_{i+1}\ldots v_3) ( v_2   v_2 ) v_1  \sigma_1 v_1 (v_2   v_2 )( v_3 \ldots  v_{i+1} ) \cdot 
\\& (v_1\ldots v_{i}) \underline{\sigma_{i+1}} (v_1\ldots v_{i-1}) \\
\stackrel{\eqref{A14}}{=}& (v_{i-1}\ldots v_1)  (v_{i}\ldots v_1) ( v_{i+1} \ldots v_2 ) \sigma_1 \underline{(v_2 \ldots v_{i+1})(v_1 \ldots v_{i})(v_i \ldots v_{1})(v_{i+1}\ldots v_{2}) } \cdot
\\& v_2 v_1  \sigma_1 v_1 v_2  \underline{( v_2 \dots  v_{i+1} ) (v_1\dots v_{i}) (v_{i}\dots v_1) ( v_{i+1} \dots v_2 )}\sigma_1 ( v_2 \dots  v_{i+1} ) (v_1\dots v_{i}) \cdot
\\&  (v_1\dots v_{i-1}) \\%
\stackrel{\eqref{A21}}{=}&  (v_{i-1}\ldots v_1)(v_{i}\ldots \underline{v_1}) ( v_{i+1} \ldots v_2 ) \sigma_1 v_2 v_1  \sigma_1 v_1 v_2 \sigma_1 ( v_2 \ldots  v_{i+1} )     (\underline{v_1}\ldots v_{i}) \cdot
\\& (v_1\ldots v_{i-1}) \\
\stackrel{\eqref{A19}}{=}&  (v_{i-1}\ldots v_1)(v_{i}\ldots v_2) ( v_{i+1} \ldots v_3)(\underline{v_1 v_2 \sigma_1  v_2 v_1 ) \sigma_1(v_1 v_2 \sigma_1 v_2  v_1 }  )( v_3 \ldots  v_{i+1} )    \cdot 
\\&   (v_2\ldots v_{i}) (v_1\ldots v_{i-1}) \\
\stackrel{\eqref{A24}}{=}&  (v_{i-1}\ldots v_1)(v_{i}\ldots v_2) ( v_{i+1} \ldots v_3) ( \sigma_1   v_1 v_2  \sigma_1 v_2 v_1  \sigma_1  )( v_3  \ldots  v_{i+1} ) (v_2\ldots v_{i}) \cdot 
\\& (v_1\dots v_{i-1}). 
\end{align*}
Therefore, the relation holds for all $i>1$, which completes the proof.
\end{proof}
For a somewhat different proof of the previous lemma (as it applies to the virtual braid group), we refer the reader to~\cite[Lemma 2]{KL1}. 

%Lemma 5
\begin{lemma}
The braid relations $\sigma_j  \sigma_{i} \tau_j = \tau_{i} \sigma_j \sigma_{i}$ hold for all $|i-j|=1$.
\end{lemma}

\begin{proof} We will show that the relation holds for the case $j = i+1$ and $i \geq $2 (the case $i = 1$ and $j = 2$ is the base case relation represented by the relation~\eqref{A24b}). The proof for the other case, namely when $j = i-1$ and $i\geq 3$ follows similarly. For the right hand side of the identity, we have:
\begin{align*}
\tau_i  \sigma_{i+1} \sigma_i  \displaystyle  \mathop{=}_{\eqref{A15}}^{\eqref{A14}}&  [ (v_{i-1}\ldots v_2v_1)(v_{i}\ldots v_3v_2)\tau_1  (v_2v_3\ldots v_{i}) \underline{(v_1v_2\ldots v_{i-1}) }] \cdot\\
& [\underline{ (v_i\ldots v_2v_1) }(v_{i+1}\ldots v_3v_2)\sigma_1 (v_2v_3\ldots v_{i+1}) \underline{ (v_1v_2\ldots v_i) }]\cdot\\
&  [ \underline{(v_{i-1}\ldots v_2v_1) } (v_{i}\ldots v_3v_2)\sigma_1  (v_2v_3\ldots v_{i})(v_1v_2\ldots v_{i-1}) ]\\
\stackrel{\eqref{A16}}{=}&  (v_{i-1}\ldots v_2v_1)(v_{i}\ldots v_3v_2)\tau_1  (\underline{v_2v_3\ldots v_{i}) (v_i \ldots v_2 } v_1 v_2\ldots v_i) (v_{i+1}\ldots v_3v_2)\sigma_1 \cdot \\ 
&(v_2v_3\ldots v_{i+1})  (v_i \ldots  v_2 v_1 \underline{ v_2\ldots v_i) (v_{i}\ldots v_3v_2)} \sigma_1  (v_2v_3\ldots v_{i})(v_1v_2\ldots v_{i-1}) \\
\stackrel{\eqref{A21}}{=}& (v_{i-1}\ldots v_2v_1)(v_{i}\ldots v_3v_2)(\tau_1   v_1 \underline{ v_2\ldots v_i) (v_{i+1}\ldots v_3v_2) } \sigma_1  \cdot \\ 
&  \underline{ (v_2v_3\ldots v_{i+1}) (v_i \ldots  v_2 } v_1  \sigma_1 )(v_2v_3\ldots v_{i})(v_1v_2\ldots v_{i-1}) \\
\stackrel{\eqref{A16}}{=}&  (v_{i-1}\ldots v_2v_1)(v_{i}\ldots v_3v_2)(\underline {\tau_1   v_1}  v_{i+1} \ldots v_3 v_2 v_3\ldots v_{i+1}) \underline{ \sigma_1}  \cdot \\ 
&(v_{i+1} \ldots v_3 v_2 v_3  \ldots  v_{i+1}  \underline{ v_1  \sigma_1 } )  (v_2v_3\ldots v_{i})(v_1v_2\ldots v_{i-1}) \\
\displaystyle  \mathop{=}_{\eqref{A19}}^{\eqref{A23}}  &   (v_{i-1}\ldots v_2v_1)(v_{i}\ldots v_3v_2) (v_{i+1} \ldots v_3) (\tau_1   v_1 v_2  \underline{v_3\ldots v_{i+1})}\cdot \\
  & \underline{(v_{i+1} \ldots v_3 } \sigma_1 v_2 v_1  \sigma_1  v_3 \ldots  v_{i+1}) (v_2v_3\ldots v_{i})(v_1v_2\ldots v_{i-1})  \\
\stackrel{\eqref{A21}}{=}&  (v_{i-1}\ldots v_1)(v_{i}\ldots v_2) ( v_{i+1} \ldots v_3) ( \tau_1   v_1 v_2  \sigma_1 v_2 v_1  \sigma_1) \cdot \\
&( v_3  \ldots  v_{i+1} )     (v_2\ldots v_{i})(v_1\ldots v_{i-1}).
\end{align*}
Now we will consider the left hand side of the identity.
\begin{align*}
\sigma_{i+1} \underline{ \sigma_i} \tau_{i+1} \stackrel{\eqref{A14}}{=}& \sigma_{i+1} (v_{i-1}\ldots v_1)  (v_i \ldots  v_2 \underline{(1_n)}) \sigma_1 \underline{(1_n)(1_n)}(v_2 \ldots v_{i}) (v_1\ldots v_{i-1})  \tau_{i+1} \\
= \,\,\,& \underline{ \sigma_{i+1} } (v_{i-1}\ldots v_1)  (v_i \ldots v_2 v_1 v_1) \sigma_1 \underline{(v_{i+1}\ldots v_3)}   (v_3 \ldots  v_{i+1} )  (\underline{v_1} v_1 )\cdot \\
&( v_2\ldots v_{i}) (v_1\ldots v_{i-1}) \underline{\tau_{i+1}}. 
\end{align*}
Employing the commuting relations in Lemma~\ref{sigv-tauv} and those in Equations~\eqref{A19} and~\eqref{A23}, we obtain:
\begin{align*}
\sigma_{i+1}  \sigma_i \tau_{i+1} = \,\,\,&  (v_{i-1}\ldots v_1) \sigma_{i+1}  (v_i \ldots v_{1})(v_{i+1}\ldots v_3)\underline{(1_n)}  v_1  \sigma_1 \,v_1 \underline{(1_n)} (v_3 \ldots  v_{i+1} ) \cdot \\
& (v_1\ldots v_{i}) \tau_{i+1} (v_1\ldots v_{i-1}) \\
= \,\,\, &  (v_{i-1}\ldots v_1) \underline{\sigma_{i+1}} (v_i \ldots v_{1})(v_{i+1}\ldots v_3) ( v_2   v_2 ) v_1  \sigma_1 v_1 (v_2   v_2 )( v_3 \ldots  v_{i+1} )  \cdot \\
&(v_1\ldots v_{i}) \underline{\tau_{i+1}}(v_1\ldots v_{i-1})\\
\displaystyle  \mathop{=}_{\eqref{A15}}^{\eqref{A14}}  &  (v_{i-1}\ldots v_1)  (v_{i}\ldots v_1) ( v_{i+1} \ldots v_2 ) \sigma_1 \underline{(v_2 \ldots v_{i+1})(v_1 \ldots v_{i})} \cdot \\
&\underline{(v_i \ldots v_{1})(v_{i+1}\ldots v_{2}) } 
v_2 v_1  \sigma_1 v_1 v_2  \underline{( v_2 \ldots  v_{i+1} ) }\cdot \\
& \underline{(v_1\ldots v_{i}) (v_{i}\ldots v_1) ( v_{i+1} \ldots v_2 )}\tau_1 ( v_2 \ldots  v_{i+1} ) (v_1\ldots v_{i}) (v_1\ldots v_{i-1}) \\
 \stackrel{\eqref{A21}}{=}& (v_{i-1}\ldots v_1)(v_{i}\ldots \underline{v_1}) ( v_{i+1} \ldots v_2 ) \sigma_1 v_2 v_1  \sigma_1 v_1 v_2 \tau_1 ( v_2 v_3 
 \ldots  v_{i+1} ) \cdot \\
 & (\underline{v_1}\ldots v_{i})(v_1\ldots v_{i-1}) \\%
 \stackrel{\eqref{A19}}{=}&  (v_{i-1}\ldots v_1)(v_{i}\ldots v_2) ( v_{i+1} \ldots v_3)(\underline{v_1 v_2 \sigma_1  v_2 v_1 ) \sigma_1(v_1 v_2 \tau_1 v_2  v_1 }  )( v_3 \ldots  v_{i+1} )   \cdot \\
&  (v_2\ldots v_{i})(v_1\ldots v_{i-1}) \\
\stackrel{\eqref{A24b}}{=}&  (v_{i-1}\ldots v_1)(v_{i}\ldots v_2) ( v_{i+1} \ldots v_3) ( \tau_1   v_1 v_2  \sigma_1 v_2 v_1  \sigma_1) \cdot \\
&( v_3 \ldots  v_{i+1} )   (v_2\ldots v_{i})(v_1\ldots v_{i-1}).
\end{align*}
Therefore, $\sigma_{i+1} \sigma_i \tau_{i+1}  = \tau_i  \sigma_{i+1} \sigma_i$ for all $i \geq $2. 
\end{proof}

%%% Lemma 6
\begin{lemma}
 The braid relations $\sigma_i  \sigma_i^{-1} = 1_n$ hold for all $1 \leq i \leq n-1$. 
\end{lemma}
 \begin{proof} It is easy to see that these relations hold.
 \end{proof}

%%%% Lemma 7
\begin{lemma}
The braid relations $\tau_i  \sigma_i = \sigma_i \tau_i$ hold for all $1 \leq i \leq n-1$. 
\end{lemma}
 \begin{proof} Let $i >1$. We first use the defining relations~\eqref{A14} and~\eqref{A15} followed by the virtual relations~\eqref{A21} to obtain:
\begin{align*}
\tau_i \sigma_i  =& [ (v_{i-1}\ldots v_2v_1)(v_{i}\ldots v_3v_2)\tau_1  (v_2v_3\ldots v_{i}) \underline{(v_1v_2\ldots v_{i-1}) ]\,  [ (v_{i-1}\ldots v_2v_1) }\cdot \\ 
& (v_{i}\ldots v_3v_2)\sigma_1 (v_2v_3\ldots v_{i})(v_1v_2\ldots v_{i-1}) ]\\
=& (v_{i-1}\ldots v_2v_1)(v_{i}\ldots v_3v_2)\tau_1  \underline{(v_2v_3\ldots v_{i})  (v_{i}\ldots v_3v_2) } \sigma_1   (v_2v_3\ldots v_{i})(v_1v_2\ldots v_{i-1}) \\
=&  (v_{i-1}\ldots v_2v_1)(v_{i}\ldots v_3v_2)\tau_1  \sigma_1  (v_2v_3\ldots v_{i})(v_1v_2\ldots v_{i-1}). 
\end{align*}
Using similar computations, we arrive at:
\begin{align*}
 \sigma_i \tau_i=(v_{i-1}\ldots v_2v_1)(v_{i}\ldots v_3v_2)\sigma_1  \tau_1(v_2v_3\ldots v_{i})(v_1v_2\ldots v_{i-1}). 
\end{align*}
But since $\tau_1 \sigma_1 = \sigma_1 \tau_1$, the statement follows.
\end{proof}

%%%% Lemma 8
\begin{lemma}
The braid relations $v_i\sigma_j v_i=v_j\sigma_iv_j$ and $v_i\tau_jv_i=v_j\tau_iv_j$ hold for all $|i-j|=1$.
\end{lemma}

\begin{proof}
It should be clear that these relations hold, since they were used in the defining relations~\eqref{A14} and~\eqref{A15}. However, we provide a proof for the second set of relations for $j = i+1$ and $i \geq 1$ (the first set of relations follow similarly).
\begin{align*}
v_i\tau_{i+1}v_i \stackrel{\eqref{A15}}{=}& \underline{v_i  (v_i} \ldots v_2v_1)(v_{i+1}\ldots v_3v_2)\tau_1 (v_2v_3\ldots v_{i+1})(v_1v_2\ldots \underline{v_i) v_i}\\
\stackrel{\eqref{A21}}{=}&  (v_{i-1}\ldots v_2v_1)(\underline{v_{i+1}}\ldots v_3v_2)\tau_1 (v_2v_3\ldots \underline{v_{i+1}})(v_1v_2\ldots v_{i-1}) \\
\stackrel{\eqref{A19}}{=}& v_{i+1}  (v_{i-1}\ldots v_2v_1)(v_{i}\ldots v_3v_2)\tau_1 (v_2v_3\ldots v_{i})(v_1v_2\ldots v_{i-1}) v_{i+1}\\
\stackrel{\eqref{A15}}{=}& v_{i+1} \tau_i v_{i+1}.
\end{align*}
This completes the proof.
\end{proof}

 \textit{Concluding remarks.} Virtual singular braids have a monoid structure that can be described by generators and relations. Specifically, in this paper, we introduced the virtual singular braid monoid as the algebraic counterpart of the diagrammatic theory of virtual singular knots and links. The virtual singular braid monoid is an extension of the singular braid monoid by the symmetric group. We have proved an Alexander-type theorem for virtual singular knots and links by providing a braiding algorithm that converts any oriented virtual singular knot or link to a virtual singular braid. We also provided two Markov-type theorems for virtual singular links and braids: (1) using an approach involving $L$-type moves and (2) the classical algebraic approach.
 
 The braiding algorithm described in this paper employs the $L$-moves for oriented virtual links introduced by Kauffman and Lambropoulou in~\cite{KL2}, which in turn are adaptations of prior work of Lambropoulou~\cite{L} on the case of oriented classical links. In particular, in this paper we introduced the singular $L_v$-equivalence for virtual singular braids as an extension of the $L$-equivalence for virtual braids introduced in~\cite{KL2}, to include $L$-type moves involving singular crossings. We first used singular $L_v$-equivalence to prove an $L$-move Markov-type theorem for virtual singular braids, and then turned this result into an algebraic Markov-type theorem for virtual singular braids of any number of strands. Finally, we derived a reduced presentation for the virtual singular braid monoid using fewer generators. The reduced presentation is based on the fact that the virtual singular braid monoid on $n$ strands is generated by three braiding elements plus the generators of the symmetric group on $n$ letters.
\\

\noindent \textbf{Acknowledgements.}  We gratefully acknowledge support from the NSF Grant DMS--1156273 through the \textit{Faculty-Undergraduate Research Student Teams} (FURST) Program. We would also like to thank the referee for reading the paper carefully and providing valuable comments and suggestions.


\begin{thebibliography}{999} 
\bibitem{A} J.W. Alexander, A lemma on systems of knotted curves, \textit{Proc. Nat. Acad. Sci. USA} \textbf{9} (1923), 93-95.  

\bibitem {Be} D. Bennequin, Entrlacements et \'{e}quations de Pfaffe, \textit{Asterisque} \textbf{107-108} (1983), 87-161.

\bibitem {B1} J.S. Birman, Braids, links and mapping class groups, \textbf{Ann. of Math. Stud.} \textbf{82}, Princeton University Press, Princeton, 1974.

\bibitem {B2} J. Birman, New Points of View in Knot Theory, \textit{Bull. Amer. Math. Soc.} (New Series) \textbf{28}, No. 2 (1993), 253-287.

\bibitem {G} B. Gemein, Singular braids and Markov's theorem, \textit{J. Knot Theory and Ramifications} \textbf{6}, No. 4 (1997), 441-454.                                                                                            

\bibitem {Ka} S. Kamada, Braid presentation of virtual knots and welded knots, \textit{Osaka J. Math.} \textbf{44} (2007), 441-458.

\bibitem {K1} L.H. Kauffman, Virtual knot theory, \textit{European. J. Combin.} \textbf{20} (1999), 663-691.

\bibitem {KL1} L.H. Kauffman, S. Lambropoulou, Virtual braids, \textit{Fund. Math.} \textbf{184} (2004), 159-186.

\bibitem {KL2} L.H. Kauffman, S. Lambropoulou, Virtual braids and the L-move,  \textit{J. Knot Theory and Ramifications} \textbf{15}, No. 6 (2006), 773-811.

\bibitem {L} S. Lambropoulou, A study of braids in 3-manifolds, Ph.D. thesis, Warwick Univ. (1993).

\bibitem {LR} S. Lambropoulou, C.P. Rourke, Markov's theorem in 3-manifolds, \textit{Topol. Appl.} \textbf{78} (1997), 95-122.

\bibitem {L2} S. Lambropoulou, L-moves and Markov theorems,  \textit{J. Knot Theory and Ramifications} \textbf{16}, No. 10 (2007), 1-10.

\bibitem {M} A.A. Markov, \"{U}ber die freie \"{A}quivalenz geschlossener Z\"{o}pfe, \textit{Recueil Math\'{e}matique Moscou} \textbf{1}, (1935).

\bibitem {Mo} H.R. Morton, Threading knot diagrams, \textit{Math. Proc. Cambridge Philos. Soc.} \textbf{99} (1986), No. 2, 247-260.

\bibitem {T} P. Traczyk, A new proof of Markov's braid theorem, \textit{Knot theory} (Warsaw, 1995), 409-419, Banach Center Publ., 42, Polish Acad. Sci., Warsaw, 1998.

\bibitem {W} N. Weinberg, Sur l'equivalence libre des tresses ferm\'{e}es, \textit{Comptes Rendus (Doklady) de l'Academie des Sciences de l'URSS} \textbf{23} (1939), 215-216.

\end{thebibliography}
\end{document}